\documentclass{article}

\input{preamble.tex}

\newcommand{\GNE}{\ensuremath{L^{\mathrm{NE}}}}
\newcommand{\weight}[1]{\ensuremath{W({#1})}}
\newcommand{\paths}[1]{\ensuremath{\Pi(#1)}}
\newcommand{\ptime}[1]{\ensuremath{L(#1)}}
\newcommand{\lshape}{\ensuremath{\ell}}

\DeclarePairedDelimiterXPP\Exhat[1]{\widehat{\mathbb{E}}}[]{}{
\renewcommand\given{\nonscript\:\delimsize\vert\nonscript\:\mathopen{}}
#1}

\newcommand{\overbar}[1]{\mkern 1.5mu\overline{\mkern-1.5mu#1\mkern-1.5mu}\mkern 1.5mu}
\newcommand{\lowerbar}[1]{\mkern 1.5mu\underline{\mkern-1.5mu#1\mkern-1.5mu}\mkern 1.5mu}

\numberwithin{equation}{section}
\newcommand{\nocontentsline}[3]{} % Hide the subsections in the intro
\let\origcontentsline\addcontentsline
\newcommand\stoptoc{\let\addcontentsline\nocontentsline}
\newcommand\resumetoc{\let\addcontentsline\origcontentsline}

\usepackage[nottoc,numbib]{tocbibind} % Bibliography in TOC

\title{The Busemann Process and Steep Highways in Directed First Passage Percolation}
\author{Sam McKeown}

\begin{document}
\maketitle
\begin{abstract}
	We consider the Busemann process in planar directed first passage percolation. We extend existing techniques to establish the existence of the process in our setting and determine its distribution in a number of integrable models. As examples of their utility, we show how these explicit distributions may be used to quantify the semi-infinite geodesics passing through thin rectangles, and the clustering phenomenon observed in competition interface angles. There is a natural connection with various particle systems, and in particular we obtain the multi-class invariant distributions for discrete-time TASEP with parallel updates.
\end{abstract}
\tableofcontents
\section{Introduction}
\label{sec:Intro}
First passage percolation (FPP) is a natural model of a random metric on a graph, which assigns distances, or \emph{passage times}, between vertices as the minimal sum of weights along paths connecting those points. We will restrict ourselves to the graph $\bZ^2$. One tool in understanding the long geodesics (minimising paths), and thus the large scale behaviour of the model, is the \emph{Busemann process}. Writing $L(x, y)$ for the (random) passage time between vertices $x,\, y \in \bZ^2$, this is the collection of (random) limits
\begin{equation}
	\label{eq:IntroBuse}
	B^{\xi}(x, y) = \lim_{v \to \xi \cdot \infty} L(x, v) - L(y, v),
\end{equation}
where by $v \to \xi \cdot \infty$ we mean that $\abs{v} \to \infty$ and $v / \abs{v} \to \xi \in \bR^2$. Such limits were introduced by Newman \cite{newman-first-buse} and the program continued in \cite{licea-newman-geodesics, hoffman-geodesics}, wherein they were used to study the large-scale geometry of FPP and derive many properties of (semi-)infinite geodesic rays, such as their existence and limiting directions, under certain unproved assumptions.

The full, unconditional existence of the limits in \eqref{eq:IntroBuse} is available only under a directedness assumption on the geodesics and was proven in the case of \emph{last passage percolation} (LPP) in \cite{geor-rass-sepp-17-buse}. So-called ``generalised Busemann functions'' have been shown to exist in broad generality by \cite{groa-janj-rass-gen-buse}, serving much the same purpose as the Busemann process and in particular giving the existence of geodesic rays in fixed directions. Their construction as weak limits, however, makes questions such as ergodicity much harder, despite being immediate when we start with the representation in \eqref{eq:IntroBuse}. 

Apart from trying to understand the general picture, it is possible also to pin down the exact distribution of the Busemann process in a handful of \emph{solvable} (or \emph{integrable}) models, which has yielded a much finer picture of their geometry. In particular, the distribution of the process has been found for exponential LPP in \cite{fan-sepp-joint-buss} and for the inverse-gamma directed polymer --- its positive temperature relative --- in \cite{bates-fan-sepp}. This understanding has also been extended, not without technical difficulty, to their various scaling limits. See \cite{sepp-soren-blpp-buse,bus-buse-scaling,bus-sepp-soren-DL, grass-KPZ}.

We give the existence of the Busemann limits in the same terms as \cite{geor-rass-sepp-17-buse}, but in the context of a more flexible \emph{edge-weight} FPP. We then determine their joint distributions for certain models with Bernoulli-exponential and Bernoulli-geometric weights, of the type considered in \cite{mart-bergeom, mart-prab-fixed}. Our description relies, as does much prior work, on variations of the \emph{multi-line process} developed by Ferrari and Martin in their influential series of papers \cite{ferr-mart-mtasep, ferr-mart-dual, ferr-mart-had}.

We give an application to the \emph{highways and byways problem} of Hammersley and Welsh \cite{hamm-welsh-fpp}. This is the problem of determining the limiting expected density of semi-infinite geodesics within the tree of geodesics, and has been resolved in its original formulation by \cite{ahlb-hans-hoff-density}, then quantitatively in \cite{dem-elb-pel-midpoint}. We give what appears to be the first almost sure result in this direction, albeit under very specific assumptions on the weights and the subsets through which we take the limit.

Finally, we find the asymptotic density of the \emph{convoys} of our model. We must defer the definitions until after we have properly introduced the model, but readers may be familiar with the phenomenon by the same name which has been observed in TASEP. See \cite{amir-angel-valko-tasep-speed}.

\section{Definitions and main results}
\label{sec:Summary}
\stoptoc
\subsection{The model}
\label{ssec:ModelDefs}
Our setting throughout is the integer lattice $\bZ^2$. Write $e_1$ for the horizontal edge emanating from the origin, or interchangeably as the point $(1, 0)$, and similarly for $e_2 = (0, 1)$. For each pair $x \le y \in \bZ^2$ (where $(x_1, x_2) \le (y_1, y_2)$ if $x_1 \le y_1$ and $x_2 \le y_2$), let $\Pi(x, y)$ be the set of directed nearest neighbour paths connecting $x$ to $y$. A path $\pi \in \paths{x, y}$ is represented by a sequence of vertices $x = \pi^0 \le \pi^1 \cdots \le \pi^n = y$, where $\pi^{i} - \pi^{i - 1} \in \set{e_1, e_2}$, $1 \le i \le n$. To lighten the notational load we also consider $\pi$ to consist of the edges $(\pi^{i - 1}, \pi^i)$.

The random environment consists of weights $\set{\weight{e}}_{e \in \cE(\bZ^2)}$ living on the edges of the lattice. These weights are used to assign to each path $\pi \in \paths{x, y}$ a \emph{passage time} $\ptime{\pi}$, by summing the weights along the edges taken:
\begin{equation}
	\ptime{\pi} = \sum_{e \in \pi}\weight{e}.
\end{equation}
The minimal passage time among $\paths{x, y}$ is the point-to-point passage time $\ptime{x, y}$, and the path attaining this minimal passage time is called a \emph{geodesic}. Observe that when the weight distribution is continuous and i.i.d, the geodesic is almost surely unique. Even when it is not, we can adopt the convention of always taking the right-most of the available geodesics. We write $\gamma_{x, y}$ for \emph{the} geodesic connecting $x$ to $y$ chosen according to this rule, and abbreviate $\gamma_{x} = \gamma_{0, x}$,

Overloading notation to write $0 \in \bZ^2$ for the origin, we will abbreviate $\paths{x} = \paths{0, x}$ and $\ptime{x} = \ptime{0, x}$. Also abbreviate $\weight{(x - e_i, x)} = W(x; i)$. Write $\lambda$ for the joint distribution of $(W(0; 1), W(0; 2))$ and $\lambda_i$ for the marginal of $W(0; i)$. The following blanket assumption will be implicit in the sequel.

\begin{assumption}
	\label{ass:WeightAssump}
	Assume that the pairs $\set{(W(x; 1), W(x; 2))}_{x \in \bZ^{2}}$ are i.i.d and non-constant, and that $\Exabs{W(0; i)}^{2 + \epsilon} < \infty$ for $i \in \set{1, 2}$ and some $\epsilon > 0$. Assume also that $(0, 0) \in \supp(\lambda)$. 
\end{assumption}

\begin{remark}
	Insisting that $(0, 0)$ lie in the support of our weights is no restriction in the directed setting. A directed path connecting $x$ to $y$ has exactly $\inner{y - x, e_1}$ horizontal edges and $\inner{y - x, e_2}$ vertical edges. Shifting the distribution of the weights by a constant affects the passage times by a linear shift and is of no consequence. 
\end{remark}

Those specialisations of the model for which there is only one random weight per vertex often admit a queueing or interacting particle representation. Taking $\weight{x; 1} = \weight{x; 2}$ gives vertex-weight FPP (or the familiar LPP, after swapping signs), with its well-known correspondence to the G/G/1 queue and to TASEP (see for example \cite{drai-mair-ocon}). If we take instead $\weight{x; 2} = 0$ and insist that the remaining horizontal weights are non-negative, then we call the model \emph{strict-weak first passage percolation} (SWFPP). SWFPP with Bernoulli weights has been studied under the name of the Sepp\"al\"ainen-Johansson model. An illustration of this model is provided in \cref{fig:SWFPP}

\begin{figure}
	\centering
	\resizebox{0.4\textwidth}{!}{\includegraphics[scale=1, page=1]{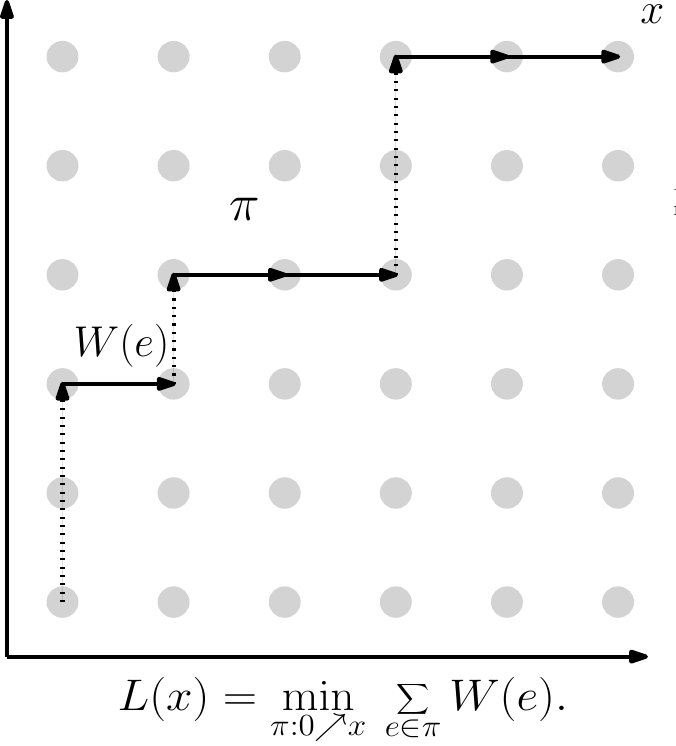}}
	\caption{An illustration of the paths considered in strict-weak first passage percolation and the rule for assigning passage times. Namely, one looks at all up-right directed paths which connect two vertices and sums over the horizontal weights collected, taking the path which minimises this quantity. Observe that there is a pronounced asymmetry between the directions, which does not appear in the more familiar last passage percolation.}
	\label{fig:SWFPP}
\end{figure}

SWFPP can be seen as a degenerate case of what we will call the \emph{Sepp\"al\"ainen-Johansson-Ransford} (SJR) \emph{model}, introduced in \cite{rans-sjr}. Here we associate to the vertices $x$ i.i.d Bernoulli switching variables $S(x)$, and set $\weight{x; 1} = S(x) \widetilde{W}(x; 1)$, $\weight{x; 2} = (1 - S(x)) \widetilde{W}(x; 2)$. Here $(\widetilde{W}(x; 1), \widetilde{W}(x; 2))$ are weights satisfying \cref{ass:WeightAssump}. In words, we choose independently at each vertex one of the horizontal or vertical incoming weights to be random, and the other to be zero. The SJR model may be seen as a random directed metric on a tree of coalescing random walks, in the sense described in \cite{veto-virag}.

We mention an undirected relative of SWFPP, where vertical edges are given a shared constant weight and horizontal weights remain random. This model was introduced in \cite{hass-directed-fpp} and we will refer to results on the limit shape established for this model.

\subsubsection{The store model}
A useful perspective in analysing SWFPP is that of a discrete-time queueing model. So as to avoid confusion with the better known G/G/1 queue, we will follow \cite{drai-mair-ocon} and hereafter call it the \emph{store model}. In this model, we consider a storage bin --- later a tandem of such bins --- which receives material and subsequently output it. In the most fundamental situation, we have a sequence $I = (I_k)_{k \in \bZ}$ of inputs and a sequence $(W_k)_{k \in \bZ}$ of attempted outputs, or \emph{service}. The formulas remain valid if these are taken to be real-valued, but for the purpose of our analogy we imagine them to be non-negative. At time $k$, our bin receives an input $I_k$, which is added to the store, and subsequently attempts to output $W_k$. If there is not enough material in the store, the entire contents of the store are output. The situation is depicted in \cref{fig:StoreDynamics}.

\begin{figure}
	\begin{subfigure}[]{0.3\textwidth}
		\centering
		\resizebox{\textwidth}{!}{\begin{tikzpicture}[x=1cm,y=1cm]
%-------------------- USER KNOBS --------------------
% Fractions of bin height, all in [0,1].
\pgfmathsetmacro{\X}{0.5}   % initial content level
\pgfmathsetmacro{\I}{0.3}   % incoming input level
\pgfmathsetmacro{\W}{0.60}   % service request level
\pgfmathsetmacro{\Gap}{0.15}   % service request level
%---------------------------------------------------

% Derived quantities
\pgfmathsetmacro{\XI}{\X+\I}
\pgfmathsetmacro{\OUT}{min(\W,\XI)}
\pgfmathsetmacro{\REM}{max(\XI-\W,0)} % leftover after service

% Bin geometry
\def\BW{2.2}     % bin width
\def\BH{3.0}     % bin height
\def\GAP{20}    % horizontal gap between stages

% Convenience heights
\pgfmathsetmacro{\hX}{\X*\BH}
\pgfmathsetmacro{\hI}{\I*\BH}
\pgfmathsetmacro{\hXI}{\XI*\BH}
\pgfmathsetmacro{\hW}{\W*\BH}
\pgfmathsetmacro{\hOUT}{\OUT*\BH}
\pgfmathsetmacro{\hREM}{\REM*\BH}
\tikzset{
	>={Classical TikZ Rightarrow[length=3pt,width=6pt]},
	bin/.style={draw,rounded corners=2pt,very thick,minimum width=\BW cm,minimum height=\BH cm},
	levelX/.style={fill=blue!50},
	levelXp/.style={fill=blue!30},
	levelI/.style={fill=green!55!black!25},
	levelIp/.style={fill=green!35!black!15},
	levelIt/.style={fill=green!40!black!5},
	levelY/.style={fill=orange!80!black!55},
	levelYp/.style={fill=orange!60!black!35},
	lab/.style={font=\sffamily\small},
	title/.style={font=\sffamily\bfseries,align=center},
	demandbrace/.style={decorate,decoration={brace,amplitude=5pt},thick},
	flowarrow/.style={->,very thick},
	panel/.style={draw=black!10,rounded corners=6pt,fill=black!2,inner sep=8pt}
}	
	
	\node[inner sep=10pt] (p1) at (\BW/2, -0.2) {%
		\begin{tikzpicture}[x=1cm,y=1cm]
			% initial level X
			\path[levelX] (0,0) rectangle (\BW,\hX);
			\node[lab] at (\BW/2, \hX/2) {$X_k$};
			
			% input I
			\path[levelI] (-\BW - \Gap,\hX) rectangle (- \Gap, \hX + \hI);
			\node[lab] at (-\BW/2 - \Gap,\hX + \hI/2) {$I_k$};
			
			\draw[flowarrow] (-\BW + \Gap,\hX + \hI + 2 * \Gap) -- ++(\BW - 4 * \Gap,0);
			
			% bin outline without the top edge
			\draw[very thick,rounded corners=2pt]
			(0,0) -- (0,\BH)   % left wall
			(0,0) -- (\BW,0)   % bottom
			(\BW,0) -- (\BW,\BH); % right wall
		\end{tikzpicture}
	};
\end{tikzpicture}}
		\subcaption{}
	\end{subfigure}
	\begin{subfigure}[]{0.32\textwidth}
		\centering
		\resizebox{\textwidth}{!}{\begin{tikzpicture}[x=1cm,y=1cm]
	%-------------------- USER KNOBS --------------------
	% Fractions of bin height, all in [0,1].
	\pgfmathsetmacro{\X}{0.5}   % initial content level
	\pgfmathsetmacro{\I}{0.3}   % incoming input level
	\pgfmathsetmacro{\W}{0.60}   % service request level
	\pgfmathsetmacro{\Gap}{0.15}   % service request level
	%---------------------------------------------------
	
	% Derived quantities
	\pgfmathsetmacro{\XI}{\X+\I}
	\pgfmathsetmacro{\OUT}{min(\W,\XI)}
	\pgfmathsetmacro{\REM}{max(\XI-\W,0)} % leftover after service
	
	% Bin geometry
	\def\BW{2.2}     % bin width
	\def\BH{3.0}     % bin height
	\def\GAP{20}    % horizontal gap between stages
	
	% Convenience heights
	\pgfmathsetmacro{\hX}{\X*\BH}
	\pgfmathsetmacro{\hI}{\I*\BH}
	\pgfmathsetmacro{\hXI}{\XI*\BH}
	\pgfmathsetmacro{\hW}{\W*\BH}
	\pgfmathsetmacro{\hOUT}{\OUT*\BH}
	\pgfmathsetmacro{\hREM}{\REM*\BH}
	\tikzset{
		>={Classical TikZ Rightarrow[length=3pt,width=6pt]},
		bin/.style={draw,rounded corners=2pt,very thick,minimum width=\BW cm,minimum height=\BH cm},
		levelX/.style={fill=blue!50},
		levelXp/.style={fill=blue!30},
		levelI/.style={fill=green!55!black!25},
		levelIp/.style={fill=green!35!black!15},
		levelIt/.style={fill=green!40!black!5},
		levelY/.style={fill=orange!80!black!55},
		levelYp/.style={fill=orange!60!black!35},
		lab/.style={font=\sffamily\small},
		title/.style={font=\sffamily\bfseries,align=center},
		demandbrace/.style={decorate,decoration={brace,amplitude=5pt},thick},
		flowarrow/.style={->,very thick},
		panel/.style={draw=black!10,rounded corners=6pt,fill=black!2,inner sep=8pt}
	}	
	
		\node[inner sep=10pt] (p1) at (\BW/2, -0.2) {%
			\begin{tikzpicture}[x=1cm,y=1cm]
				% padding
				\draw[black!0] (- \BW * 0.5, \BH * 0.5) -- ++ (\BW * 2, 0);
				
				% current contents
				\path[levelX] (0,0) rectangle (\BW,\hX);
				\node[lab] at (\BW/2, \hX/2) {$X_k$};
				\path[levelI] (0,\hX) rectangle (\BW, \hXI);
				\node[lab] at (\BW/2, \hX + \hI/2) {$I_k$};

				% service W
				\draw[very thick]
				(\BW+0.3,\hI + \hX - \hW) -- ++(0.2,0)              % bottom tick
				(\BW+0.5,\hI + \hX) -- ++(0,-\hW)          % vertical
				(\BW+0.3,\hI + \hX) -- ++(0.2,0);           % top tick
				\node[right] at (\BW+0.5,\hI + \hX - \hW / 2) {$W_k$};
				
				% bin outline without the top edge
				\draw[very thick,rounded corners=2pt]
				(0,0) -- (0,\BH)   % left wall
				(0,0) -- (\BW,0)   % bottom
				(\BW,0) -- (\BW,\BH); % right wall

			\end{tikzpicture}
		};
\end{tikzpicture}}
		\subcaption{}
	\end{subfigure}
	\begin{subfigure}[]{0.3\textwidth}
		\centering
		\resizebox{\textwidth}{!}{\begin{tikzpicture}[x=1cm,y=1cm]
	%-------------------- USER KNOBS --------------------
	% Fractions of bin height, all in [0,1].
	\pgfmathsetmacro{\X}{0.5}   % initial content level
	\pgfmathsetmacro{\I}{0.3}   % incoming input level
	\pgfmathsetmacro{\W}{0.60}   % service request level
	\pgfmathsetmacro{\Gap}{0.15}   % service request level
	%---------------------------------------------------
	
	% Derived quantities
	\pgfmathsetmacro{\XI}{\X+\I}
	\pgfmathsetmacro{\OUT}{min(\W,\XI)}
	\pgfmathsetmacro{\REM}{max(\XI-\W,0)} % leftover after service
	
	% Bin geometry
	\def\BW{2.2}     % bin width
	\def\BH{3.0}     % bin height
	\def\GAP{20}    % horizontal gap between stages
	
	% Convenience heights
	\pgfmathsetmacro{\hX}{\X*\BH}
	\pgfmathsetmacro{\hI}{\I*\BH}
	\pgfmathsetmacro{\hXI}{\XI*\BH}
	\pgfmathsetmacro{\hW}{\W*\BH}
	\pgfmathsetmacro{\hOUT}{\OUT*\BH}
	\pgfmathsetmacro{\hREM}{\REM*\BH}
	\tikzset{
		>={Classical TikZ Rightarrow[length=3pt,width=6pt]},
		bin/.style={draw,rounded corners=2pt,very thick,minimum width=\BW cm,minimum height=\BH cm},
		levelX/.style={fill=blue!50},
		levelXp/.style={fill=blue!30},
		levelI/.style={fill=green!55!black!25},
		levelIp/.style={fill=green!35!black!15},
		levelIt/.style={fill=green!40!black!5},
		levelY/.style={fill=orange!80!black!55},
		levelYp/.style={fill=orange!60!black!35},
		lab/.style={font=\sffamily\small},
		title/.style={font=\sffamily\bfseries,align=center},
		demandbrace/.style={decorate,decoration={brace,amplitude=5pt},thick},
		flowarrow/.style={->,very thick},
		panel/.style={draw=black!10,rounded corners=6pt,fill=black!2,inner sep=8pt}
	}	
	
		\node[inner sep=8pt] (p1) at (\BW/2, -0.2) {%
			\begin{tikzpicture}[x=1cm,y=1cm]
				% after service
				\path[levelXp] (0,0) rectangle (\BW,\hREM);
				%				\node[lab] at (\BW/2, \hREM/2) {$X' = $\hfill$(X + I - W)^+$};
				\node[lab] at (\BW/2, \hREM/2) {$X_{k + 1}$};
				
				% output
				\path[levelIp] (\BW + \Gap,\hREM) rectangle (2*\BW + \Gap, \hREM + \hOUT);
				%				\node[lab] at (\BW + \BW/2 + \Gap,\hREM + \hOUT/2) {$I' =$ \hfill $(X + I) \wedge W$};
				\node[lab] at (\BW + \BW/2 + \Gap,\hREM + \hOUT/2) {$I'_{k}$};
				
				\draw[flowarrow] (\BW + \BW * 0.1 + 1 * \Gap,\hREM + \hOUT + 2 * \Gap) -- ++(\BW - 4 * \Gap,0);
				
				% bin outline without the top edge
				\draw[very thick,rounded corners=2pt]
				(0,0) -- (0,\BH)   % left wall
				(0,0) -- (\BW,0)   % bottom
				(\BW,0) -- (\BW,\BH); % right wall
			\end{tikzpicture}
		};
\end{tikzpicture}}
		\subcaption{}
	\end{subfigure}
	\caption{The movement of material within our store at time step $k$. The store has some initial quantity $X_k$, which is added to by input $I_k$. Then an attempt is made to output the desired amount $W_k$. In this instance, there was enough material to cover the demand. Here the new store size is $X_{k + 1} = (X_k + I_k - W_k)^+$ and the output is $I'_k = (X_k + I_k) \wedge W_k$.}
	\label{fig:StoreDynamics}
\end{figure}
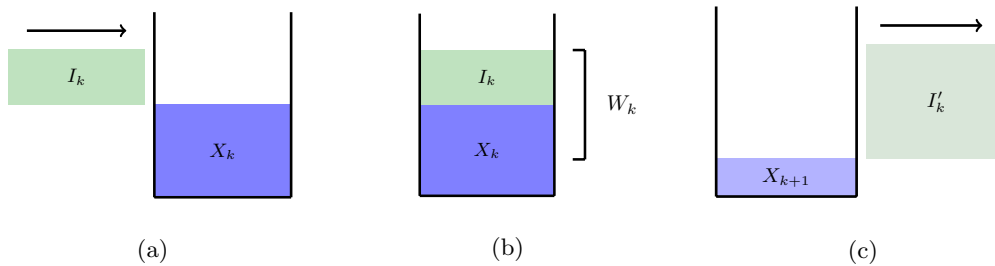

To express this in formulas, let $I' = (I'_k)$ be the sequence of outputs from the store and write $I' = H(I, W)$ to for this operation. Let $X_k = \max_{l \le k} \sum_{j = l}^{k - 1}(I_j - W_j)$ represent the initial amount in the bin at time $k$. Then we have
\begin{equation}
	\label{eq:HDef}
	I'_k = (I_k + X_k) \wedge W_k.
\end{equation}
We assume that our sequences are such that the maximum in $X_k$ is always attained.

There are any number of interpretations we may attach to these dynamics. For example, the bin may represent the stock available to a wholesaler on day $k$, for which we write $X_k$. The wholesaler receives stock $I_k$ in the morning and attempts to fill orders to the amount of $W_k$ in the afternoon, to the extent that the stock will allow. The amount of stock dispatched will be $I'_k$. The authors of \cite{mart-prab-fixed} prefer to consider a discrete-time G/G/1 queue. We adopt the more colourless image of a bin receiving and outputting material in the hope that it permits clarity of language.

A more thorough picture of this model can be found in either \cite{drai-mair-ocon} or \cite{mart-bergeom}. The precise definitions, further maps based on this model, and particle system interpretations will be given in \cref{ssec:UpdateMaps}.

\subsubsection{The limit shape}
We call a function $\ell: \bR_{\ge 0}^2 \to \bR$ the \emph{time constant}, or by conflating it with its level sets, the \emph{limit shape}, if we have 
\begin{equation}
	\label{eq:LimitShapeDef}
	\ell(s, t) = \lim_{n \to \infty}\frac{L(\floor{n s}, \floor{n t})}{n} \as.
\end{equation}
Such limit shapes exist under much weaker assumptions than we take here. It will be useful for this convergence to be uniform across all directions. A result of this kind was proved for vertex-weight last passage percolation by Martin in \cite{mart-limit-shape}. It may be of independent interest to note that such convergence holds in our setup too.
\begin{theorem}
	\label{thm:LimitShapeUniform}
	The limit shape $\lshape$ is continuous on $\bR^2_{\ge 0}$, and almost surely,
	\begin{equation}
		\lim_{n \to \infty}\frac{1}{n}\max_{x \in \bZ_{\ge}^{2}, \abs{x} = n}\abs{\ptime{x} - \lshape(x)} = 0.
	\end{equation}
\end{theorem}
We once again abbreviate $\lshape(0, x) = \lshape(x)$.

In fact, Martin's results go further in the planar case to give universal asymptotics for the limit shape near the axis \cite[Theorem 2.4]{mart-limit-shape}. An extension of his argument yields the same asymptotics for edge weights (or indeed on the directed triangular lattice), albeit with a subtlety when weights in one direction are taken to be constant. Such a discussion would lead us too far astray, however.
%The now classical Cox-Durrett theorem remains the best we can say in total generality. We state it here under \cref{ass:WeightAssump} and included only the parts relevant to us.
%\begin{theorem}[Cox-Durrett \cite{cox-durr}]
%	\label{thm:LimitShapeExist}
%	There exists a deterministic, convex function $\lshape: \bR^2_{\ge 0} \to \bR$ such that for each $(x_1, x_2) \in  \bR^2_{\ge 0}$, 
%	\begin{equation}
%		\label{eq:BasicLimitShape}
%		\lim_{n \to \infty}n^{-1}\ptime{\floor{n x_1}, \floor{n x_2}} = \lshape(x_1, x_2),
%	\end{equation}
%	almost surely and in $L^1$. Moreover, almost surely the limit of \eqref{eq:BasicLimitShape} holds simultaneously for all $(x_1, x_2) \in \bR^2_{\ge 0}$.
%\end{theorem}

\subsection{Existence of the Busemann process}
\label{ssec:BusemannResults}
Our first set of results extend those of \cite{geor-rass-sepp-17-buse} to our setting. The theorem below summarises our knowledge of the Busemann process when $\lshape$ is assumed to be differentiable and strictly convex. This remains unproven in any level of generality, but can be directly verified for the models we treat in the remainder of the paper, where explicit representations exist. The most general statement, and the attendant notation, are given in \cref{app:Busemann}.

In the following, $\scrU = \set{(t, 1 - t) : t \in (0, 1)}$ is the set of directions into the first quadrant. We order these angles by their first component, so that larger angles are those closer to horizontal. Let $\Theta_z : \bR^{\bZ^2} \to \bR^{\bZ^2}$ be the shift of the environment by $z \in \bZ^2$. That is, $\Theta_z(W){x} = \weight{x - z}$. 

\begin{theorem}
	\label{thm:BusemannExistence}
	Suppose the limit shape $\ell$ is differentiable and strictly convex. With probability $1$, there is a (random) subset $\scrU_0 \subseteq \scrU$ with countable complement, such that for a fixed $\xi$ we have $\Prob{\xi \in \scrU_0} = 1$, and for all $\xi \in \scrU_0$ the limits
	\begin{equation}
		\label{eq:BusemannExistenceLimits1}
		B^{\xi}(x, y) = \lim_{v_n \to \xi \cdot \infty} \ptime{x, v_n} - \ptime{y, v_n}
	\end{equation}
	exist for all $x,\, y \in \bZ^2$, where $\brac{v_n} \subseteq \bZ^2$ is a sequence with $\abs{v_n} \to \infty$ and $v_n / \abs{v_n} \to \xi$. Moreover, for all $\xi \in \scrU$, the limits
	\begin{equation}
		\label{eq:BusemannExistenceLimits2}
		B^{\xi^+}(x, y) = \lim_{\ov{\zeta \to \xi^+}{\zeta \in \scrU_0}}B^{\zeta}(x, y), \qquad B^{\xi^-}(x, y) = \lim_{\ov{\zeta \to \xi^-}{\zeta \in \scrU_0}}B^{\zeta}(x, y)
	\end{equation}
	exist. 
	
	Let $x,\, y,\, z \in \bZ^2$ be arbitrary vertices. The random functions $B^{\xi \pm}(x, y)$ satisfy the following properties:
	\begin{thmlist}
		\item \label{thm:BusemannExistenceCocycle} (Cocycle) We have the identity 
		\begin{equation}
			\label{eq:BusemannExistenceCocycle}
			B^{\xi \pm}(x, y) = B^{\xi \pm}(x, z) + B^{\xi \pm}(z, y).
		\end{equation}
		
		\item \label{thm:BusemannExistenceErgodicity} (Ergodicity) The $\bR^{4}$-valued process $\set{\psi_{x}^{\xi \pm}}_{x \in \bZ^{2}, \xi \in \scrU}$ defined by
		\begin{equation}
			\label{eq:BusemannExistenceStat}
			\psi_{x}^{\xi \pm} = (\weight{x; 1}, \weight{x; 2}, B^{\xi \pm}( x, x + e_{1}), B^{\xi \pm}(x, x + e_{2}))
		\end{equation}
		is stationary under translations $\Theta_{z}$, and moreover is \emph{jointly ergodic} under the group of such translations. 
		
		\item \label{thm:BusemannExistenceMeasurability} (Measurability) The values $B^{\xi +}(x, y),\, B^{\xi -}(x, y)$ are measurable with respect to $\set{W(z) : z \ge x \text{ or } z \ge y}$, and hence independent of all other weights.
		
		\item \label{thm:BusemannExistenceAdaptedness} (Adaptedness) We have
		\begin{equation}
			\label{eq:BusemannExistenceAdapt}
			\min_{i \in \set{1, 2}} \brac[\Big]{W(x, x + e_{i}) - B^{\xi \pm}(x, x + e_{i})} = 0.
		\end{equation}
		
		\item \label{thm:BusemannExistenceMonotonicity} (Monotonicity) If $\xi \cdot e_{1} < \zeta \cdot e_{1}$, then
		\begin{equation}
			\label{eq:BusemannExistenceMono1}
			B^{\xi -}(x, x + e_{1}) \le B^{\xi +}(x, x + e_{1}) \le B^{\zeta -}(x, x + e_{1})
		\end{equation}
		and
		\begin{equation}
			\label{eq:BusemannExistenceMono2}
			B^{\xi -}(x, x + e_{2}) \ge B^{\xi +}(x, x + e_{2}) \ge B^{\zeta - }(x, x + e_{2}).
		\end{equation}
		
		\item \label{thm:BusemannExistenceCont} (One-sided continuity) If $\zeta_{n} \cdot e_{1} \searrow \xi \cdot e_{1}$, then
		\begin{equation}
			\label{eq:BusemannExistenceCont1}
			\lim_{n \to \infty}B^{\zeta_{n}\pm}(x, y) = B^{\xi +}(x, y).
		\end{equation}
		Similarly, if $\zeta_{n} \cdot e_{1} \nearrow \xi \cdot e_{1}$, then
		\begin{equation}
			\label{eq:BusemannExistenceCont2}
			\lim_{n \to \infty}B^{\zeta_{n}\pm}(x, y) = B^{\xi - }(x, y).
		\end{equation}
		When $\xi \in \scrU_0$, these limits coincide and the map $\xi \mapsto B^{\xi}(x, y)$ is continuous here.
		
		\item \label{thm:BusemannExistenceDistinct} (Distinct means) The mean vectors $h(\xi \pm) = h(B^{\xi \pm})$ defined by
		\begin{equation}
			h(\xi \pm) \cdot e_{i} = \Ex{B^{\xi \pm}(0, e_{i})}
		\end{equation}
		satisfy
		\begin{equation}
			h(\xi \pm) = \nabla \lshape(\xi).
		\end{equation}
		In particular, Busemann limits in two distinct directions can not have the same expectation.
\end{thmlist}
\end{theorem}

From the cocycle property in \eqref{eq:BusemannExistenceCocycle}, we see that it is enough to consider nearest neighbour Busemann functions. We abbreviate $B^\xi_i(x) = B^\xi(x, x + e_i)$.

\begin{remark}
	\label{rem:GeodesicConstruction}
	For a fixed $\xi \in \scrU$, the above theorem tells us that the Busemann limit exists almost surely and we may construct \emph{Busemann geodesics} by the following procedure. Starting from any $x_0 \in \bZ^2$, we take $i_1 \in \set{1, 2}$ such that $B^{\xi}(x_0, x_0 + e_{i_1}) = W(x_0, x_0 + e_{i_1})$ (although such a choice may not be unique). Then we repeat this process at $x_1 = x_0 + e_{i_1}$, and so on. The resulting infinite path is a semi-infinite geodesic in direction $\xi$. This is the standard method of producing geodesics from the Busemann process, but for completeness we include a proof in \cref{app:Busemann}. See also \cite[Section 5.3]{auff-dam-hans-50-years} and \cite[Section 8]{rass-cgm-18}. Even when the limit does not exist, we may carry this process out using $B^{\xi+}$ or $B^{\xi-}$, producing two distinct geodesics.
	
	There is a pleasing simplification in SWFPP: to produce a geodesic from $B^{\xi}$, we need only look at $B^{\xi}(x, x + e_2)$. If it's zero, we can take an up step, and otherwise go right. That the decision may be done without reference to the horizontal increments or the environment of weights allows us to perform direct computations which would be quite unwieldy in exponential LPP. Some results of this type are in \cref{ssec:HighwaysResults,ssec:Competition}.
\end{remark}

\begin{remark}
	\label{rem:NESWBuse}
The functions defined in \cref{thm:BusemannExistence} are \emph{north-east} Busemann functions, in the sense that the sequence of points $v_n$ in the definition eventually lie in the top-right quadrant. We may equivalently consider \emph{south-west} Busemann functions, of the form
\begin{equation}
	\label{eq:SWBusemannLimits}
	B^{\xi}(x, y) = \lim_{v_n \to - \xi \cdot \infty} \ptime{v_n, x} - \ptime{v_n, y}.
\end{equation}
	In fact, it is these functions whose joint distribution we identify in \cref{thm:BuseDists}. The distribution of the north-east Busemann process is then the result of applying a reflection about the origin.
\end{remark}

\subsection{Exact distributions for Bernoulli-exponential weights}
\label{ssec:ExactDists}
In \cite{fan-sepp-joint-buss}, the joint distributions of the Busemann functions in exponential LPP are characterised as the unique distribution invariant under a vector-valued M/M/1 queue. This is then identified as the result of a certain recursive process, which we will choose to call a \emph{multi-line} distribution. Such constructions were introduced by Ferrari and Martin in the context of multi-class TASEP \cite{ferr-mart-mtasep}, and as we mention in the introduction, have been successfully adapted to treat many other models. We follow the same route, although now with a slightly expanded set of weight distributions, namely \emph{Bernoulli-exponential} (and Bernoulli-geometric) variables. 

A Bernoulli-exponential variable is simply the product of independent Bernoulli and exponential variables. By $W \sim \Ber(p)\Exp(\tau)$, we mean that
\begin{equation}
	\label{eq:BerExpDef}
	\Prob{W \ge x} = (1 - p) \bone(x = 0) + p e^{- \tau x}, \quad x \ge 0.
\end{equation}
These variables arise naturally when working with exponentials and max-plus operations. If, for example, $X \sim \Exp(\tau)$ and $Y \ge 0$ independent, then $(X - Y)^+ \sim \Ber(p)\Exp(\tau)$ for some $p$. Taking $p = 1$ gives us an exponential variable. 

We denote by $\Geom_+(\alpha)$ a geometric variable with positive support, and by $\Geom_0(\alpha)$ one with non-negative support. That is, if $X_0 \sim \Geom_0(\alpha)$, then
\begin{equation}
	\Prob{X_0 = k} = (1 - \alpha)^k\alpha,\quad k \ge 0,
\end{equation}
and $X_0 + 1 \sim \Geom_+(\alpha)$. The Bernoulli-geometric distribution is defined in the obvious way, as $\Ber(p)\Geom_+(\alpha)$. Observe that $p = 1$ gives a $\Geom_+(\alpha)$ variable and $p = 1 - \alpha$ gives a $\Geom_0(\alpha)$ variable.

For concreteness, let us consider only $\Ber(p)\Exp(1)$ weights, $0 < p \le 1$, and let $\lambda$ be the distribution on $\bR^{\bZ}$ whose components are i.i.d $\Ber(p)\Exp(1)$ distributed. We may equally consider $\Ber(p)\Geom_+(\alpha)$ or $\Ber(p)$ distributions, with little modification.

For $\rho_1 < p$, let $\nu_{H}^{\rho_1}$ have instead i.i.d $\Ber(q)\Exp(\tau)$ components, where $q,\, \tau$ satisfy
\begin{equation}
	\label{eq:HParamConst}
	\rho_1 = \frac{q}{\tau}, \qquad \frac{q \tau }{1 - q} = \frac{p}{1 - p}.
\end{equation}
When $p = 1$, we simply fix $q = 1$. These equations can be solved to give rather complicated expressions:
\begin{equation}
	\label{eq:ParamExpres}
	q = \frac{2 \sqrt{p \rho_1}}{\sqrt{p \rho_1} + \sqrt{4 (1 - p) + p \rho_1}}, \quad \tau = \frac{q}{\rho_1}.
\end{equation}
Observe that $\rho_1$ is the mean of $\nu_{H}^{\rho_1}$. Varying $\rho_1 \in (0, p)$ gives us parameters $q$ lying in $(0, p)$ and $\tau$ in $(1, \infty)$.

Now let $\rho = (\rho_1, \dots, \rho_n)$ be a vector of means with $p > \rho_1 > \cdots > \rho_n$, and write $\nu_{H}^{\rho} = \nu_{H}^{\rho_1} \otimes \cdots \otimes \nu_{H}^{\rho_n}$. Take $(I^1, \dots, I^n) \sim \nu_{H}^{\rho}$ and consider the following procedure. Set $J^1 = I^1$, and for $k \ge 2$ let $J^k = H(I^k, J^{k - 1})$. Finally, write $\mu_{H}^{\rho}$ for the distribution of $(J^1, \dots, J^n)$. We call the coupling resulting from this iterative procedure a \emph{multi-line} distribution. Martin and Prabhakar \cite{mart-prab-fixed} have shown the following for store output map $H$.

\begin{theorem}[Theorem 7.1 of \cite{mart-prab-fixed}]
	\label{thm:HMultiStat}
	Take $W \sim \Ber(p)\Exp(1)$ and $(J^1, \dots, J^n) \sim \mu_{H}^{\rho}$, independent. Then:
	\begin{thmlist}
		\item The marginal distributions are $J^k \sim \nu_{H}^{\rho_k}$.
		\item We have component-wise monotonicity: $J^1 \ge J^2 \ge \cdots \ge J^n$.
		\item The distribution $\mu_{H}^{\rho}$ is jointly invariant under $H$:
		\begin{equation}
			\label{eq:HMultiStat}
			(H(J^1, W), \dots, H(J^n, W)) \disteq (J^1, \dots, J^n).
		\end{equation}
	\end{thmlist}
\end{theorem}

When the parameters are suitably adjusted, the same construction produces the multi-class (or vector-valued) invariant measures for two further maps arising from our store model. For the moment let us just say that the $H$ map captures the relationship between the Busemann functions on a horizontal line and on the line directly above it. Similarly, we may introduce maps $A$ and $V$, capturing this evolution on antidiagonal and vertical lines. These are introduced properly in \cref{ssec:UpdateMaps}. 

Given a mean $\rho_1 > 0$, let $\nu_{A}^{\rho_1}$ and $\nu_{V}^{\rho_1}$ be distributions on $\bZ$-indexed sequences whose entries have i.i.d Bernoulli-Exponential distribution with mean $\rho_1$, the parameters explicit functions of $\rho_1$. The formulas for these parameters are involved and uninformative. Later we will consult \cref{tab:SWFPPInvariantDists} and choose the distribution corresponding to our weights and with the desired mean.

%Let us define these now. Let $\rho_1 > 0$ be arbitrary and let $\nu_{A}^{\rho_1}$ have i.i.d $\Ber(r)\Exp(\gamma)$ entries, where $r,\, \gamma$ satisfy
%\begin{equation}
%	\label{eq:AParamConst}
%	\rho_1 = \frac{r}{\gamma}, \quad r = \frac{1}{1 + (1 - p)\gamma}.
%\end{equation}
%These can once again be solved but lead to messy and uninformative expressions. Also define $\nu_{V}^{\rho_1}$, where the marginals are $\Ber(r')\Exp(\gamma')$, and the parameters satisfy
%\begin{equation}
%	\label{eq:VParamConst}
%	\rho_1 = \frac{r'}{\gamma'}, \quad r' = \frac{1}{1 + \gamma'}.
%\end{equation}
%These constraints were derived in \cite{mart-bergeom}.

As before, define $\nu_{A}^{\rho} = \nu_{A}^{\rho_1} \otimes \cdots \otimes \nu_{A}^{\rho_n}$, and $\nu_{V}^{\rho}$ analogously. Let $\mu_{A}^{\rho}$ (resp. $\mu_{V}^{\rho}$) be the distributions after iteratively applying $H$ to $\nu_{A}^{\rho}$ (resp. $\nu_{V}^{\rho}$), as in the definition of $\mu_{H}^{\rho}$. We will prove in \cref{sec:Multiline} that these distributions satisfy the conclusions of \cref{thm:HMultiStat}, now with maps $A$ and $V$. It is curious to note that the maps $A$ and $V$ do not appear in the constructions of their respective multi-line distributions.

This detour into the fixed points of these maps is ultimately with the purpose of identifying the distribution of the Busemann process along horizontal, antidiagonal and vertical lines.

\begin{theorem}
	\label{thm:BuseDists}
	In SWFPP, let our weights follow a Bernoulli-exponential, Bernoulli-geometric, or Bernoulli distribution. Set $e_H = e_1$, $e_A = e_1 - e_2$ and $e_V = -e_2$ and take $U \in \set{H, A, V}$. Let $\xi_1 > \cdots > \xi_n$ be a sequence of directions with $\Ex{B^{\xi_i}(0, e_{U})} = \rho_i$, where here our Busemann functions are directed \emph{south-west}. Then 
	\begin{equation}
		\label{eq:BuseDists}
		\brac[\big]{B^{\xi_1}(k e_U, (k + 1)e_U), \dots, B^{\xi_n}(k e_U, (k + 1)e_U)}_{k \in \bZ} \sim \mu_{U}^{\rho}.
	\end{equation}
\end{theorem}

A surprising but straightforward fact is that when $p = 1$ and our weights are simply rate $1$ exponential, the distributions $\mu_{H}^{\rho}$ and $\mu_{A}^{\rho}$ coincide when the choice of $\rho$ is valid for each. Less straightforwardly, these are also related to the distributions $\mu^{\rho}$ described by Fan and Sepp\"al\"ainen in \cite{fan-sepp-joint-buss}, there in the context of exponential LPP.

\begin{proposition}
	\label{prop:EqualDists}
	Denote by $B$ the \emph{south-west} Busemann process for exponential SWFPP, and by $\overline{B}$ the \emph{north-east} Busemann process for exponential LPP. For a parameter $0 < \rho < 1$, let $\xi(\rho)$ be the angle such that $\Ex{B^{\xi(\rho)}(e_2, e_1)} = \rho^{-1}$ and $\overline{\xi}(\rho)$ the angle such that $\Ex{\overline{B}^{\overline{\xi}(\rho)}(e_2, 0)} = \rho^{-1}$. Then
	\begin{multline}
		\label{eq:EqualDists}
			\set{B^{\xi(\rho)\pm}(k(e_1 - e_2), (k + 1)(e_1 - e_2)) : k \in \bZ, 0 < \rho < 1} \\
			\disteq \set{\overline{B}^{\overline{\xi}(\rho)\pm}((k + 1)e_2, k e_2) : k \in \bZ, 0 < \rho < 1}
	\end{multline}
\end{proposition}

Much is known about the Busemann functions in exponential LPP, and we can, for example, observe that the results of \cite{sepp-soren-blpp-buse, bus-buse-scaling} imply convergence of our SWFPP Busemann process to the stationary horizon when $p = 1$. One expects that the stationary horizon appears as a limit of the Busemann process for all choices of $p > 0$, though we avoid the necessary verifications here. Indeed, this convergence has already been shown for Bernoulli weights in \cite{bus-sepp-soren-tasep-speed}.

We remark that whereas the results of \cite{fan-sepp-joint-buss} capture the distribution of the Busemann process along horizontal and (equivalently) vertical lines in exponential LPP, we are additionally able to describe the distribution along the antidiagonal for our model.

\subsection{Parallel TASEP}
\label{ssec:PTASEP}
As a byproduct of identifying the multi-class invariant distributions for the maps $A$ and $V$, we can also answer a question found in \cite{ferr-mart-mtasep, mart-schi-discrete} concerning the \emph{parallel TASEP}. We will be more explicit with our definition in \cref{ssec:UpdateMaps}, but informally, this is a discrete-time TASEP wherein all of the particles moving at a given time step move simultaneously. If a particle is to move forward on a given time step, then the space in front of the particle must be vacant before that time step. The multi-line distribution in the following statement should be constructed with Bernoulli-geometric variables, rather than Bernoulli-exponential.

\begin{corollary}
	\label{cor:ParallelTASEP}
	The multi-class distribution on $\set{0, 1}^{\bZ}$ whose inter-particle distances are given by $\mu_{A}^\rho$ is invariant under the dynamics of parallel TASEP with Bernoulli-geometric jumps.
\end{corollary}

\subsection{The SJR model}
\label{ssec:SJRDists}
We may also consider the Busemann process of the SJR model. As it turns out, the distribution in this case largely reduces to the question for SWFPP.

\begin{theorem}
	\label{thm:SJRfromSWFPP}
	Let the switching variables of our SJR model be $\Ber(\alpha)$-distributed and let $\xi^* = (1 - \alpha, \alpha)$ be the percolation angle. Let $\xi_1 > \cdots > \xi_n > \xi^*$ be a sequence of angles and let $(B^{\xi_1}, \dots, B^{\xi_n})$ be the associated Busemann functions. Consider now $(B_0^{\xi_1}, \dots, B_0^{\xi_n})$, the Busemann functions for the SWFPP model whose horizontal weights are identical to those of our original model, but whose vertical weights are all zero. Then
	\begin{equation}
		\label{eq:SJRfromSWFPP}
		(B^{\xi_1}, \dots, B^{\xi_n}) \disteq (B_0^{\xi_1}, \dots, B_0^{\xi_n}).
	\end{equation}
\end{theorem}
This gives a near-complete description of the Busemann process when the weights are Bernoulli-exponential, although we still cannot say much about the joint distribution of $(B^{\xi_1}, B^{\xi_2})$, when $\xi_1 > \xi^* > \xi_2$. Two special cases are the river delta model \cite{barr-rych-river-delta} with exponential weights, and the degenerate model wherein each vertex has weight $1$ on a random incoming edge and $0$ on the other. 
%
%\begin{remark}
%	That we have a distributional equality in \eqref{eq:SJRfromSWFPP} implies that there is a coupling of the two models (the SJR and the SWFPP) witnessing this equality. Curiously, this is not the natural coupling.
%\end{remark}

The connection between the Busemann process and the limit shape contained in \cref{thm:BusemannExistenceDistinct} gives a decomposition in terms of two SWFPP limit shapes.
\begin{corollary}[Theorem 1 of \cite{rans-sjr}]
	\label{cor:SJRLimitShape}
	Let $\lshape$ be the time constant for an SJR model, and $\lshape_H$ (resp. $\lshape_V$) be the time constant for the SWFPP whose horizontal (resp. vertical) weights are unchanged but whose vertical (resp. horizontal) weights are constant zero. Then
	\begin{equation}
		\label{eq:SJRLimitShape}
		\lshape = \lshape_H \vee \lshape_V.
	\end{equation}
\end{corollary}
Ransford's original proof is considerably more direct and avoids the Busemann machinery, but both Ransford's proof and our own proof of \cref{thm:SJRfromSWFPP} have as their essential ingredient the observation contained in Lemma 2 of \cite{rans-sjr}, which implies that SJR and an appropriate SWFPP are indistinguishable in stationarity.

\subsection{Highways and byways} 
\label{ssec:HighwaysResults}
The immediate purpose of Newman's introduction of Busemann functions to FPP \cite{newman-first-buse} was to study the existence and structure of semi-infinite geodesics. The existence of the Busemann limits and their ability to produce semi-infinite geodesics in a given direction immediately implies the following general theorem. This is a direct analogue of Theorem 2.1 of \cite{geor-rass-sepp-geodesics} in our slightly more general setting, and is largely contained in Theorem 5.8 of \cite{groa-janj-rass-gen-buse}.

\begin{corollary}
	\label{thm:BuseGeodesics}
	Return to the general setting of \cref{ass:WeightAssump} and consider directions $\xi \in \scrU$ at which the limit shape is differentiable and strictly convex. Then
	\begin{thmlist}
		\item Almost surely, simultaneously for all such $\xi$, there is at least one geodesic ray from the origin in direction $\xi$.
		\item For a fixed $\xi$, almost surely the geodesic ray from the origin is unique.
		\item For fixed $\xi$, almost surely all geodesic rays (and not just those from the origin) in direction $\xi$ coalesce.
	\end{thmlist}
\end{corollary}

With the explicit descriptions of the Busemann process for the Bernoulli-exponential SWFPP described above, we may hope for finer information on the nature of the semi-infinite geodesics. Let $\Gamma$ be the tree of (right-)geodesics rooted at $0$ and let $\Gamma_\infty$ be the subtree formed by its infinite branches. We identify $\Gamma_\infty$ with its set of vertices. The subtree trivially contains the axes, and from \cref{thm:BuseGeodesics} we see that it contains a continuum of topological ends.

We call the edges (and by abuse of language, also the vertices) of $\Gamma_\infty$ \emph{highways}. The underlying analogy is that one expects a geodesic between distant points to first travel to seek out a path of highly favourable edges going in the desired direction, travel with this path for most of the journey, and then turn off only when close to its destination. These highly favourable edges form the semi-infinite geodesics. The highways-and-byways problem mentioned in the introduction may be understood as determining the asymptotics of
\begin{equation}
	\label{eq:HighwaySquare}
	\Ex*{\frac{\abs[\big]{\Gamma_\infty \cap ([0, n] \times [0, n])}}{n^2}},
\end{equation}
which is the expected proportion of vertices inside an $n \times n$ square footed at the origin which lie on a highway from the origin. The specific question found in \cite{hamm-welsh-fpp} is whether this proportion goes to zero. While this has been answered affirmatively in the harder setting of undirected FPP by \cite{ahlb-hans-hoff-density, dem-elb-pel-midpoint}, one may ask: rather than look at the limit in expectation, is there anything to be said of the almost sure limit? That the expected proportion in \eqref{eq:HighwaySquare} goes to zero immediately implies by Fatou's lemma that the limit infimum of the proportion is zero almost surely, but doesn't rule out \emph{a priori} some intermittent explosive branching in the tree producing large clusters of highways.

We don't provide an answer to the question outlined above, but can say something in this spirit. Namely, we consider the intersection of $\Gamma_\infty$ not with a square but with a thin rectangle, and find that there is nontrivial behaviour. The following results are stated for SWFPP with rate 1 exponential weights. 

\begin{theorem}
	\label{thm:HighwaysRect}
	Write $A_k(n)$ for the cardinality of $\Gamma_\infty \cap (\set{k} \times [0, n])$. Then for any $k \ge 1$, we have
	\begin{equation}
		\label{eq:HighwaysRect}
		\liminf_{n\to\infty}\frac{{A_{k}(n)}}{n}=0,\qquad \limsup_{n\to\infty}\frac{{A_{k}(n)}}{n}=1\mathrm{\;{a.s.}}
	\end{equation}
\end{theorem}

\begin{theorem}
	\label{thm:HighwaysBranch}
	Write $B(n)$ for the number of branches of $\Gamma_\infty$ along the vertical axis up to height $n$ --- that is, the number of connected components of $\Gamma_\infty \cap (\set{1} \times [0, n])$. Then there exist constants $0 < C_1 \le C_2 < 1$, such that for all $n$ large enough we have 
	\begin{equation}
		\label{eq:HighwaysBranch}
		C_1 \log n \le B(n) \le C_2 \log n.
	\end{equation}
\end{theorem}

The strict bound $C_2 < 1$ is conditional on a claimed result in \cite{mart-schi-discrete} concerning collisions in a discrete-time TASEP speed process, for which the authors do not provide a proof. Without this input, our statement still holds, except weakened to have $C_2 \le 1$.

\subsection{The competition interface} 
\label{ssec:Competition}
The geodesic connecting the origin to $x \in \bZ_{\ge 0}^2$, if it is unique, must go through exactly one of $e_1$ or $e_2$. Write $\cR_1 = \set{x \in \bZ_{\ge 0}^2 : x \ne 0,\, e_1 \in \gamma_x}$, and $\cR_2$ the same with $e_2$ in place of $e_1$. These sets are disjoint and the overlap of their boundaries is called the \emph{competition interface} for the tree rooted at the origin. The interface is known to be asymptotically linear in LPP under mild assumptions on the limit shape \cite[Theorem 2.6]{geor-rass-sepp-geodesics}. To state the analogous result here, let $r_n = \max \set{r \ge 0 : (r, n) \in \cR_1}$ be the horizontal position of the interface at height $n$.

\begin{proposition}
	\label{prop:CriticalAngle}
	In the general setup of \cref{ass:WeightAssump}, suppose our weights follow a continuous distribution and that $\lshape$ is differentiable everywhere. Then the limit
	\begin{equation}
		\label{eq:CriticalAngleDef}
		\xi^* = \lim_{n \to \infty}\frac{r_n}{r_n + n}
	\end{equation}
	exists almost surely, and $\lshape$ has an exposed point in direction $\xi^*$.
\end{proposition}

This $\xi^*$ is the \emph{critical angle} at the origin. We may define $\xi^*(x)$ in the same way, using the tree of geodesics rooted at $x$. 

The critical angle is intimately related to the discontinuity points of the Busemann process. The following may be seen by considering the procedure of \cref{rem:GeodesicConstruction}. Recall that we order angles by saying that an angle closer to the horizontal is larger.

\begin{proposition}
	\label{prop:CritFromBuse}
	In SWFPP and under the assumptions of \cref{prop:CriticalAngle}, we may alternatively define the critical angle as
	\begin{equation}
		\xi^* = \sup \set{\xi : B^\xi (0, e_2) = 0}.
	\end{equation}
	In particular, by monotonicity,
	\begin{equation}
		\Prob{\xi^* \ge \xi} = \Prob{B^\xi (0, e_2) = 0}.
	\end{equation}
\end{proposition}

The joint distributions of critical angles along a line are, morally, related to the distribution of speeds in a corresponding speed process. For exponential LPP, this is the TASEP speed process \cite{amir-angel-valko-tasep-speed}, while the integrable models in SWFPP are closely related to the discrete-time particle systems considered in \cite{mart-schi-discrete}. There is some nontrivial correlation between the speeds of each particle, and the formation of ``convoys'' is observed. These are sets of particles sharing the same speed, which after a long time eventually meet and travel together. This is surprising in light of the fact that the speeds have continuous marginal distributions. A description of a single convoy is found in \cite[Theorem 1.8]{amir-angel-valko-tasep-speed}, and here we give the analogue in SWFPP. For the following, recall that a renewal process can be defined as a point set in $\bR$ whose increments (\emph{holding times}) are independent and identically distributed.

\begin{theorem}
\label{thm:SWFPPCompetitionAngles}
\begin{thmlist}
	Consider Bernoulli or Bernoulli-Exponential weights.
	\item \label{thm:SWFPPCompetitionAnglesProbs} The ordering of adjacent competition interfaces has
	\begin{align}
		\Prob{\xi^*(0) > \xi^*(e_2)} &= \frac{1}{3},\\
		\Prob{\xi^*(0) = \xi^*(e_2)} &= \frac{1}{6},\\
		\Prob{\xi^*(0) < \xi^*(e_2)} &= \frac{1}{2}.
	\end{align}
	
	\item \label{thm:SWFPPCompetitionAnglesInterval} For fixed $\xi_1 \in \scrU$, the set $\set{k : \xi^*(k e_2) \ge \xi_1}$ is a Bernoulli process with density $p_1 = \Prob{B^{\xi_1} (0, e_2) = 0}$. If $\xi_2 > \xi_1$, then the set $\set{k : \xi_1 \le \xi^*(k e_2) \le \xi_2}$ is a renewal process. whose holding times have mean $(p_1 - p_2)^{-1}$.
	
	\item \label{thm:SWFPPCompetitionAnglesConvoy} Write $\alpha = \xi^*(0)$. Then $C_\alpha = \set{k : \xi^*(k e_2) = \alpha}$ is a renewal process whose holding times have infinite mean, and asymptotically satisfies
	\begin{equation}
		\label{eq:IntroConvoyDensity}
		\lim_{n \to \infty}\frac{\abs{C_\alpha \cap [0, n]}}{\sqrt{n}} = c_\alpha \as,
		\end{equation}
		where $c_\alpha$ is an explicit constant depending on $\alpha$. In particular, almost surely every convoy along the vertical axis is infinite.
	\end{thmlist}
\end{theorem}

Of these, \cref{thm:SWFPPCompetitionAnglesProbs} is a straightforward computation using the multi-line distributions, and we refer to \cite[Theorem 1.7]{amir-angel-valko-tasep-speed} for an example. To our knowledge, the asymptotic density for a convoy in the TASEP speed process (in the sense of \eqref{eq:IntroConvoyDensity}) has not appeared explicitly, but our calculation may be followed with little modification to produce the leading order term. The explicit constant $c_\alpha$ for exponential weights is given in \cref{prop:ConvoyDensityConstant}.

%\subsection{Literature review} 

\subsection{Outline for the remainder of the paper} 
\label{ssec:Outline}
In \cref{sec:GeneralUpdates} we give the background on the queueing maps which underlie the proofs in the integrable cases and prove the uniqueness of their fixed points. That the multi-line distributions yield such fixed points is verified in \cref{sec:Multiline}. The calculations behind the results on semi-infinite geodesics and competition interfaces are given in \cref{sec:NearAxisComps}. The extension of Martin's uniform limit shape convergence to our setting forms \cref{app:UniformShape}, and the piecing together of the elements of \cite{geor-rass-sepp-17-buse} and \cite{groa-janj-rass-gen-buse} to establish the convergence of the Busemann limits is in \cref{app:Busemann}.

\resumetoc
\section{Update maps and their fixed points}
\label{sec:GeneralUpdates}
As well as point-to-point passage times $\ptime{x, y}$, we may also consider \emph{line-to-point} passage times $L^\Gamma (x)$ from a path $\Gamma \subseteq \bZ^2$. We must specify a function $g : \Gamma \to \bR$ representing passage times on the path, and then define
\begin{equation}
	\label{eq:BoundaryPassage}
	L^\Gamma (x) = \min_{p \in \Gamma}g(p) + \ptime{p, x}.
\end{equation}
Observe that if $\Gamma \subseteq \bZ_{\ge 0}^2$ is a \emph{down-right} path (one consisting only of steps $e_1$ and $-e_2$), then by setting $g(p) = \ptime{p}$, we have that the passage times $\ptime{x}$ and $L^\Gamma(x)$ are equal whenever both are defined.

Shifting $g$ by a constant will merely shift the resulting passage times $L^\Gamma(x)$ by the same constant, so we may consider the passage times to be instead defined, up to a constant, by the increments $I = \Delta g = \set{g(p_{i + 1}) - g(p_i) : i \in \bZ,\,(p_i, p_{i + 1}) \in \Gamma}$. In this indexing, the points of our path are ordered from top-left to bottom-right. 

Suppose we translate our path forward to $\Gamma + e_1$ and look at $g' : \Gamma \to \bR$ given by $g'(p) = L^\Gamma(p + e_1)$. The resulting map on the increments taking $\Delta g = I \mapsto I' = \Delta g'$ is called an \emph{update map}. Since this map depends on the weights on edges above and to the right of vertices of $\Gamma$, we denote this set of weights by $W$ and write $I' = U^\Gamma(I, W, e_1)$. We can of course replace $e_1$ by any vector in $\bZ_{\ge 0}^2$.

This map is not well-defined for arbitrary sequences $I,\, W$, so for the time being let us call these sequences \emph{admissible} if the map makes sense. We will give precise constraints when we consider specific examples of update maps below.

A useful property that holds in complete generality is that update maps are monotone in the first two arguments. 
\begin{lemma}
	\label{lem:UpdateMapMono} Let $\Gamma \subseteq \bZ^2$ be an arbitrary down-right path and take $v \in \bZ_{\ge 0}^2$. If $I,\, I'$ are sequences of increments along $\Gamma$ and $W,\, W'$ are sets of weights adjacent to $\Gamma$, all admissible, then $I \ge I'$ and $W \ge W'$ implies
	\begin{equation}
		\label{eq:UpdateMapMono}
		U^\Gamma(I, W, v) \ge U^\Gamma(I', W', v).
	\end{equation}
	Here the comparisons are made component-wise.
\end{lemma}

Our interest is less so in the update maps themselves, but rather in their fixed points. We call a distribution $\rho$ a \emph{fixed point} of the update map if when $I \sim \rho$ and $W \sim \lambda$, we get $U^\Gamma(I, W, v) \sim \rho$ (either for some fixed $v$ or all $v \in \bZ_{\ge 0}^2$ --- clearly having this for all $v \in \set{e_1, e_2}$ is enough to imply the latter). The Busemann process provides a source of particularly strong fixed points. In the statement below, it is important to consider south-west Busemann functions. 

\begin{proposition}
	\label{prop:BusemannFixedPoint}
	In the general setting of \cref{ass:WeightAssump}, let $\xi_1 > \cdots > \xi_n \in \scrU$ be a sequence of directions and consider $I^k = \set{B^{\xi_k}(p_{i + 1}, p_{i}) : i \in \bZ,\, (p_i, p_{i + 1}) \in \Gamma}$. Then if $W$ is the set of weights adjacent to $\Gamma$,
	\begin{equation}
		\label{eq:BusemannFixedPoint1}
		U^\Gamma(I^k, W, v) = (I')^k,
	\end{equation}
	where the right sequence consists of the Busemann functions evaluated along the shifted path $\Gamma + v$. 
	
	In particular
	\begin{equation}
		\label{eq:BusemannFixedPoint2}
		\brac[\big]{U^\Gamma(I^1, W, v), \dots,  U^\Gamma(I^n, W, v)} \disteq (I^1, \dots, I^n).
	\end{equation}
	and the distributions $(I^1, \dots, I^n)$ are jointly ergodic under applications of the update map.
\end{proposition}

The first claim may be seen directly from the definition of the Busemann limits in \eqref{eq:SWBusemannLimits}. The second claim follows from the stationarity in \cref{thm:BusemannExistenceErgodicity}.

\subsection{Update maps in SWFPP}
\label{ssec:UpdateMaps}
As one might expect, the update maps associated with the horizontal and vertical lines, as well as the antidiagonal, have especially simple explicit expressions and are quite tractable. This is even more true when we restrict ourselves to SWFPP. For each choice of path, let us define the maps in terms of increments and weights, and see how they may be recast in terms of the store model introduced in \cref{ssec:ModelDefs} and as the dynamics of various interacting particle systems. 

We quickly note that passage times in SWFPP possess a monotonicity not found in general edge-weight FPP. Recall that in SWFPP, horizontal edges are non-negative and vertical edges are zero.
\begin{lemma}
	\label{lem:SWFPPMono}
	Take $y \le x \in \bZ_{\ge 0}^2$ and suppose we are in the setting of SWFPP. Then
	\begin{equation}
		\label{eq:SWFPPMono}
		\ptime{y, x + e_2} \le \ptime{y, x} \le \ptime{y, x + e_1}.
	\end{equation}
	If the horizontal weights satisfy $\weight{x, x + e_1} > 0$, then the right inequality is strict and the left has equality if and only if the geodesic $\gamma(y, x + e_2)$ goes through $x$.
\end{lemma}
It will be convenient in this context to write $\weight{x - e_1, x} = \weight{x; 1} = \weight{x}$.

\subsubsection{The antidiagonal update} 
\label{sssec:UpdateMapsA}
Let $\Gamma = \set{k(e_1 - e_2) : k \in \bZ}$ be the antidiagonal. We can happily take our increments $I$ to be completely arbitrary non-negative sequences, but it will be enough for us to consider those in the set
\begin{equation}
	\label{eq:AIncSet}
	\cS^{A} = \set[\Big]{(Y_k) \in \bR_{\ge 0}^{\bZ} : \lim \frac{1}{n}\sum_{j = k}^{k + n - 1}Y_j = \mu, \text{ for some } \mu \ge 0  \text{ and all } k \in \bZ}.
\end{equation}
This is simply the set of non-negative sequences with a well-defined average value. This average value is a conserved quantity (see \cref{lem:AMass}), and so it is natural to impose this restriction. The same restriction is imposed in \cite{fan-sepp-joint-buss}, and is the discrete version of the slope condition which appears in the context of the KPZ equation and other continuous models \cite{grass-KPZ}.

Let $g : \Gamma \to \bR$ be passage times along the boundary. Write $W_k = \weight{k(e_1 - e_2) + e_1}$. Then if $x_k = k(e_1 - e_2) \in \Gamma$, the passage time of $x + e_1$ is simply
\begin{equation}
	\label{eq:APassageTimes}
	L^\Gamma(x_k + e_1) = (g(x_k) + W_k) \wedge g(x_{k + 1}).
\end{equation}
If $Y = \Delta g$ and $Y'$ is the sequence increments along $\Gamma + e_1$, then
\begin{equation}
	\label{eq:AIncrements}
	Y'_k = (Y_k - W_k)^+ + Y_{k - 1} \wedge W_{k - 1}.
\end{equation}
Define $A(Y, W) = Y'$ to represent this updating of the sequence $Y$.

Recall the store model from earlier, which we now extend to a $\bZ$-indexed tandem of bins. The bins receive input from their lower-indexed neighbour and output to their higher-indexed neighbour. At each time step, each bin is given an amount of service, and will move the lesser of the service amount, or the entire amount in the bin. Let $Y_k$ represent the amount in bin $k$ at a particular time, and imagine that we apply the services $W_k$ starting from the higher-indexed bins and working down. Then when bin $k$ has its service, it has not yet received any input from bin $k - 1$. It outputs $Y_k \wedge W_k$ and has $(Y_k - W_k)^+$ remaining. Then bin $k - 1$ has its service, and the total in bin $k$ comes to $(Y_k - W_k)^+ + Y_{k - 1} \wedge W_{k - 1}$, matching the expression in \eqref{eq:AIncrements}. This process is depicted in \cref{fig:TandemDynamics}.

\begin{figure}
	\begin{subfigure}[]{\textwidth}
		\centering
		\resizebox{0.75\textwidth}{!}{\begin{tikzpicture}[x=1cm,y=1cm]
	%%-------------------- USER KNOBS --------------------
	%%Fractions of bin height, all in [0,1].
	\pgfmathsetmacro{\X}{0.5}   % initial content level
	\pgfmathsetmacro{\I}{0.3}   % incoming input level
	\pgfmathsetmacro{\W}{0.60}   % service request level
	\pgfmathsetmacro{\Gap}{0.05}   % service request level
	%%---------------------------------------------------
	
	% Derived quantities
	\pgfmathsetmacro{\XI}{\X+\I}
	\pgfmathsetmacro{\OUT}{min(\W,\XI)}
	\pgfmathsetmacro{\REM}{max(\XI-\W,0)} % leftover after service
	
	% Bin geometry
	\def\BW{2.2}     % bin width
	\def\BH{3.0}     % bin height
	\def\GAP{3.6}    % horizontal gap between stages
	
	% Convenience heights
	\pgfmathsetmacro{\hX}{\X*\BH}
	\pgfmathsetmacro{\hI}{\I*\BH}
	\pgfmathsetmacro{\hXI}{\XI*\BH}
	\pgfmathsetmacro{\hW}{\W*\BH}
	\pgfmathsetmacro{\hOUT}{\OUT*\BH}
	\pgfmathsetmacro{\hREM}{\REM*\BH}
	\tikzset{
		>={Classical TikZ Rightarrow[length=3pt,width=6pt]},
		bin/.style={draw,rounded corners=2pt,very thick,minimum width=\BW cm,minimum height=\BH cm},
		levelX/.style={fill=blue!50},
		levelXp/.style={fill=blue!30},
		levelI/.style={fill=green!55!black!25},
		levelIp/.style={fill=green!35!black!15},
		levelIt/.style={fill=green!40!black!5},
		levelY/.style={fill=orange!80!black!55},
		levelYp/.style={fill=orange!60!black!35},
		lab/.style={font=\sffamily\small},
		title/.style={font=\sffamily\bfseries,align=center},
		demandbrace/.style={decorate,decoration={brace,amplitude=5pt},thick},
		flowarrow/.style={->,very thick},
		panel/.style={draw=black!10,rounded corners=6pt,fill=black!2,inner sep=8pt}
	}
	\begin{scope}[shift={(-100,0)}]
		% panel background
		\node[inner sep=10pt] (p1) at (0,0) {%
			\begin{tikzpicture}
				% Define box dimensions and the gap between them
				\def\BW{1cm}
				\def\BH{1.5cm}
				\def\props{{0.6,0.7,0.4,0.5,0.9,0.3,0.4}}
				\def\gap{0.4cm}
				
				\foreach[evaluate=\i as \myi using int(\i - 3)] \i in {0,...,2} {
					% Use a scope to shift each figure horizontally
					\begin{scope}[shift={(\i * \BW + \i * \gap, 0)}]
						% after service
						\path[levelY] (0,0) rectangle (\BW,\BH * \props[\i]);
						\node[lab] at (\BW/2, -0.3) {$Y_{\myi}$};
						
						\draw[very thick,rounded corners=2pt] 
						(0,0) -- (0,\BH)      % left wall 
						(0,0) -- (\BW,0)      % bottom 
						(\BW,0) -- (\BW,\BH); % right wall
					\end{scope}
				}
				\foreach[evaluate=\i as \myi using int(\i - 3)] \i in {3,...,3} {
					% Use a scope to shift each figure horizontally
					\begin{scope}[shift={(\i * \BW + \i * \gap, 0)}]
						% after service
						\path[levelX] (0,0) rectangle (\BW,\BH * \props[\i]);
						\node[lab] at (\BW/2, -0.3) {$Y_{\myi} - I_{1}$};
						
						\draw[very thick, dashed, levelIt] (0, \BH * \props[\i]) rectangle (\BW, \BH * \props[\i] + \BH * 0.2);
						%						\node[lab] at (\BW/2, \BH * \props[\i] + \BH * 0.35) {$I_{\myi}$};
						
						\draw[very thick,rounded corners=2pt] 
						(0,0) -- (0,\BH)      % left wall 
						(0,0) -- (\BW,0)      % bottom 
						(\BW,0) -- (\BW,\BH); % right wall
						
						% service W_0
						\draw[very thick]
						(\BW+5.0,\BH * \props[\i]) -- ++(-0.1,0)              % bottom tick
						(\BW+5.0,\BH * \props[\i]) -- (\BW+5.0,\BH * \props[\i] + \BH * 0.2)          % vertical
						(\BW+5.0,\BH * \props[\i] + \BH * 0.2) -- ++(-0.1,0);           % top tick
						\node[right] at (\BW+0.5,\BH * \props[\i] + \BH * 0.1) {$W_{\myi}$};
					\end{scope}
				}
				\foreach[evaluate=\i as \myi using int(\i - 3)] \i in {4,...,6} {
					% Use a scope to shift each figure horizontally
					\begin{scope}[shift={(20 + \i * \BW + \i * \gap, 0)}]
						% after service
						\path[levelYp] (0,0) rectangle (\BW,\BH * \props[\i]);
						\node[lab] at (\BW/2, -0.3) {$Y'_{\myi}$};
						
						\draw[very thick,rounded corners=2pt] 
						(0,0) -- (0,\BH)      % left wall 
						(0,0) -- (\BW,0)      % bottom 
						(\BW,0) -- (\BW,\BH); % right wall
					\end{scope}
				}
				% Loop to draw arrows between bins
				\foreach \i in {-1,...,2} {
					\draw[->, dashed, thick, black]
					( {\i*(\BW+\gap) + \BW * 0.6}, \BH ) to[bend left=30] ( {(\i+1)*(\BW+\gap) + \BW * 0.4}, \BH );
				}
				\foreach \i in {3,...,3} {
					\draw[->, thick, black]
					( {\i*(\BW+\gap) + \BW * 0.6}, \BH ) to[bend left=30] ( {(\i+1)*(\BW+\gap) + 20 + \BW * 0.4}, \BH );
				}
				\foreach \i in {4,...,6} {
					% Draw a gray, arced arrow from the top-right of the current
					% bin to the top-left of the next bin.
					\draw[->, thick, black]
					( {\i*(\BW+\gap) + 20 + \BW * 0.6}, \BH ) to[bend left=30] ( {(\i+1)*(\BW+\gap) + 20 + \BW * 0.4}, \BH );
				}
			\end{tikzpicture}
		};
	\end{scope}
	
\end{tikzpicture}}
		\subcaption{The movement of material considered in the $A$ map. The bins are processed right to left. Here all of the bins to the right of and including bin $0$ have been processed. Since the service $W_0$ was less than the initial amount $Y_0$, the output from bin $0$ is $I_1 = W_0$. Next, the output from bin $-1$ will be received, completing the movement of material in bin $0$.}
	\end{subfigure}
	\begin{subfigure}[]{\textwidth}
		\centering
		\resizebox{0.75\textwidth}{!}{\begin{tikzpicture}[x=1cm,y=1cm]
	%-------------------- USER KNOBS --------------------
	% Fractions of bin height, all in [0,1].
	\pgfmathsetmacro{\X}{0.5}   % initial content level
	\pgfmathsetmacro{\I}{0.3}   % incoming input level
	\pgfmathsetmacro{\W}{0.60}   % service request level
	\pgfmathsetmacro{\Gap}{0.05}   % service request level
	%---------------------------------------------------
	
	% Derived quantities
	\pgfmathsetmacro{\XI}{\X+\I}
	\pgfmathsetmacro{\OUT}{min(\W,\XI)}
	\pgfmathsetmacro{\REM}{max(\XI-\W,0)} % leftover after service
	
	% Bin geometry
	\def\BW{2.2}     % bin width
	\def\BH{3.0}     % bin height
	\def\GAP{3.6}    % horizontal gap between stages
	
	% Convenience heights
	\pgfmathsetmacro{\hX}{\X*\BH}
	\pgfmathsetmacro{\hI}{\I*\BH}
	\pgfmathsetmacro{\hXI}{\XI*\BH}
	\pgfmathsetmacro{\hW}{\W*\BH}
	\pgfmathsetmacro{\hOUT}{\OUT*\BH}
	\pgfmathsetmacro{\hREM}{\REM*\BH}
	\tikzset{
		>={Classical TikZ Rightarrow[length=3pt,width=6pt]},
		bin/.style={draw,rounded corners=2pt,very thick,minimum width=\BW cm,minimum height=\BH cm},
		levelX/.style={fill=blue!50},
		levelXp/.style={fill=blue!30},
		levelI/.style={fill=green!55!black!25},
		levelIp/.style={fill=green!35!black!15},
		levelIt/.style={fill=green!40!black!5},
		levelY/.style={fill=orange!80!black!55},
		levelYp/.style={fill=orange!60!black!35},
		lab/.style={font=\sffamily\small},
		title/.style={font=\sffamily\bfseries,align=center},
		demandbrace/.style={decorate,decoration={brace,amplitude=5pt},thick},
		flowarrow/.style={->,very thick},
		panel/.style={draw=black!10,rounded corners=6pt,fill=black!2,inner sep=8pt}
	}
	\begin{scope}[shift={(0,0)}]
		% panel background
		\node[inner sep=10pt] (p1) at (0,0) {%
			\begin{tikzpicture}
				% Define box dimensions and the gap between them
				\def\BW{1cm}
				\def\BH{1.5cm}
				\def\props{{0.3,0.0,0.7,0.5,0.6,0.9,0}}
				\def\gap{0.4cm}
				
				\foreach[evaluate=\i as \myi using int(\i - 3)] \i in {0,...,2} {
					% Use a scope to shift each figure horizontally
					\begin{scope}[shift={(\i * \BW + \i * \gap, 0)}]
						% after service
						\path[levelXp] (0,0) rectangle (\BW,\BH * \props[\i]);
						\node[lab] at (\BW/2, -0.3) {$X'_{\myi}$};
						
						\draw[very thick,rounded corners=2pt] 
						(0,0) -- (0,\BH)      % left wall 
						(0,0) -- (\BW,0)      % bottom 
						(\BW,0) -- (\BW,\BH); % right wall
					\end{scope}
				}
				\foreach[evaluate=\i as \myi using int(\i - 3)] \i in {3,...,3} {
					% Use a scope to shift each figure horizontally
					\begin{scope}[shift={(\i * \BW + \i * \gap, 0)}]
						% after service
						\path[levelX] (0,0) rectangle (\BW,\BH * \props[\i]);
						\node[lab] at (\BW/2, -0.3) {$X_{\myi} + I_{\myi}$};
						
						\path[levelI] (0, \BH * \props[\i]) rectangle (\BW, \BH * \props[\i] + \BH * 0.2);
						%						\node[lab] at (\BW/2, \BH * \props[\i] + \BH * 0.35) {$I_{\myi}$};
						
						\draw[very thick,rounded corners=2pt] 
						(0,0) -- (0,\BH)      % left wall 
						(0,0) -- (\BW,0)      % bottom 
						(\BW,0) -- (\BW,\BH); % right wall
						
						% service W_0
						\draw[very thick]
						(\BW+5.0,0) -- ++(-0.1,0)              % bottom tick
						(\BW+5.0,0) -- (\BW+5.0,\BH * 0.8)          % vertical
						(\BW+5.0,\BH * 0.8) -- ++(-0.1,0);           % top tick
						\node[right] at (\BW+0.5,\BH * 0.4) {$W_{\myi}$};
					\end{scope}
				}
				\foreach[evaluate=\i as \myi using int(\i - 3)] \i in {4,...,6} {
					% Use a scope to shift each figure horizontally
					\begin{scope}[shift={(20 + \i * \BW + \i * \gap, 0)}]
						% after service
						\path[levelX] (0,0) rectangle (\BW,\BH * \props[\i]);
						\node[lab] at (\BW/2, -0.3) {$X_{\myi}$};
						
						\draw[very thick,rounded corners=2pt] 
						(0,0) -- (0,\BH)      % left wall 
						(0,0) -- (\BW,0)      % bottom 
						(\BW,0) -- (\BW,\BH); % right wall
					\end{scope}
				}
				% Loop to draw arrows between bins
				\foreach \i in {-1,...,2} {
					\draw[->, thick, black]
					( {\i*(\BW+\gap) + \BW * 0.6}, \BH ) to[bend left=30] ( {(\i+1)*(\BW+\gap) + \BW * 0.4}, \BH );
				}
				\foreach \i in {3,...,3} {
					\draw[->, dashed, thick, black]
					( {\i*(\BW+\gap) + \BW * 0.6}, \BH ) to[bend left=30] ( {(\i+1)*(\BW+\gap) + 20 + \BW * 0.4}, \BH );
				}
				\foreach \i in {4,...,6} {
					% Draw a gray, arced arrow from the top-right of the current
					% bin to the top-left of the next bin.
					\draw[->, dashed, thick, black]
					( {\i*(\BW+\gap) + 20 + \BW * 0.6}, \BH ) to[bend left=30] ( {(\i+1)*(\BW+\gap) + 20 + \BW * 0.4}, \BH );
				}
			\end{tikzpicture}
		};
	\end{scope}
	
\end{tikzpicture}}
		\subcaption{The movement of material considered in the $V$ map. The bins are processed left to right. Here all of the bins to the left of bin $0$ have been processed. Bin $-1$ has placed its output in bin $0$, adding it to the initial amount $X_0$. Bin $0$ will produce output according to the service amount $W_0$. Because $W_0 > X_0 + I_0$, all of the material in bin $0$ will be output, giving $I_1 = X_0 + I_0$.}
	\end{subfigure}
	\caption{}
	\label{fig:TandemDynamics}
\end{figure}

For another perspective, consider the boundary passage times $g(x_k)$ as being the position of a particle $\eta_k$ on the real line. These particles then move according to a simple rule: particle $\eta_k$ will attempt to move forward by $W_k$ but stops upon hitting particle $\eta_{k + 1}$. If we update the particles with lower index first, then the new positions $\eta'_k$ are
\begin{equation}
	\label{eq:AParticleDynamics}
	\eta'_k = (\eta_k + W_k) \wedge \eta_{k + 1},
\end{equation}
which we see is precisely the same relation as in \eqref{eq:APassageTimes}. We might consider applying the update in \eqref{eq:AParticleDynamics} repeatedly, to arrive at a discrete-time interacting particle system. If the initial positions are $\set{\eta_{k, 0}}$, then denote by $\set{\eta_{k, t}}$ the position after $t$ steps, using weights $\set{W_{k, s}}$. With Bernoulli and geometric weights, this is the parallel TASEP of \cite{ferr-mart-dual} and R2 of \cite{mart-schi-discrete}, respectively.

The point-to-point passage times can be seen as line-to-point passage times with appropriate boundary conditions.
\begin{lemma}
	\label{lem:AParticles}
	Let $\eta_{k, 0} = 0$ for $k \le 0$ and $\eta_{k, 0} = \infty$ for $k \ge 1$. Consider SWFPP passage times with horizontal weights $\set{\weight{x}}$. Then if we set $W_{k, t} = \weight{(t + 1)e_1 + k(e_1 - e_2)}$, we have the identity
	\begin{equation}
		\label{eq:AParticles}
		\eta_{k, t} = \ptime{0, t e_1 + k(e_1 - e_2)}
	\end{equation}
	for all $k \le 0$.
\end{lemma}

\subsubsection{The vertical update} 
\label{sssec:UpdateMapsV}
Now let $\Gamma = \set{-k e_2 : k \in \bZ}$ be the vertical axis oriented downward. We would again like our sequences to have a well-defined average, but now this average must be positive to ensure the update is well-defined. Set
\begin{equation}
	\label{eq:VIncSet}
	\cS^{V} = \set[\Big]{(X_k) \in \bR_{\ge 0}^{\bZ} : \lim \frac{1}{n}\sum_{j = k}^{k + n - 1}X_j = \mu, \text{ for some } \mu > 0  \text{ and all } k \in \bZ}.
\end{equation}
As a means of distinguishing the two settings, we use $X = (X_k)$ to denote the increments here.

As before, let $g : \Gamma \to \bR$ be passage times along the boundary and write $W_k = \weight{- k e_2 + e_1}$. Letting $x_k = -k e_2 \in \Gamma$,
\begin{equation}
	\label{eq:VPassageTimes}
	L^\Gamma(x_k + e_1) = \min_{l \ge k} \brac[\big]{g(x_l) + W_{l}}.
\end{equation}
The expression on the increment side is more complicated:
\begin{equation}
	\label{eq:VIncrements}
	X'_k = (X_k + I_k - W_k)^+,
\end{equation}
where $I_k = \min_{l \ge k + 1} W_l + \sum_{j = k}^{l - 1}X_j$. This minimum exists for $X \in \cS^{V}$ almost surely, precisely because we require the average $\mu$ appearing \eqref{eq:VIncSet} to be positive. Write $V(X, W) = X'$.

This map can be interpreted in the same ways as the antidiagonal map. Now the tandem of bins is updated from the lower-indexed bins up, so that the output from bin $k - 1$ is added to the quantity in bin $k$ before the service is processed. We refer again to \cref{fig:TandemDynamics}. Similarly, in the interacting particle system, we move the particles beginning with the lowest index. With Bernoulli weights, this is sequential TASEP of \cite{ferr-mart-dual}, or alternatively the R1 of \cite{mart-schi-discrete}. This rule is illustrated in \cref{fig:SeqTASEP}

\begin{figure}
	\centering
	\begin{subfigure}[]{0.35\textwidth}
	\resizebox{\textwidth}{!}{\includegraphics[scale=1, page=7]{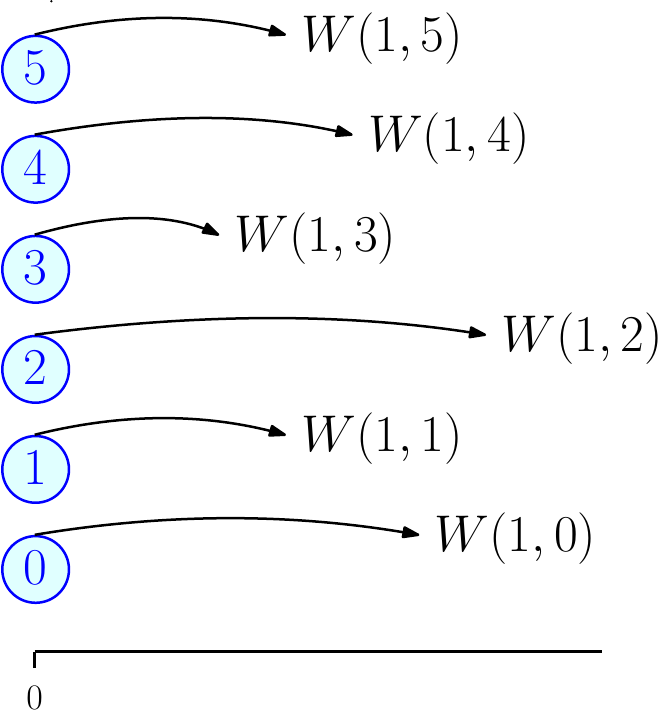}}
	\subcaption{Before the first time step.}	
	\end{subfigure}
	\hspace{1cm}
	\begin{subfigure}[]{0.35\textwidth}
		\resizebox{\textwidth}{!}{\includegraphics[scale=1, page=8]{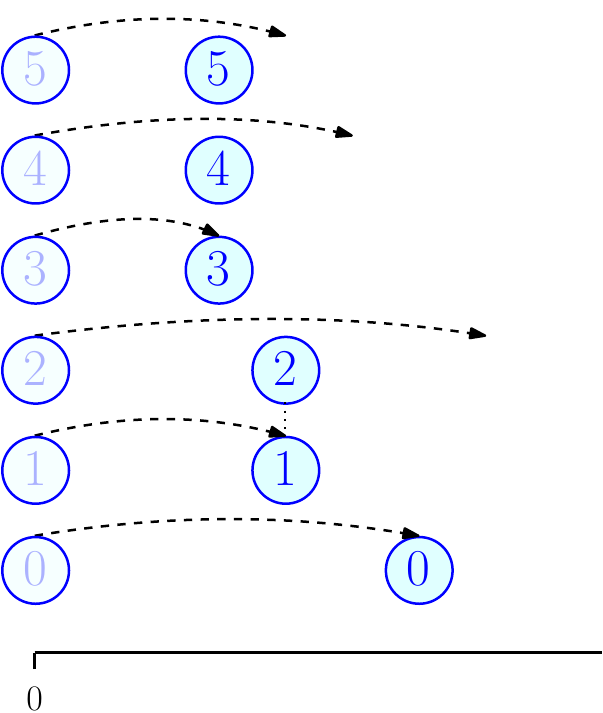}}
		\subcaption{Immediately after the first time step.}	
	\end{subfigure}
	\caption{A time step under the sequential TASEP dynamics, with the particles initially stacked up at the origin. The particles each receive a jump size (weight), and will attempt to jump that distance unless blocked by the particle in front of them.}
	\label{fig:SeqTASEP}
\end{figure}

\begin{remark}
	TASEP enjoys a symmetry under exchanging particles and holes: as particles in TASEP move rightwards, the holes move leftwards under the same dynamics. This symmetry breaks down in discrete-time TASEP, but there remains some particle-hole duality. If we consider sequential TASEP with Bernoulli jumps, then the holes undergo parallel TASEP with geometric jumps. Using this duality, we may write passage times under Bernoulli weights as a function of passage times under a coupled geometric-weight model, and vice versa. In this sense, it is natural that both Bernoulli and geometric / exponential weights give rise to solvable SWFPP models, and that interpolations in the form of Bernoulli-geometric weights remain tractable. 
\end{remark}

\subsubsection{The horizontal update} 
\label{sssec:UpdateMapsH}
The last map to consider is that for the horizontal line $\Gamma = \set{k e_1 : k \in \bZ}$. For a fixed sequence of weights $W$, we require our increments to be smaller on average than the weights, in the sense of belonging to the set
\begin{equation}
	\label{eq:HIncSet}
	\cS^{H} = \set[\Big]{(I_k) \in \bR_{\ge 0}^{\bZ} : \lim \frac{1}{n}\sum_{j = k}^{k + n - 1}I_j = \mu < \limsup \frac{1}{n}\sum_{j = k}^{k + n - 1}W_j  \text{ for all } k \in \bZ}.
\end{equation}

With $g : \Gamma \to \bR$ the passage times along $\Gamma$ with increments $\Delta g = I \in \cS^{H}$ and $x_k = k e_1 \in \Gamma$, we have
\begin{equation}
	\label{eq:HPassageTimes}
	L^\Gamma(x_k + e_2) = \min_{l \le k} \brac[\big]{g(x_l) + \sum_{j = l + 1}^{k}W_{j}}.
\end{equation}
For the increments, the update is
\begin{equation}
	\label{eq:HIncrements}
	I'_k = (I_k + X_k) \wedge W_k,
\end{equation}
where $X_k = \max_{l \le k} \sum_{j = l}^{k - 1}(I_j - W_j)$. The existence of the maximum is where we need $I \in \cS^{H}$. Write $H(I, W) = I'$.

This map $H$ is precisely the $D$ map considered in \cite{mart-prab-fixed}.\footnote{Except that in Section 2 of that paper, they have $D$ refer to the G/G/1 queue instead.} Indeed, that the measures $\mu_{H}^{\rho}$ are multi-class invariant measures for this map was already proved in that work. 

We do not make use of the particle system representation for $H$, but we provide an illustration in \cref{fig:HParticles}. The picture makes clear a connection to the G/G/1 queue, particularly the graphical representation found in \cite[Figure 2]{ferr-mart-mtasep}. We come back to this comparison in \cref{ssec:LPPConnection}, where we show that our $H$ map is dual in some sense to the $D$ map of \cite{fan-sepp-joint-buss}, which can be seen as the output of a G/G/1 queue.

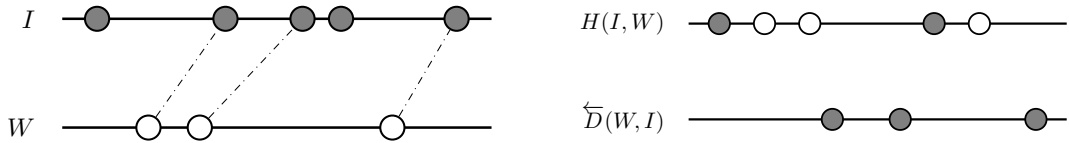
\begin{figure}
	\centering
	\begin{subfigure}[]{0.45\textwidth}
		\centering
		\resizebox{\textwidth}{!}{\begin{tikzpicture}
    [service/.style={circle,draw=black!100,fill=white!100,thick},
    input/.style={circle,draw=black!100,fill=gray!100,thick}]
    \node (upper) at (0, 1.5) [label=left:$I$]{};
    \draw[line width=1] (upper) -- ++(6,0);
    \node (lower) at (0, 0) [label=left:$W$]{};
    \draw[line width=1] (lower) -- ++(6,0);
    
    \node[service] (serv1) [right=of lower, xshift=0] {};
    \node[service] (serv2) [right=of lower, xshift=20] {};
    \node[service] (serv3) [right=of lower, xshift=95] {};
    
    \node[input] (input1) [right=of upper, xshift=-20] {};
    \node[input] (input2) [right=of upper, xshift=30] {};
    \node[input] (input3) [right=of upper, xshift=60] {};
    \node[input] (input4) [right=of upper, xshift=75] {};
    \node[input] (input5) [right=of upper, xshift=120] {};

    \draw[dash dot] (serv1) -- (input2);
    \draw[dash dot] (serv2) -- (input3);
    \draw[dash dot] (serv3) -- (input5);
\end{tikzpicture}}
	\end{subfigure}
	\hspace{2em}
	\begin{subfigure}[]{0.45\textwidth}
		\centering
		\resizebox{\textwidth}{!}{\begin{tikzpicture}
    [service/.style={circle,draw=black!100,fill=white!100,thick},
    input/.style={circle,draw=black!100,fill=gray!100,thick}]
    \node (upperpost) at (0, 1.5) [label=left:{$H(I, W)$}]{};
    \draw[line width=1] (upperpost) -- ++(6,0);
    \node (lowerpost) at (0, 0) [label=left:{$\overleftarrow{D}(W, I)$}]{};
    \draw[line width=1] (lowerpost) -- ++(6,0);
    
    \node[service] (serv1) [right=of upperpost, xshift=0] {};
    \node[service] (serv2) [right=of upperpost, xshift=20] {};
    \node[service] (serv3) [right=of upperpost, xshift=95] {};
    
    \node[input] (input1) [right=of upperpost, xshift=-20] {};
    \node[input] (input2) [right=of lowerpost, xshift=30] {};
    \node[input] (input3) [right=of lowerpost, xshift=60] {};
    \node[input] (input4) [right=of upperpost, xshift=75] {};
    \node[input] (input5) [right=of lowerpost, xshift=120] {};
\end{tikzpicture}}
	\end{subfigure}
	\caption{An illustration of a particle interpretation of the $H$ map. On the left are the arguments to the map. The inter-particle distances on the top row are given by the sequence $I$ and on the bottom by $W$. The particles on the bottom row are matched to a particle ahead of and above them. The unmatched particles on the top row stay where they are, and the matched particles move back, switching places with their match. The result is $H(I, W)$. The leftover particles are themselves interesting and give the output of a G/G/1 queue. Compare with \cite[Figure 2]{ferr-mart-mtasep}.}
	\label{fig:HParticles}
\end{figure}

\subsection{Uniqueness of fixed points} 
\label{ssec:FPUniqueness}
The strategy for identifying the distribution of the Busemann process, which we borrow from \cite{fan-sepp-joint-buss}, requires that the maps we defined above have unique ergodic fixed points. We introduce some notation before stating the theorem. Fix $n$ and draw weights $(\weight{x})_{x \in \bZ^2}$ from $\lambda$ and let $M^{V}$ be the set of measures on $(\cS^{V})^n$ which are jointly stationary and ergodic under shifts in the index and under applications of $V(\cdot, W)$. That is, if $X_{(0)} = ((X^1_{0, k}), \dots, (X^n_{0, k})) \sim \mu \in M^{V}$ and $X_{(t + 1)} = V(X_{(t)}, (W_{t, k})_{k \in \bZ})$, we would like the tuples $(X^1_{t, k}, \dots, X^n_{t, k})$ to be stationary and ergodic under shifts in $t$ and $k$, $\lambda$-a.s. Define $M^{A}$ and $M^{H}$ analogously.

\begin{theorem}
	\label{thm:UniqueFixedPoints}
	For a map $U \in \set{A, V, H}$, the elements of $M^{U}$ are uniquely determined by their mean vectors, provided we are in one of the following situations:
	\begin{thmlist}
		\item \label{thm:UniqueFixedPointsGeneral} For each choice of $c,\, \epsilon > 0$, either $\Prob{\weight{0; 1} > c,\, 0 \le \weight{0; 2} < \epsilon} > 0$, or this condition holds with the axes swapped.
		\item \label{thm:UniqueFixedPointsSWFPP} The weights are non-negative and $\weight{0; 2} = 0 \as$, or this condition holds with the axes swapped. That is, we are in SWFPP.
		\item \label{thm:UniqueFixedPointsSJR} The weights are non-negative and $\weight{0; 1} \wedge \weight{0; 2}= 0 \as$. That is, we are in the SJR model.
	\end{thmlist}
\end{theorem}

We will not need \cref{thm:UniqueFixedPointsGeneral} in the sequel and merely note that it may be proved along the lines of \cite{chang-unique-queue}, following \cite{fan-sepp-joint-buss}.

Although we have only given explicit descriptions in the case of SWFPP, each of the update maps makes sense in our most general setting (that of \cref{ass:WeightAssump}). Of the three, the $A$ map is always the simplest, owing to its finite range of dependence. The following proposition shows that it is enough to establish uniqueness for $A$. 

\begin{proposition}
	\label{prop:MeasureBijection}
	Then there are bijections $M^{V} \leftrightarrow M^{A} \leftrightarrow M^{H}$. Moreover, if $m^{V}$ is the set of mean vectors for measures in $M^{V}$, and $m^{A},\, m^{H}$ are analogous, then these bijections restrict to bijections $m^{V} \leftrightarrow m^{A} \leftrightarrow m^{H}$.
\end{proposition}

An immediate consequence is that if suitably ergodic invariant measures for one map are determined by their mean, then this is true for the other maps.

\begin{proof}
	This is mostly obvious once we pass to the FPP setting, but for lightness of notation let us take $n = 1$ and just look at the relationship between $V$ and $H$. Suppose $X_{(0)} \sim \mu \in M^{V}$ and set $\Gamma_{(0)} = \set{- k e_2 : k \in \bZ}$. Denote by $L_{(0)}(t, k)$ the passage time on the right half-plane with boundary increments $X$. Recall that we defined such passage times in \eqref{eq:BoundaryPassage}. Then the $X_{(t)}$ in the statement is merely the sequence consisting of increments
	\begin{equation}
		X_{t, k} = L_{(0)}(t, -k - 1) - L_{(0)}(t, -k).
	\end{equation}
	By stationarity, we may find increments $X_{(-1)} \sim \mu$ such that the corresponding passage times $L_{(-1)}(x)$ with boundary $\set{-e_1 - k e_2: k \in \bZ}$ satisfy
	\begin{equation}
		X_{t, k} = L_{(-1)}(t, -k - 1) - L_{(-1)}(t, -k).
	\end{equation}
	for all $t \ge -1,\, k \in \bZ$. Continuing in this way, we may define $L_{(-s)}$ for all $s \ge 0$, and from this a consistent set of increments $\set{X_{t, k}: t \in \bZ,\, k \in \bZ}$, valid for all $t \in \bZ$.
	
	Next define horizontal increments
	\begin{equation}
		I_{k, t} = L_{(-s)}(k + 1, t) - L_{(-s)}(k, t),
	\end{equation}
	where $-s \le t$. Since the vertical increments are stationary and ergodic under vertical and horizontal shifts, it is easy to see the horizontal increments are also. Denote by $\nu$ the distribution of $I_{(0)} = (I_{k, 0})$. Then we take $\nu$ to be the image of $\mu$ in $M^{H}$ under our map. It is easy to see that we may go the other way: starting with horizontal increments and using the same weights, we may recover the distribution of the vertical increments.
	
	For the statement about means, it is enough to show that this map is well-defined. That is, two measures of $M^{V}$ with the same mean will be mapped to measures with the same mean. Injectivity will follow by reversing the roles of $V$ and $H$. Suppose $X \sim \mu \in M^{V}$ and $\widetilde{X} \sim \widetilde{\mu} \in M^{V}$, with $\Ex{X_0} = \Ex{\widetilde{X}_0} = \alpha$. Let $\nu,\, \widetilde{\nu} \in M^H$ be the corresponding horizontal increment distributions, with means $\beta,\, \widetilde{\beta}$. Let $L^X,\, L^{\widetilde{X}}$ be passage times with $X,\, \widetilde{X}$ as boundary increments. Stationarity of the increments gives us the law of large numbers for these passage times:
	\begin{equation}
		\label{eq:BoundaryLLN}
		\lim_{n \to \infty}\frac{1}{n} L^X(n,n) = \beta - \alpha, \qquad \lim_{n \to \infty}\frac{1}{n} L^{\widetilde{X}}(n,n) = \widetilde{\beta} - \alpha.
	\end{equation}
	On the other hand, if $\ell$ is the time constant for the passage times without boundary, then 
	\begin{equation}
		\lim_{n \to \infty}\frac{1}{n} L^X(n,n) = \min_{0 \le t \le 1} (\ell(1, t) + (1 - t)\alpha),
	\end{equation}
	and the same expression for $\widetilde{X}$. Hence both quantities in \eqref{eq:BoundaryLLN} are equal, and we get $\beta = \widetilde{\beta}$. For another example of this computation, see \cite[Section 3]{sepp-cgm-18}.
\end{proof}

\subsubsection{Uniqueness in SWFPP} In this section, $Y \in (\bR_{\ge 0})^{\bZ}$ is a sequence of store quantities and $W \in (\bR_{\ge 0})^{\bZ}$ is the service. We write $Y'$ for $A(Y, W)$, the result of applying the service. We begin with some easy lemmas.

\begin{lemma}
	\label{lem:AMass}
	With the notation above and $l < k \in \bZ$:
	\begin{equation}
		\sum_{i = l}^{k}Y'_i =  \sum_{i = l}^{k}Y_i + (Y_{l - 1} \wedge W_{l - 1}) - (Y_{k} \wedge W_{k}).
	\end{equation}
	\begin{proof}
		This follows from observing $(Y_i - W_i)^+ + (Y_{i} \wedge W_{i}) = Y_i$ and summing. 
	\end{proof}
\end{lemma}

\begin{lemma}
	\label{lem:AMono}
	If $\widetilde{Y}$ is another sequence of quantities such that $Y \ge \widetilde{Y}$ component-wise, then $Y' \ge \widetilde{Y}'$ (where we apply the same service $W$ to both).
	\begin{proof}
		This is just the statement of \cref{lem:UpdateMapMono} adapted to our setting.
	\end{proof}
\end{lemma}

\begin{lemma}
	\label{lem:AStab}
	Suppose $W$ consists of i.i.d random variables and $Y \sim \mu$, where $\mu$ is stationary and ergodic under shifts, with finite mean. Then $\Ex{Y'_0} = \Ex{Y_0}$.
	\begin{proof}
		Being a function of $Y$ and $W$, $Y'$ is itself stationary and ergodic. With full probability then
		\begin{equation}
			\Ex{Y'_0} = \lim_{k \to \infty}\frac{1}{k}\sum_{i = 0}^{k}Y'_i.
		\end{equation}
		But by \cref{lem:AMass}, this can be rewritten as 
		\begin{equation}
			\Ex{Y'_0} = \lim_{k \to \infty}\frac{1}{k}\sum_{i = 0}^{k}Y_i + \lim_{k \to \infty}\frac{Y_{-1} \wedge W_{-1}}{k} - \lim_{k \to \infty}\frac{Y_{k} \wedge W_{k}}{k}.
		\end{equation}
		The first limit is $\Ex{Y_0}$ almost surely, by ergodicity, and the second is zero trivially. Rearranging then,
		\begin{equation}
			\Ex{Y_0} - \Ex{Y'_0} =  \lim_{k \to \infty}\frac{Y_{k} \wedge W_{k}}{k}.
		\end{equation}
		Thus the limit on the right exists almost surely and is deterministic. As the $W_k$ are i.i.d, we can easily extract a subsequence which converges to zero, forcing the full limit to be zero. Hence $\Ex{Y'_0} = \Ex{Y_0}$, as needed.
	\end{proof}
\end{lemma}

Now we proceed along a contraction argument, following \cite{chang-unique-queue}. We will show that ergodic fixed points are unique in the single class case, when $n = 1$. This extends to general $n$ along the lines of \cite{fan-sepp-joint-buss}. The central tool will be the $\overline{\rho}$ metric on stationary sequences. Given stationary measures $\nu, \mu$, define
\begin{equation}
	\label{eq:rhobar}
	\overline{\rho}(\nu, \mu) = \inf_{(X, Y) \in \cM} \Exabs{X_0 - Y_0},
\end{equation}
where $\cM$ consists of the jointly stationary couplings of $(X, Y)$, where $X \sim \nu$, $Y \sim \mu$. When $\eta,\,\mu$ are ergodic, this is equal to the infimum over jointly ergodic and stationary couplings \cite[Theorem 8.3.1]{gray-ergodic}.

Chang proceeds by showing that the update map is strictly contractive for the $\overline{\rho}$ metric, under the assumption that the weights $W$ are unbounded. We are unable to show this strict contractivity here, but we find a weaker statement sufficient for our purposes.

In the following, let $\mu'$ be the distribution of $Y'$ when $Y \sim \mu$. Write $\mu^{(N)}$ for the result of applying this procedure $N$ times (with fresh, independent realisations of the weights $W$).

\begin{proposition}
	\label{prop:AContract}
	Let $\nu, \mu$ be ergodic and stationary measures on $\bR_{\ge 0}^{\bZ}$ with finite mean.
	\begin{proplist}[label=(\roman*)]
		\item \label{prop:AContractWeak} We have $\overline{\rho}(\nu', \mu') \le \overline{\rho}(\nu, \mu)$.
		\item \label{prop:AContractStrict} Suppose also that $\nu,\, \mu$ are distinct distributions but with the same mean, and that the support of $\lambda$ contains, but is not equal to, $0$. Then there is $N \ge 1$ such that $\overline{\rho}(\nu^{(N)}, \mu^{(N)}) < \overline{\rho}(\nu, \mu)$.
	\end{proplist} 
	\begin{proof}
		Fix some jointly ergodic and stationary coupling of $(\nu, \mu)$ and take sample sequences $X,\, Y$. To prove (i), we need to show that $\Exabs{X'_0 - Y'_0} \le \Exabs{X_0 - Y_0}$. Set $Z_i = X_i \vee Y_i$ and $\widetilde{Z}_i =  X'_i \vee Y'_i$, and write $Z'$ for the output of $A(Z, W)$. By \cref{lem:AMono}, we see that $Z'_i \ge \widetilde{Z}_i$.
		
		Observe
		\begin{equation}
			\abs{X_0 - Y_0} = 2Z_0 - X_0 - Y_0.    
		\end{equation} 
		On the other side,
		\begin{equation}
			\abs{X'_0 - Y'_0} = 2\widetilde{Z}_0 - X'_0 - Y'_0 \le 2Z'_0 - X'_0 - Y'_0.    
		\end{equation}
		Now by \cref{lem:AStab}, we take expectations on both sides to find 
		\begin{equation}
			\label{eq:AContract1}
			\Exabs{X'_0 - Y'_0} \le 2\Exabs{Z'_0} - \Exabs{X'_0} - \Exabs{Y'_0} = 2\Exabs{Z_0} - \Exabs{X_0} - \Exabs{Y_0} = \Exabs{X_0 - Y_0}.    
		\end{equation} 
		
		For (ii), we would like to identify an event on which $\widetilde{Z}_0 < Z'_0$, which will in turn make the inequality in \eqref{eq:AContract1} strict. In words, we would like to demonstrate a situation in which taking the maximum of the two tandems before the update leaves us with more material than taking the maximum after updating them separately. 
		
		Following Chang, the ergodicity, equal means, and the distributions being distinct together imply that a crossing occurs with positive probability. That is, there is $k \ge 1$, $c > 0$ such that the event 
		\begin{equation}
			A = \set{X_0 < Y_0,\, X_1 = Y_1, \dots, X_{k - 1} = Y_{k - 1},\, X_k > Y_k,\, X_i,\, Y_i \le c, \text{ for } 0 \le i \le k}
		\end{equation}
		has $\Prob{A} > 0$. 
		
		For the sake of simplicity, assume $\Prob{W = 0} > 0$. Choose $d > 0$ such that $\Prob{W \ge d} > 0$ and set $q = \ceil{c / d}$. As we will be iterating the process, let $(W^r_i)_{i \in \bZ}$ be the weights for the $r$-th iteration, all i.i.d. Set $N = q (k + 1)$ and consider the event
		\begin{equation}
			B = \set{W^r_{-1} = W^r_{k} = 0, W^r_{i} \ge d \text{ for } 1 \le r \le N,\, 1 \le i \le k - 1}.
		\end{equation}
		Our assumptions ensure $\Prob{B} > 0$, and independence means that also $\Prob{A \cap B} > 0$. On $A \cap B$ and with initial amounts given by $X$, no material enters bin $0$ and no material leaves bin $k$ over $N$ iterations. All of the material is eventually moved to bin $k$, which is left with $\sum_{i = 0}^{k}X_i$. Similarly, starting from $Y$ leaves bin $k$ with $\sum_{i = 0}^{k}Y_i$, and from $Z$ with $\sum_{i = 0}^{k}Z_i$. The crossing means that 
		\begin{equation}
			\sum_{i = 0}^{k}Z_i > \max\brac[\large]{\sum_{i = 0}^{k}X_i, \sum_{i = 0}^{k}Y_i},
		\end{equation}
		which is what we wanted.
		
		When $\Prob{W = 0} = 0$, we add an additional parameter $\epsilon > 0$ to the crossing event and make it
		\begin{equation}
			A = \set{X_0 \le Y_0 - \epsilon,\, X_1 = Y_1, \dots, X_{k - 1} = Y_{k - 1},\, X_k \ge Y_k + \epsilon,\, X_i,\, Y_i \le c, \text{ for } 0 \le i \le k}.
		\end{equation}
		Then we set
		\begin{equation}
			B = \set{W^r_{-1},\, W^r_{k} < \epsilon / 3 N, W^r_{i} \ge d, \text{ for } 1 \le r \le N,\, 1 \le i \le k - 1}.
		\end{equation}
		The argument proceeds in the same way.
	\end{proof}
\end{proposition}

\begin{proof}[Proof of \cref{thm:UniqueFixedPointsSWFPP}]
	We prove only the case $n = 1$. If $\mu,\, \nu$ are two shift ergodic measures invariant under $A$ with the same mean, then in the above notation, $\mu^{(N)} = \mu$ for all $N$, and similarly for $\nu$. Thus
	\begin{equation}
		\overline{\rho}(\nu^{(N)}, \mu^{(N)}) = \overline{\rho}(\nu, \mu)
	\end{equation}
	for all $N$. But by \cref{prop:AContractStrict}, if they were distinct measures then this would eventually be strict inequality. Hence $\mu = \nu$. Generalising to $n \ge 2$ is straightforward, see \cite{fan-sepp-joint-buss}.
\end{proof}

We've actually proved \cref{thm:UniqueFixedPointsSWFPP} under weaker assumptions - we didn't need ergodicity under applications of $A$ here.

\subsubsection{Extension to SJR} 
The SJR model is a modest generalisation of SWFPP: instead of weights on horizontals and zeros on verticals, we allow the roles to swap independently on each vertex. It might then be unsurprising that the store and particle system pictures valid in SWFPP can be adapted to accommodate this generalisation. In the store picture, the correct thing to do is to allow both \emph{positive} and \emph{negative} material, one annihilating the other. The positive material behaves as before, but the negative material obeys different dynamics.\footnote{The generalisation of the particle picture is less natural: particles may have some random width and we allow some form of limited overtaking.}

As before, we consider a $\bZ$-indexed tandem of storage bins which push material forward, simultaneously at each time step. We take a sequence $Y \in (\bR)^{\bZ}$ of real-valued quantities, and real-valued weights $W \in (\bR)^{\bZ}$. That we allow the $Y$ to be negative reflects the failure of monotonicity in the SJR model (\cref{lem:SWFPPMono} need not hold in general). The signs of our weights reflect whether the weight lives on a horizontal or vertical edge. If $Y_i \ge 0$, the $i$-th bin contains positive material and we \emph{move} up to $W_i^+$ to the next bin. If $Y_i < 0$, then we have negative material, and $W_i^-$ tells us how much material to \emph{keep}, moving the remainder to the next bin. Positive and negative material in the same bin annihilate one another (i.e., we add the quantities). The update can be written succinctly as 
\begin{equation}
	\label{eq:RansRec}
	Y'_i = (Y^+_i - W_i^+)^+ - (Y_i^- \wedge W_i^-) + (Y^+_{i - 1} \wedge W^+_{i - 1}) - (Y_{i - 1}^- - W_{i - 1}^-)^+.
\end{equation}
Observe that if $Y \ge 0$, then this is the same as $A(Y, W^+)$ and the expression reduces to that for SWFPP.

The similarity of \eqref{eq:RansRec} to \eqref{eq:AIncrements} leads to immediate generalisations of \cref{lem:AMono,lem:AMass,lem:AStab}. Monotonicity and stability of the mean are as before, and the conservation of mass is adjusted as follows.

\begin{lemma}
	\label{lem:RansMass}
	With $l < k \in \bZ$:
	\begin{equation}
		\sum_{i = l}^{k}Y'_i = (Y^+_{l - 1} \wedge W^+_{l - 1}) - (Y_{l - 1}^- - W_{l - 1}^-)^+ - (Y^+_{k} \wedge W^+_{k}) + (Y_{k}^- - W_{k}^-)^+.
	\end{equation}
\end{lemma}

More importantly, it is apparent from the description that positive material and negative material are separately conserved, unless they meet in a single bin. In that case, the amount of each material decreases. This suggests that only one type of material may exist in stationarity.

We adjust the notation for the next proof. Consider a double-indexed collection of weights $\set{W_{i, t}}_{i,\, t \in \bZ^2}$, all i.i.d with distribution $\lambda$. At the time $t$ update we use $\set{W_{i, t}}_{i \in \bZ}$ as the weights. We write $Y_{i, 0}$ for the initial amount in bin $i$, and $Y_{i, t}$ for the amount after $t$ iterations.

\begin{proposition}
	\label{prop:RansStat}
	Let $(Y_{i, 0}) \sim \mu$, with $\mu$ a stationary, invariant distribution for $A$ which is suitably ergodic in the sense of \cref{thm:UniqueFixedPoints} and has $\Exabs{Y_{0, 0}} < \infty$. Then either $Y_{0, 0} \ge 0 \as$ or $Y_{0, 0} \le 0 \as$.
	
	\begin{proof}
		Consider $O_{i, t} = Y_{i, t}^- \wedge W_{i, t}^-$, the negative material output by the $i$-th bin at time $t$. The $O_{i, t}$ form a stationary sequence ergodic in both indices. Moreover, $O_{i, t} \le \abs{Y_{i, t}}$, and so $\Ex{O_{0, 0}} < \infty$. Thus by stationarity and ergodicity in time, we have almost surely that for all $i$,
		\begin{equation}
			\lim_{T \to \infty}\frac{1}{T}\sum_{t = 0}^{T - 1}O_{i, t} = \Ex{O_{0, 0}}.
		\end{equation}
		This is to say that the asymptotic rate of negative material leaving each bin, and hence entering each bin, is the same. This implies that only a vanishing proportion of negative material is annihilated at each time step.
		
		Toward a contradiction, suppose that under $\mu$, both $\Prob{Y_{0, 0} > 0}$ and $\Prob{Y_{0, 0} < 0}$ are positive. Now ergodicity and stationarity in the first index imply that there is $k \ge 1$, $c > 0$ such that the event 
		\begin{equation}
			A = \set{Y_{0, 0} < -c,\, Y_{1, 0} = \cdots Y_{k - 1, 0} = 0,\, Y_{k, 0} > c}
		\end{equation}  
		occurs with positive probability. For simplicity assume again that $\Prob{W_{0, 0} = 0} > 0$ and consider the event
		\begin{equation}
			B = \set{W_{i, t} = 0 \text{ for } -1 \le i \le k,\, 0 \le t \le k - 1}.
		\end{equation}
		Then $\Prob{A \cap B} > 0$, and on this event the negative material from bin $0$ is brought to bin $k - 1$ over $k$ time steps. It follows, making use of stationarity and invariance, that the event 
		\begin{equation}
			A' = \set{Y_{0, 0} < -c,\, Y_{1, 0} > c,\, W_{0, 0} = W_{0, 1} = 0}
		\end{equation} 
		happens with positive probability. On this event we will get $Y_{1, 1} = Y_{0, 0} + Y_{1, 0}$, and $\abs{Y_{1, 1}} < \abs{Y_{0, 0}} \vee \abs{Y_{1, 0}} - c$. Write $\alpha = \Prob{A'}$. We use now the assumption of joint ergodicity, which allows us to conclude that negative material is annihilated in bin $1$ at a rate of at least $c \alpha$. This contradicts the condition we derived in the first paragraph. 
		
		The case when $\Prob{W = 0} = 0$ can be handled as in $\cref{prop:AContract}$.
	\end{proof}
\end{proposition}

\begin{corollary}
	Suppose $\mu,\, \nu$ are invariant for $A(\cdot, W)$ and have equal means. Then $\mu = \nu$.
	\begin{proof}
		Suppose the common mean is non-negative. Then \cref{prop:RansStat} shows that they are supported on non-negative sequences. The remark following \eqref{eq:RansRec} shows that they are then invariant for $A(\cdot, W^+)$. By our uniqueness statement for SWFPP, $\mu = \nu$.
	\end{proof}
\end{corollary}

This is essentially \Cref{thm:UniqueFixedPointsSJR}.

\begin{remark}
	Nothing stops us from defining SWFPP and SJR in higher dimensions. For the former, we fix a direction and the edges in all other directions are constrained to have zero weight. For the latter we choose independently at each vertex the incoming edge to hold a nonzero weight. We should expect the invariant measures of the SJR to be those of an appropriate SWFPP. One direction is easy - an invariant measure for SWFPP is also invariant for SJR. Our proof in the other direction for two dimensions relied on a representation in terms of positive and negative material moving through stores, which doesn't appear to be available in general.
\end{remark}

\section{The multi-line distribution as a fixed point}
\label{sec:Multiline}
Throughout this section, the maps $H$, $V$, and $A$ are those corresponding to the SWFPP model (that is, the store model) discussed in \cref{ssec:UpdateMaps}.
\subsection{Single-class fixed points}
\label{ssec:FPMarginals}
Here we make a note of the single-class fixed points of our maps under the integrable weights. These all turn out to be product measures, and so we need only describe the marginals. We collect the results of \cite{mart-bergeom} and take various limits suggested there, and arrive at \cref{tab:SWFPPInvariantDists}.

First, we say that the distributions $\Ber(p)\Exp(\alpha)$ and $\Ber(q)\Exp(\beta)$ are \emph{compatible} if the parameters satisfy
\begin{equation}
	\label{eq:ExpCompat}
	\frac{\alpha p}{1 - p} = \frac{\beta q}{1 - q}.
\end{equation}
One consequence of compatibility is that $q \le p$ implies $\beta \ge \alpha$, and that for a fixed $p,\, \alpha$, there is only one choice of $q,\, \beta$ producing a distribution with a given mean. Their discrete counterparts $\Ber(p)\Geom_+(a)$ and $\Ber(q)\Geom_+(b)$ are compatible if 
\begin{equation}
	\label{eq:GeomCompat}
	\frac{a}{1 - a}\frac{p}{1 - p} = \frac{b}{1 - b}\frac{q}{1 - q}.
\end{equation}
For a fixed weight distribution, we can check that the one parameter families of distributions appearing in \cref{tab:SWFPPInvariantDists} are compatible for different values of $\rho$ and $c$. From this it is clear that the measures defined in \cref{ssec:ExactDists} by iterated application of the update maps do indeed have the claimed marginals.

\begin{table}[ht]
	\centering
	\small
	\renewcommand{\arraystretch}{1.4}
	\begin{tabular}{
			>{\arraybackslash}p{2.5cm}
			>{\arraybackslash}p{4.8cm}
			>{\arraybackslash}p{3.5cm}
			>{\arraybackslash}p{3.5cm}
		}
		\toprule
		$W$ & $H(\cdot, W)$ & $V(\cdot, W)$ & $A(\cdot, W)$ \\
		\midrule
		$\Ber(p)\Exp(\alpha)$ & $\Ber\brac[\big]{\frac{\alpha p}{\alpha + (1 - p)\rho}}\Exp(\alpha + \rho)$  & $\Ber\brac[\big]{\frac{\alpha}{\alpha + \rho}}\Exp(\rho)$   & $\Ber\brac[\big]{\frac{\alpha}{\alpha + (1 - p)\rho}}\Exp(\rho)$   \\
		$\Ber(p)\Geom_+(a)$ & $\Ber(q(p, a, c))\Geom_+(a(1 - c) + c)$  & $\Ber(r(p, a, c))\Geom_+(c)$   & $\Ber(s(p, a, c))\Geom_+(c)$   \\
		$\Ber(p)$ & $\Ber\brac[\big]{\frac{(1 - c) p}{1 - c p}}$  & $\Geom_0(c)$   & $\Ber(1 - c + c p)\Geom_+(c)$\\
		\bottomrule
	\end{tabular}
	\caption{When  $W$ is a sequence of i.i.d Bernoulli-exponential, Bernoulli-geometric, or Bernoulli variables, the invariant measures for the update maps are product measures with marginals given in the table. We get a valid distribution for any choice of $\rho \ge 0$ or $c \in [0, 1]$. In the second row, we take $q(p, a, c) = \frac{a (1 - c) p}{a(1 - c) + c(1 - p)}$, $r(p, a, c) = \frac{a(1 - c)}{a(1 - c) + c}$, and $s(p, a, c) = \frac{a(1 - c) + c p}{a(1 - c) + c}$.}
	\label{tab:SWFPPInvariantDists}
\end{table}

The following is part of Theorem 4.1 of \cite{mart-bergeom} and will be useful in making calculations.
\begin{lemma}
	Consider an i.i.d sequence of service $W$, whose marginals appear in the left column of \cref{tab:SWFPPInvariantDists}, and suppose we have inputs $I$ which follow the corresponding distribution in the second column. Then if we consider the store with inputs $I$ and service $W$ and look at the store size at time $0$ before and after input, but before service, then these quantities will follow the distribution in the third and fourth columns, respectively. 
	
	Similarly, if we have a sequences $X$ and $Y$ of store quantities for a tandem, before and after input, following the respective distributions in \cref{tab:SWFPPInvariantDists}, then the output from store $0$ will follow the distribution in the second column of the table.
\end{lemma}

This description of the single-class fixed points tells us in particular the distribution of the Busemann functions in a fixed direction.
\begin{proposition}
	Suppose the weight distribution $\lambda$ is among the first column of \cref{tab:SWFPPInvariantDists}. Then for each direction $\xi \in \scrU$, the collections
	\begin{equation}
		\set{B^\xi (k e_1, (k + 1)e_1) : k \in \bZ}, \quad \set{B^\xi (- k e_2, -(k + 1)e_2) : k \in \bZ},\quad  \set{B^\xi (k(e_1 - e_2), (k + 1)(e_1 - e_2)) : k \in \bZ}
	\end{equation}
	are distributed as sequences of independent random variables with marginals given by the second, third, and fourth columns of \cref{tab:SWFPPInvariantDists}, respectively, for some choice of parameters $\rho$ and $c$.
\end{proposition}

The link between the parameters $\rho$ and $c$ is in \cref{thm:BusemannExistenceDistinct}: if we know the gradients of the limit shape in a given direction, then we may equate it to the mean of the Busemann function and solve for the parameters. The limit shape may be determined explicitly in each of the integrable cases and we refer to \cite[Section 5]{rans-thesis} for a complete list. All of our explicit calculations will be performed where the expression is easiest, namely when the weights are $\Exp(1)$. Then $\ell(x, y) = (\sqrt{x + y} - \sqrt{y})^2$, and if $\rho > 0$, the corresponding direction is
\begin{equation}
	\xi(\rho) = \brac[\bigg]{1 - \frac{\rho^2}{(1 + \rho)^2}, \frac{\rho^2}{(1 + \rho)^2}}.
\end{equation}

From this, finding the distribution of the critical angle $\xi^*$ requires no great leap. Recall that the critical angle is the asymptotic direction of the competition interface, the angle below which taking an initial $e_1$ step is always favourable. We observed in \cref{prop:CritFromBuse} that $\xi^* = \sup \set{\xi : B^\xi (0, e_2) = 0}$. If $\xi^* = \xi(\rho^*)$, then this becomes
\begin{equation}
	\rho^* = \inf \set{\rho : B^{\xi(\rho)} (0, e_2) = 0}.
\end{equation}
The distribution of $B^{\xi(\rho)} (0, e_2)$ can be seen from \cref{tab:SWFPPInvariantDists} to be $- \Ber((1 + \rho)^{-1})\Exp(\rho)$, so
\begin{equation}
	\Prob{\rho^* \le \rho} = \Prob{B^{\xi(\rho)} (0, e_2) = 0} = \frac{\rho}{1 + \rho}.
\end{equation}
Some calculation shows that in contrast to LPP, the distribution of the critical angle is not symmetric about $(1/2, 1/2)$. Indeed, here we have $\Ex{\xi^*} = (1/3, 2/3)$.

The remainder of this section will largely be interested in going further to identify multi-class (or vector-valued) fixed points. Recall the distributions $\mu_{H}^{\rho}$, $\mu_{A}^{\rho}$ and $\mu_{V}^{\rho}$ defined towards the beginning, in \cref{ssec:ExactDists}.

\begin{theorem}
	\label{thm:MultiStat}
	Take $U \in \set{H, A, V}$, and let $\rho = (\rho_1, \dots, \rho_n)$ be a vector of positive means with $\rho_1 > \cdots > \rho_n > 0$ (and $\rho_1 < p$ when $U = H$). Suppose $W \sim \nu$ and $(J^1, \dots, J^n) \sim \mu_{U}^{\rho}$. Then 
	\begin{enumerate}
		\item The marginal distributions are $J^k \sim \nu_{U}^{\rho_k}$.
		\item We have component-wise monotonicity: $J^1 \ge J^2 \ge \cdots \ge J^n$.
		\item The distribution $\mu_{U}^{\rho}$ is jointly invariant under $U$:
		\begin{equation}
			\label{eq:MultiStat}
			(U(J^1, W), \dots, U(J^n, W)) \disteq (J^1, \dots, J^n).
		\end{equation}
	\end{enumerate}
\end{theorem}

\subsection{Reversibility and reverse weights}
\label{ssec:Reversibility}
The key observation made by the authors of \cite{mart-prab-fixed} is the central role of reversibility in analysing the integrable cases of this model. In what follows, we let $\overleftarrow{I} = (I_{-k})_{k \in \bZ}$ the reversal of a sequence $I$, and $\overleftarrow{U}(I, W) = \overleftarrow{\brac[\big]{U(\overleftarrow{I}, \overleftarrow{W})}}$ be the result of running an update map $U$ ``in reverse''.

\begin{proposition}[Theorem 7.1 of \cite{mart-prab-fixed}]
	Let $I^1$, $I^2$ be sequences whose distributions belong to one of the families in the second column of \cref{tab:SWFPPInvariantDists}, and assume that $\Ex{I^1_0} \ge \Ex{I^2_0}$. Let $O^2 = H(I^2, I^1)$. Then:
	\begin{proplist}
		\item The input and output of the store are reversible in the sense that $(I^2, O^2) \disteq (\overleftarrow{O}^2, \overleftarrow{I}^2)$.
		\item There is a sequence $O^1$ living on the same probability space such that $(I^1, I^2, O^2) \disteq (\overleftarrow{O}^1, \overleftarrow{I}^2, \overleftarrow{O}^2)$, and in particular $I^1 = \overleftarrow{H}(O^1, O^2)$.
		\item For any sequence $I^3$ for which the maps makes sense, we have a deterministic equality
		\begin{equation}
			\label{eq:HIntertwining}
			H(H(I^3, I^2), I^1) = H(H(I^3, O^1), O^2).
		\end{equation}
	\end{proplist}
\end{proposition}

We would like to find a version of this result for the maps $A$ and $V$. From this, the arguments of \cite{mart-prab-fixed} can be recycled to prove \cref{thm:MultiStat}. 

Our arguments rely on a local form of reversibility. To set this up, let $\Gamma$ be a boundary of the sort considered in \cref{sec:GeneralUpdates}, containing the origin. Suppose $L^\Gamma(e_2) = 0$, and define $X$ and $I$ by $L^\Gamma(0) = X,\, L^\Gamma(e_1) = I + X$. Recall that SWFPP passage times decrease as the end point moves up and increase when it move right, so $I \ge 0$. Let $W = \weight{e_1 + e_2}$. Then $L^\Gamma(e_1 + e_2) = (I + X) \wedge W$. This is the local rule for computing passage times on a square, from which we get updated increments
\begin{equation}
	\label{eq:SquareIncs}
	I' = L^\Gamma(e_1 + e_2) - L^\Gamma(e_2) = (I + X) \wedge W,\quad X' = L^\Gamma(e_1) - L^\Gamma(e_1 + e_2)= (I + X - W)^{+}.
\end{equation}
This operation is illustrated in \cref{fig:LocalUpdates}.

\begin{figure}
	\centering
	\begin{subfigure}[]{0.9\textwidth}
		
		\resizebox{0.4\textwidth}{!}{\includegraphics[scale=1, page=2]{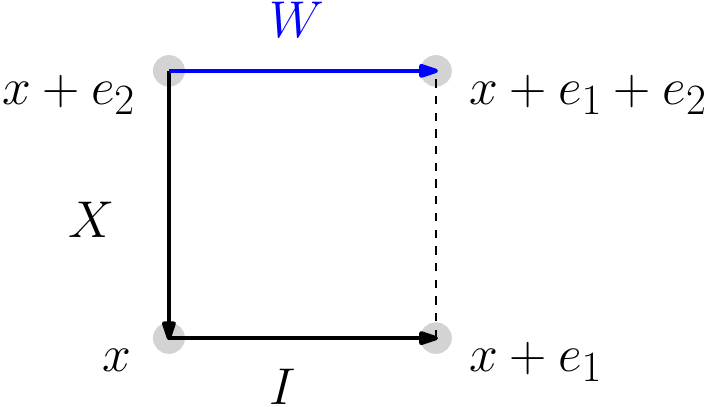}}
		\hspace{2cm}
		\resizebox{0.4\textwidth}{!}{\includegraphics[scale=1, page=3]{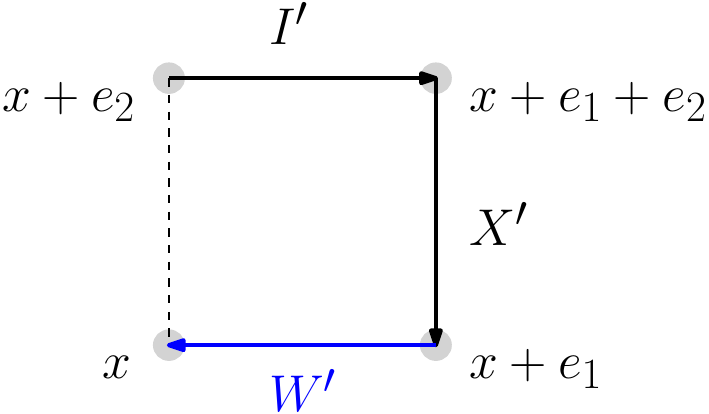}}
		\subcaption{The notional placements of the local variables on the square.}
	\end{subfigure}
	\centering
	\begin{subfigure}[]{0.9\textwidth}
		\vspace{1cm}
		\resizebox{0.4\textwidth}{!}{\includegraphics[scale=1, page=4]{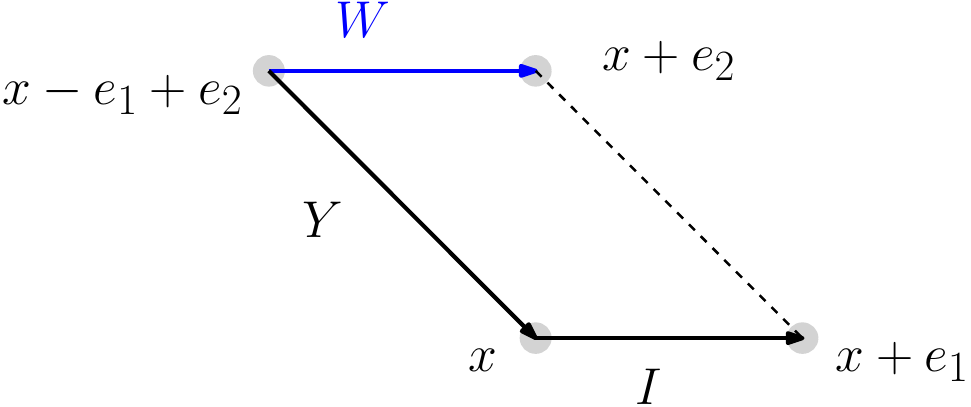}}
		\hspace{1cm}
		\resizebox{0.4\textwidth}{!}{\includegraphics[scale=1, page=5]{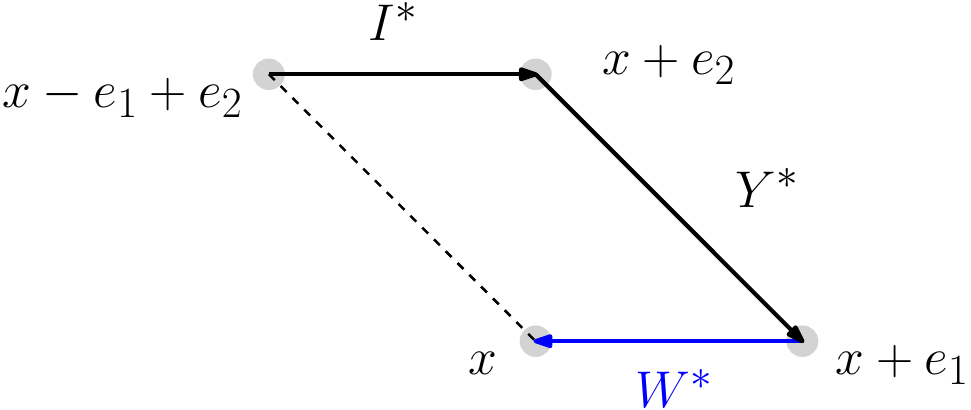}}
		\subcaption{The notional placements of the local variables on the diamond.}
	\end{subfigure}
	\caption{}
	\label{fig:LocalUpdates}
\end{figure}

We may also look at a diamond. Now let $L^\Gamma(e_2 - e_1) = 0$, and $L^\Gamma(0) = Y,\, L^\Gamma(e_1) = I + Y$, and set $W = \weight{e_2}$. The passage time to $e_2$ is $L^\Gamma(e_2) = Y \wedge W$, and we have increments
\begin{equation}
	\label{eq:DiamondIncs}
	I^* = L^\Gamma(e_2) - L^\Gamma(e_2 - e_1) = Y \wedge W,\quad Y^* = L^\Gamma(e_1) - L^\Gamma(e_2)= I + (Y - W)^{+}.
\end{equation}

\begin{proposition}
	Let $W,\, I,\, X,\, Y$ be independent (scalar) variables with distributions in the first, second, third and fourth columns, respectively, of \cref{tab:SWFPPInvariantDists} (in the same row and with the same choice of parameters). Then $(I, X, I', X') \disteq (I', X', I, X)$ and $(I, Y, I^*, Y^*) \disteq (I^*, Y^*, I, Y)$. 
	\begin{proof}
		For each family of distributions, one may explicitly (and laboriously) compute the Laplace transform
		\begin{equation}
			\phi(s, t, s', t') = \Ex{e^{s I + t X + s' ((I + X) \wedge W) + t' (I + X - W)^+}}
		\end{equation}
		and verify that $\phi(s, t, s', t') = \phi(s', t', s, t)$. Similarly for the other set of increments.
	\end{proof}
\end{proposition} 

An immediate consequence is that there is a $W'$ such that $(I', X', W') \disteq (I, X, W)$ and
\begin{equation}
	\label{eq:SquareRecovery}
	I = (I' + X') \wedge W',\quad X = (I' + X' - W')^+.
\end{equation}
This is again true of the increments around the diamond, with some $W^*$. We call these \emph{reverse weights}, and they feature prominently in, for example, \cite{ferr-mart-dual, fan-sepp-joint-buss, sepp-soren-blpp-buse}. 

As it will sometimes be necessary to be precise, we make a fixed choice of function $f' : \bR^4 \to \bR$ for each choice of distribution on the other variables, such that $W' = f'(I, X, W, \omega)$. Here $\omega \in \bR$ is an auxiliary variable providing an additional source of randomness, always assumed to be independent of everything else. For the reverse weights on the diamond we may choose some $f^*$ serving the same function. In certain cases we can write down these functions. For example, when $W \sim \Ber(p)$ we can take $f'(I, X, W) = (W - (I + X))^+$, and when $W \sim \Exp(1)$ we can take $f^*(I, Y, W) = I + (W - Y)^+$. For an example where the additional randomness is needed, see \cite{ciech-geor-strict-strict}. In general, we do not expect to have a neat expression.

\begin{proposition}
	\label{prop:AVReversibility}
	Let $W,\, X,\, Y$ be i.i.d sequences following distributions from the first, third and fourth columns of \cref{tab:SWFPPInvariantDists} and consider $X' = V(X, W)$, $Y^* = A(Y, W)$. Then there is a sequence $W'$ living on the same probability space such that $(W, X, X') \disteq (\overleftarrow{W}', \overleftarrow{X}', \overleftarrow{X})$, and a sequence $W^*$ such that $(W, Y, Y^*) \disteq (\overleftarrow{W}^*, \overleftarrow{Y}^*, \overleftarrow{Y})$. Hence 
	\begin{equation}
		X = \overleftarrow{V}(X', W') \qquad \text{and} \qquad Y = \overleftarrow{A}(Y^*, W^*).
	\end{equation}
	\begin{proof}
		We prove this just for the $V$ map, as this is the harder case. Consider the entries of $X$ as the increments of passage times along the line $\Gamma = \set{k e_2 : k \in \bZ}$, so that $L^\Gamma(k e_2) = - \sum_{j = 1}^{k} X_k$ for $k \ge 0$, and $L^\Gamma(- k e_2) = \sum_{j = 0}^{k - 1} X_{-j}$. The sequence $X'$ consists of the increments $X'_k = L^\Gamma(e_1 + (k - 1) e_2) - L^\Gamma(e_1 + k e_2)$, and we may define $I_k = L^\Gamma(e_1 + (k - 1) e_2) - L^\Gamma((k - 1)e_2)$. Their realisation as passage times means that
		\begin{equation}
			I_{k + 1} = (I_k + X_k) \wedge W_k, \qquad X'_{k} = (I_k + X_k - W_k)^+.
		\end{equation}
		Then from our discussion above, we may find a $W'_k$ such that 
		\begin{equation}
			\label{eq:SquareDistEq}
			(I_{k + 1}, X'_k, W'_k) \disteq (I_k, X_k, W_k),
		\end{equation}
		with $W'_k$ independent of the rest of our variables, conditional on $I_k$ and $X_k$.
		
		To show that this equality holds for the joint distributions of the entire sequences, we must show that $X'$ and $W'$ are independent as sequences, and that $X = \overleftarrow{V}(X', W')$. The independence will follow if we show that $S_1 = \set{I_j : j \ge k}$ is independent of $S_2 = \set{(X'_j, W'_j) : l \le j \le k - 1}$, for any pair of indices $l < k$. To see why this is enough, observe that the only dependence between $\set{(X'_j, W'_j) : j \ge k - 1}$ and $\set{(X'_j, W'_j) : j \ge k}$ is through $I_k$, since otherwise the latter is a function of fresh, independent variables. 
		
		Now, it is clear that $S_1$ and $S_2$ are independent conditional on $I_k$, again since the elements of the former set are functions of $I_k$ and new variables. So we need only show that $S_2$ is independent of $I_k$. Consider starting from $I_{l}$ and $\set{(X_j, W_j) : l \le j \le k - 1}$, all independent. By \eqref{eq:SquareDistEq}, The pair $(X'_l, W'_l)$ is independent of $I_{l + 1}$, and thus of $I_j$ for all $j \ge l + 1$. The next pair $(X'_{l + 1}, W'_{l + 1})$  is in turn independent of $(X'_l, W'_l)$ and $I_j$, $j \ge l + 2$. We continue and find that all of $S_2$ is independent of $I_k$.
		
		Now we show that the original sequence $X$ can be recovered as $X = \overleftarrow{V}(X', W')$, or equivalently that $\overleftarrow{X} = V(\overleftarrow{X}', \overleftarrow{W}')$. Consider the ``reversed'' situation, where we have increments $X'$ along the line $\Gamma' = \set{e_1 + k e_2}$ and weights $W'$, and look at down-left geodesics rather than up-right geodesics. Let $L^{\Gamma'}$ be the passage times here, normalised so that $L^{\Gamma'}(e_1) = 0$. Let $\widetilde{I}_k = L^{\Gamma'}(k e_2) - L^{\Gamma'}(e_1 + k e_2)$ be the sequence of horizontal increments (note the shift relative to the definition of $I_k$). These variables are related by
		\begin{equation}
			\label{eq:ReverseOutputEvolution}
			\widetilde{I}_{k - 1} = (\widetilde{I}_{k} + X'_k) \wedge W'_k.
		\end{equation}
		
		To show that $\overleftarrow{X} = V(\overleftarrow{X}', \overleftarrow{W}')$ in this picture is just to show that the increments of the reverse passage times $L^{\Gamma'}$ along our original line $\Gamma = \set{k e_2 : k \in \bZ}$ are the sequence $X$ we started with. If this is the case, the horizontal increments $\widetilde{I}$ will also equal the original $I$, up to a shift in indexing. To show this, let $k$ be a time such that $I_k \ge W_k$ and $X_k > 0$, so that the geodesic from $\Gamma$ to $e_1 + k e_2$ under the ``forward'' passage times $L^\Gamma$ has its horizontal step at height $k$. In the store picture, this says that store $k$ initially has material, and that the amount of material increases after the update. 
		
		Now look at \eqref{eq:SquareIncs} and \eqref{eq:SquareRecovery} to see what can be said of $W'_k$. We know that $I_k + X_k \ge W_k$, so $I_{k + 1} = W_k$ and $X'_k = (I_k + X_k - W_k)$. They satisfy $X_k = (I_{k + 1} + X'_k - W'_k)^+$, which combined with $X_k > 0$, implies that $W'_k = I_k$. Moreover, $X'_k \ge W'_k$. Feeding this in to \eqref{eq:ReverseOutputEvolution} gives the equality $\widetilde{I}_{k - 1} = W'_k = I_k$. So we have that the horizontal increments are equal at height $k$. From \eqref{eq:SquareRecovery}, we know that we can recover $I_{k - 1}$ given $I_k$, $X'_{k - 1}$ and $W'_{k - 1}$, and comparing to \eqref{eq:ReverseOutputEvolution} (after a shift of index) we see that in fact $I_{k - 1} = \widetilde{I}_{k - 2}$. Continuing in this way, we have $I_j = \widetilde{I}_{j - 1}$ once $j \le k$. But our assumptions on $k$ hold infinitely often by ergodicity, so in fact $I_j = \widetilde{I}_{j - 1}$ for all $j$.
		
		To conclude, we use the local update rule from \eqref{eq:SquareIncs} and get
		\begin{align}
			L^{\Gamma'}((k - 1)e_2) - L^{\Gamma'}(k e_2) &= (\widetilde{I}_k + X'_k - W'_k)^+\\
			&= (I_{k + 1} + X'_k - W'_k)^+\\
			&= X_k,
		\end{align}
		the last part from \eqref{eq:SquareRecovery}, again. Thus the $X_k$ are the vertical increments in the reversed picture (with the indexing reversed), which is to say that $\overleftarrow{X} = V(\overleftarrow{X}', \overleftarrow{W}')$.
	\end{proof}
\end{proposition}

Although we don't have an expression for $W'$ and $W^*$ in general, we can read off some properties from this reversibility.
\begin{corollary}
	\label{cor:AVReverseProps}
	In the setting of \cref{prop:AVReversibility}, we have
	\begin{corlist}
		\item $W'_k \ge (X_k + I_k) \wedge W_k$, with equality if $X_{k - 1} > 0$.
		\item $W^*_k \ge Y_k \wedge W_k$, with equality if $W_{k - 1} \le Y_{k - 1}$.
	\end{corlist} 
\end{corollary}

\subsection{Intertwining and multi-class fixed points}
\label{ssec:Intertwining}
This section will be spent showing counterparts of \eqref{eq:HIntertwining} for $A$ and $V$. Following \cite{fan-sepp-joint-buss}, we call these ``intertwining identities''. From these identities we show how to quickly derive \cref{thm:MultiStat}.

\begin{proposition}
	\label{prop:VIntertwining}
	Let $X^1$, $X^2,\, X^3$ be independent sequences whose distributions belong to one of the families in the third column of \cref{tab:SWFPPInvariantDists} and assume $\Ex{X^1_0} \ge \Ex{X^2_0} > \Ex{X^3_0}$. Let $Z^2 = V(X^2, X^1)$ and $Z^1$ the corresponding reverse weights, discussed above. Almost surely, we have
	\begin{equation}
		\label{eq:VIntertwining}
		V(H(X^3, X^2), X^1) = H(V(X^3, Z^1), Z^2).
	\end{equation}
\end{proposition}

\begin{remark}
	The equality \eqref{eq:VIntertwining} is essentially deterministic and its proof requires no probabilistic reasoning. We phrase it as an almost sure result in terms of random variables to avoid enumerating the various regularity assumptions we would need on the sequences. 
\end{remark}

The identity for $A$ is the same: when $Y^1$, $Y^2$ have distributions in the fourth column of \cref{tab:SWFPPInvariantDists}, and $W^2 = A(Y^2, Y^1)$ with $W^1$ the reverse weights:
\begin{equation}
	\label{eq:AIntertwining}
	A(H(Y^3, Y^2), Y^1) = H(A(Y^3, W^1), W^2).
\end{equation}

The following series of lemmas show a particular entry in the output of our maps depends only on finitely many of the inputs. Throughout, we take some integers $a \le b$ and let $S^{(a, b)}$ be the truncation of a sequence $S$ defined by $S^{(a, b)}_k = \bone (-a \le k \le b) S_k$. 
\begin{lemma}
	\label{lem:VTrunc}
	Let $X,\, W$ be i.i.d sequences of non-negative random variables, such that the support of the $W_k$ contains $0$, and set $Z = V(X, W)$. For each $n$, there is almost surely $a$ and $b$ large enough such that the sequence $\widetilde{Z} = V(X, W^{(a, b)})$ has $\widetilde{Z}_k = Z_k$ for $-n \le k \le n$.
	\begin{proof}
		Here it is helpful to imagine the corresponding particle system, as described in \cref{ssec:UpdateMaps}. The $X_k$ represent the distance between particles $k$ and $k + 1$, and the $W_k$ are the distances the particles intend to jump. Let $b = \min\set{k \ge n : W_k \le X_k} - 1$, which is finite by our assumptions. Regardless of the jumps performed ahead of it, there is no way for particle $b + 1$ to interact with particle $b + 2$, as its jump is too short. So we may replace the jumps ahead of particle $b$ by $0$ without affecting the dynamics of particle $b + 1$ and all those behind it, and thus the inter-particle distances are unaffected from particle $b$, downward. Also, the trajectory of a particle is not affected by those behind it, so we may take $a = n$.
	\end{proof}
\end{lemma}

\begin{lemma}
	\label{lem:VReplace}
	Let $W$ and $X$ be as before, and $W'$ be an arbitrary sequence. Set $\widetilde{W} = W' - (W')^{(a, b)} + W^{(a, b)}$, which is $W'$ with an interval of values replaced by those of $W$, and $Z = V(X, W)$, $\widetilde{Z} = V(X, \widetilde{W})$. For each $n$, there is almost surely $a$ and $b$ large enough such that $\widetilde{Z}_k = Z_k$ for all $-n \le k \le n$.
	\begin{proof}
		This is a direct consequence of the previous lemma. The truncations of $W$ and $\widetilde{W}$ agree, and we have seen that a sufficiently large truncation determines the values of $Z_k$, $\widetilde{Z}_k$, for $-n \le k \le n$.
	\end{proof}
\end{lemma}

\begin{lemma}
	\label{lem:VReverseTrunc}
	Let $W$ follow one of the distributions in the first column of \cref{tab:SWFPPInvariantDists}\footnote{Or indeed, from any of the others. The distributions of the other columns also appear in the first, with rescaling and the correct choice of parameters.} and let $X$ follow a corresponding distribution from the third. Let $W'$ be reverse weights for $V(X, W)$, and $\widetilde{W}$ reverse weights for $V(X, W^{(a, b)})$\footnote{We use the same function $f'$ to produce $\widetilde{W}$}. There is a coupling of the two, such that for each $n$, there is almost surely $a$ and $b$ large enough such that $\widetilde{W}_k = W'_k$ for all $-n \le k \le n$.
	\begin{proof}
		Recall that for each valid choice of the distribution of $(X, W)$, we had a function $f'$ giving the reverse weights which may take some independent source of randomness $\omega$ as an independent argument. We may write $W'_k = f'(X_k, W_k, I_k, \omega_k)$, where $I_k$ are the horizontal increments we discussed in the proof of \cref{prop:AVReversibility}. In the store picture, they are the inputs to each queue before service. Coupling them by using the same external randomness, we may take $\widetilde{W}_k = f'(X_k, \bone(-a \le k \le b)W_k, \widetilde{I}_k, \omega_k)$. (Here $\widetilde{I}_k$ are the inputs of the stores when we use the truncated service.) Clearly these will be equal if $-a \le k \le b$ and $\widetilde{I}_k = I_k$. For the first condition, we should just make sure that $a,\, b \ge n$.
		
		To show $\widetilde{I}_k = I_k$, we go back to our particle picture. Each $I_k$ is just the distance moved by particle $k + 1$. As we argued in \cref{lem:VTrunc}, taking our $a$ and $b$ large enough leaves the dynamics of particle $-(n + 1), \dots, n + 1$ unchanged after replacing the jumps for the other particles by $0$. In particular, the distances moved are unchanged, and thus $\widetilde{I}_k = I_k$.
	\end{proof}
\end{lemma}

We collect similar statements to the above, but now for the $H$ map.
\begin{lemma}
	\label{lem:HInputTrunc}
	Let $I,\, W$ be i.i.d sequences of non-negative random variables such that $\Ex{I_0} < \Ex{W_0} < \infty$, or such that $\Ex{I_0} = \Ex{W_0}$ and both have finite variance. Let $O = H(I, W)$ and $\widetilde{O} = H(I^{(a, b)}, W)$. For each $n \ge 0$, there is almost surely $a$ and $b$ large enough so that $O_k = \widetilde{O}_k$ for all $-n \le k \le n$.
	\begin{proof}
		Consider the store with services $W$ and inputs $I$. Our assumptions on the means and variance say that the store empties infinitely often (this is equivalent to the random walk with steps $I_k - W_k$ having infinitely many running minimums). Choose $a \ge n$ such that $-(a + 1)$ is such a time. Set $b = n$. Both the store with input $I$ and input $I^{(a, b)}$ are initially empty at time $-a$ (the latter for trivial reasons). Thereafter, the stores receive the same input and service, up until time $b$. They will thus have the same output, which is what we claimed.
	\end{proof}
\end{lemma}

The statement below is then immediate from \cref{lem:HInputTrunc}:
\begin{lemma}
	\label{lem:HInputReplace}
	Let $I$ and $W$ be as before, and $I'$ another i.i.d sequence satisfying the same assumptions as $I$. Set $\widetilde{I} = I' - (I')^{(a, b)} + I^{(a, b)}$, and $O = H(I, W)$, $\widetilde{O} = H(\widetilde{I}, W)$. For each $n$, there is almost surely $a$ and $b$ large enough such that $\widetilde{O}_k = O_k$ for all $-n \le k \le n$.
\end{lemma}

\begin{lemma}
	\label{lem:HServiceReplace}
	Let $I,\, W,\, \widehat{W}$ be i.i.d sequences of non-negative random variables, all with finite mean, such that either $\Ex{I_0} < \Ex{W_0} \wedge \Ex{\widehat{W}_0}$, or such that $\Ex{I_0} = \Ex{W_0} \wedge \Ex{\widehat{W}_0}$ and all have finite variance. Set $\widetilde{W} = \widehat{W} - \widehat{W}^{(a, b)} + W^{(a, b)}$, and Let $O = H(I, W)$ and $\widetilde{O} = H(I, \widetilde{W})$. For each $n \ge 0$, there is almost surely $a$ and $b$ large enough so that $O_k = \widetilde{O}_k$ for all $-n \le k \le n$.
	\begin{proof}
		Set $b = n$. Consider the stores corresponding to $H(I, W)$ and $H(I, \widehat{W})$, which we call the first and second store, respectively. These stores will be empty infinitely often, so find $a$ and $a_0$ such that $a \ge a_0 \ge -n$, and such that the first store is initially empty at time $-a_0$, and the second store is initially empty at time $-a$.  A third store, with service $\widetilde{W}$, is also initially empty at time $-a$, while at this time the first store may not be. Thereafter, the first and third stores receive the same input and service, and so the first dominates the third. At time $-a_0$, both stores are empty. From this point on, both stores are in the same state and produce the same output.
	\end{proof}
\end{lemma}

We put these lemmas together into the form which we will actually need.
\begin{lemma}
	\label{lem:IntertwiningTrunc}
	Let $X^1,\, X^2,\, X^3$ be as in \cref{prop:AVReversibility}, with $Z^2 = V(X^2, X^1)$ and $Z^1$ the corresponding reverse weights. Let $\widetilde{X}^1$ be the truncation $(X^1)^{(a, b)}$. Set $\widetilde{Z}^2 = V(X^2, \widetilde{X}^1)$ and let $\widetilde{Z}^1$ be the reverse weights for this update with truncated service. For all $n \ge 0$, almost surely we may choose $a,\, b$ large enough so that 
	\begin{equation}
		\label{eq:VCutoff}
		(V(H(X^3, X^2), X^1))_k = (V(H(X^3, X^2), \widetilde{X}^1))_k
	\end{equation}
	and 
	\begin{equation}
		\label{eq:HCutoff}
		(H(V(X^3, Z^1), Z^2))_k = (H(V(X^3, \widetilde{Z}^1), \widetilde{Z}^2))_k
	\end{equation}
	for all $-n \le k \le n$.
	\begin{proof}
		The equality in \eqref{eq:VCutoff} follows just from \cref{lem:VTrunc}. For \eqref{eq:HCutoff}, we first use \cref{lem:HServiceReplace} to replace $Z^2$ by $\widetilde{Z}^2$. Then \cref{lem:VReverseTrunc} means that $Z^1$ and $\widetilde{Z}^1$ agree on an arbitrarily long interval, and so \cref{lem:VReplace} says that also $V(X^3, Z^1)$ and $V(X^3, \widetilde{Z}^1)$ agree on an arbitrarily long interval. Now \cref{lem:HInputReplace} gives the result.
	\end{proof}
\end{lemma}

The takeaway from \cref{lem:IntertwiningTrunc} is that to verify \cref{prop:AVReversibility}, it suffices to check finite truncations of the sequence $X^1$. Before carrying this out, let us introduce some language which will help visualise what \eqref{eq:VIntertwining} is saying and better organise the argument. Switch to the particle system picture and let $\set{(\eta^2(k), 1)}_{k \in \bZ}$, $\set{(\eta^3(k), 0)}_{k \in \bZ}$ be the positions of particles on vertically shifted copies of the real line, with $\eta^i(k + 1) - \eta^i(k) = X^i_k$ and translated such that $\eta^2(0) = \eta^3(0) = 0$.  Consider paths starting at $(-\infty, 0)$ on the lower line and ending at $(\eta^i(k), 3 - i)$. We may jump rightwards along the lower line from one particle  to the next. If our destination is on the upper line, we have the option while at $\eta^3(j)$ to jump for free to $\eta^2(j)$ on the upper line, and then continue to jump on the upper line towards our destination. 

If $\gamma$, $\gamma'$ are two such paths, possibly with different endpoints, it makes sense to talk about the difference $\abs{\gamma} - \abs{\gamma'}$ as
\begin{equation}
	\abs{\gamma} - \abs{\gamma'} = \abs{\gamma \setminus \gamma'} - \abs{\gamma' \setminus \gamma}.
\end{equation}
Here $\gamma \setminus \gamma'$ is the segment of $\gamma$ not shared with $\gamma'$. 

Let $\gamma_0$ be the unique path to $(\eta^3(0), 0)$, and define the ``length'' of $\gamma$ as $L(\gamma) = \abs{\gamma} - \abs{\gamma_0}$. Let $\gamma(k)$ be the path to $(\eta^2(k), 1)$ with minimal length and define $\zeta(k) = L(\gamma(k))$. In light of the fact that $H$ represents the horizontal update in SWFPP, it is not difficult to see that in fact $(\zeta(k + 1) - \zeta(k)) = H(X^3, X^2)$. Observe that 
\begin{equation}
	\label{eq:zetaRecur}
	\zeta(k) = (\zeta(k - 1) + X^2_{k - 1}) \wedge \eta^3(k).
\end{equation}
Also, $\zeta(k - 1) \le \eta^3(k - 1) \le \eta^3(k)$, and so if $X^2_{k - 1} = 0$, we must have $\zeta(k) = \zeta(k - 1)$.

The purpose of the long-winded interpretation of $H$ in the preceding paragraph is to provide some meaning to the expressions $V(H(X^3, X^2), X^1)$ and $H(V(X^3, Z^1), Z^2)$. The former represents computing the positions of these $\zeta(k)$, then jumping them forward by $X^1$ in the sense described in \cref{sssec:UpdateMapsV}. Let $\zeta'(k)$ be the positions of the particles in this instance. The latter expression represents first jumping the $\eta^2(k)$ forward by $X^1$ to produce $\tilde{\eta}^2(k)$, then jumping the $\eta^3(k)$ with some compensatory jumps to produce $\tilde{\eta}^3(k)$, and finally using these moved particles with the scheme described above to compute positions $\widetilde{\zeta}(k)$. We would like to show that these are equivalent and that $\zeta'(k) = \widetilde{\zeta}(k)$.

The following simple bound will be useful in proving the intertwining identity.
\begin{lemma}
	\label{lem:zetatildeBound}
	For all $k$,
	\begin{equation}
		\widetilde{\zeta}(k) \le \zeta(k) + X^1_k.
	\end{equation}
	\begin{proof}
		Consider a path $\gamma$ going from $(-\infty, 0)$ to $(\eta^3(k), 1)$ along the original particles $\eta^2$ and $\eta^3$. Suppose this path moves to the upper line at particle $l$, and let $\widetilde{\gamma}$ be the path along the moved particles $\widetilde{\eta}^2$ and $\widetilde{\eta}^3$ which itself moves to the upper line at particle $l$. The difference in lengths is 
		\begin{equation}
			\label{eq:ShiftedPathDifferences}
			\abs{\widetilde{\gamma}} - \abs{\gamma} = (\widetilde{\eta}^2(k) - \eta^2(k)) + (\widetilde{\eta}^3(l) - \eta^3(l)) - (\widetilde{\eta}^2(l) - \eta^2(l)).
		\end{equation}
		Suppose $X^2_{l - 1} > 0$, so that $Z^1_l = X^2_l \wedge X^1_l$ by \cref{cor:AVReverseProps}. The differences in the expression above have
		\begin{align}
			\widetilde{\eta}^2(k) - \eta^2(k) &= X^2_k \wedge X^1_k \le X^1_k,\\
			\widetilde{\eta}^3(l) - \eta^3(l) &= X^3_k \wedge Z^1_l \le X^2_l \wedge X^1_l,\\
			\widetilde{\eta}^2(l) - \eta^2(l) &= X^2_l \wedge X^1_l.
		\end{align} 
		Plugging into \eqref{eq:ShiftedPathDifferences}, one gets $\abs{\widetilde{\gamma}} - \abs{\gamma} \le X^1_k$ and
		\begin{equation}
			\label{eq:ShiftedPathBound}
			\abs{\widetilde{\gamma}} - \abs{\widetilde{\gamma_0}} \le \abs{\gamma} - \abs{\gamma_0} + X^1_k.
		\end{equation} 
		
		In the other case, when $X^2_{l - 1} = 0$, we see that the path which moves to the upper line at particle $l - 1$ is no longer than $\gamma$, so we may discard $\gamma$ from the collection of paths we consider. Taking the minimum in \eqref{eq:ShiftedPathBound} over paths $\gamma$, we arrive at $\widetilde{\zeta}(k) \le \zeta(k) + X^1_k$.
	\end{proof}
\end{lemma}

At last, we are ready to prove \cref{prop:VIntertwining}
\begin{proof}[Proof of \cref{prop:VIntertwining}]
	Suppose $X^1_k = 0$ unless $-a \le k \le b$. This suffices, by \cref{lem:IntertwiningTrunc}. We use two-sided induction: if we can show the equality for \emph{some} $k \ge b + 1$, then we may use induction in either direction to get the equality for all $k$. Let $m = \max_{-a \le j \le b} X^1_j$ and let $\pi_{j, k}$ be the path $(-\infty, 0) \to (\eta^3(j), 0) \to (\eta^2(j), 1) \to (\eta^2(k), 1)$. Find $k_0 \ge b + 1$ large enough such that $\zeta(k_0) \le \min \set{\abs{\pi_{j, k_0}} : j \le b} - 2m$. Such a $k_0$ must exist almost surely, since $\Ex{X^3_0} < \Ex{X^2_0}$, and so the optimal path will tend to stay on the lower line for most of its journey. Let $l$ be minimal such that $\pi_{l, k_0}$ is an optimal path to $(\eta^2(k_0), 1)$, and call this path $\gamma$. Observe that $l \ge b$.
	
	If $X^2_{l - 1} = 0$, then $l$ would not be minimal, so we must have $X^2_{l - 1} > 0$. Also, $l \ge b + 1$ and so $X^1_l = Z^1_l = 0$. Thus $\widetilde{\eta}^3(l) = \eta^3(l)$ and $\widetilde{\eta}^3(k_0) = \eta^3(k_0)$. Then $\gamma$ is a valid path to $(\widetilde{\eta}^2(k), 1)$ under the shifted particles, hence $\widetilde{\zeta}(k_0) \le \abs{\gamma} = \zeta(k_0)$. Trivially, $\widetilde{\zeta}(k_0) \ge \zeta(k_0)$ (the particles only move forward) and $\zeta'(k_0) = \zeta(k_0)$, so $\widetilde{\zeta}(k_0) = \zeta'(k_0)$.
		
	With $k_0$ as our base case, suppose we know $\widetilde{\zeta}(k + 1) = \zeta'(k + 1)$. Since we get $\zeta'(k)$ by shifting the $\zeta$ forward, we have the recurrence
	\begin{equation}
		\zeta'(k) = (\zeta(k) + X^1_k) \wedge \zeta'(k + 1).
	\end{equation}
	The $\widetilde{\zeta}$ represent the lengths of optimal paths, and have
	\begin{equation}
		\widetilde{\zeta}(k + 1) = (\widetilde{\zeta}(k) + Z^2_k) \wedge \widetilde{\eta}^3(k + 1),
	\end{equation}
	Therefore, by the induction hypothesis
	\begin{equation}
		\zeta'(k) = (\zeta(k) + X^1_k) \wedge (\widetilde{\zeta}(k) + Z^2_k) \wedge \widetilde{\eta}^3(k + 1).
	\end{equation}
	Our task becomes showing that the expression on the right also equals $\widetilde{\zeta}(k)$. Note that each of the terms dominates $\widetilde{\zeta}(k)$: the first by \cref{lem:zetatildeBound}, the second trivially, and the third by $\widetilde{\zeta}(k) \le \widetilde{\eta}^3(k) \le \widetilde{\eta}^3(k + 1)$. So it suffices to show that any of these is equal to $\widetilde{\zeta}(k)$.
	
	The equality is immediate if we suppose $Z^2_k = 0$ or that $\widetilde{\zeta}(k) = \widetilde{\eta}^3(k) = \widetilde{\eta}^3(k + 1)$, so suppose $Z^2_k > 0$. This implies $\widetilde{\eta}^2(k) = \eta^2(k) + X^1_k$. We consider three cases:
	\begin{enumerate}
		\item Suppose $\widetilde{\zeta}(k) < \widetilde{\eta}^3(k)$, so that $\widetilde{\zeta}(k) = \widetilde{\zeta}(k - 1) + Z^2_{k - 1}$. Let $\widetilde{\gamma}$ be the minimising path with jump point $l$ maximal. Then $l \le k - 1$. We must have $\widetilde{\eta}^3(l + 1) - \widetilde{\eta}^3(l) > Z^2_l \ge 0$, as otherwise $l$ would fail to be maximal. So $\widetilde{\eta}^3(l) = \eta^3(l) + Z^1_l$. Also, the reversibility property in \cref{prop:AVReversibility}, satisfied by $Z^1$, gives $Z^1_l \ge \widetilde{\eta}^2(l) - \eta^2(l)$. So the length of $\gamma = \pi_{k, l}$ along the original particles is
		\begin{equation}
			L(\gamma) = L(\widetilde{\gamma}) - X^1_k + (\widetilde{\eta}^2(l) - \eta^2(l)) - Z^1_l \le L(\widetilde{\gamma}) - X^1_k.
		\end{equation}
		Hence $\zeta(k) \le L(\widetilde{\gamma}) - X^1_k = \widetilde{\zeta}(k) - X^1_k$. It follows from \cref{lem:zetatildeBound} that $\widetilde{\zeta}(k) = \zeta(k) + X^1_k$.
		\item Suppose $\widetilde{\zeta}(k) = \widetilde{\eta}^3(k) < \widetilde{\eta}^3(k + 1)$. Then we can repeat the argument in the previous point with $l = k$ to find that $Z^1_k \ge \widetilde{\eta}^2(k) - \eta^2(k)$. Once again, this implies $\widetilde{\zeta}(k) = \zeta(k) + X^1_k$.
		\item Suppose $\widetilde{\zeta}(k) = \widetilde{\eta}^3(k) = \widetilde{\eta}^3(k + 1)$. Then in particular, $\widetilde{\zeta}(k) = \widetilde{\eta}^3(k + 1)$.
	\end{enumerate} 
	
	With the backward induction complete, suppose instead that we know $\widetilde{\zeta}(k - 1) = \zeta'(k - 1)$. We may assume $k \ge b + 2$. From the definition, $\zeta'(k - 1) = (\zeta(k - 1) + X^1_{k - 1}) \wedge \zeta(k) = \zeta(k - 1)$ (since $X^1_{k - 1} = 0$ after the truncation), and our induction hypothesis means that
	\begin{align}
		\widetilde{\zeta}(k) &= (\widetilde{\zeta}(k - 1) + Z^2_{k - 1}) \wedge \widetilde{\eta}^3(k)\\
		&= (\zeta(k - 1) + X^2_{k - 1}) \wedge (\eta^3(k) + Z^1_k) \wedge \widetilde{\eta}^3(k + 1).\label{eq:tildezetaPostTruncDecomp}
	\end{align}
	Because $\zeta(k)$ is majorised by each of the three right hand terms, we just need to show equality with one. If $X^2_{k - 1} = 0$, then $\zeta(k) = \zeta(k - 1)$ and we're done. If $X^2_{k - 1} > 0$, then $Z^1_k \le X^1_k = 0$, and the first two terms of \eqref{eq:tildezetaPostTruncDecomp} are $(\zeta(k - 1) + X^2_{k - 1}) \wedge \eta^3(k)$. This is the recursive form of $\zeta(k)$.
\end{proof}

Having worked to prove the intertwining identities, we may now prove the main theorem rather easily. Define $H^{(1)}(X^1) = X^1$ and
\begin{equation}
	H^{(n + 1)}(X^{n + 1}, \dots, X^1) = H(H^{(n)}(X^{n + 1}, \dots, X^2), X^1).
\end{equation}
For example, $H^{(2)}(X^2, X^1) = H(X^2, X^1)$. The map $H^{(n)}$ takes $n$ sequences and applies $H$ iteratively, the output from one store becoming the input of the next. Define also $\widehat{V}(X^2, X^1)$ for the reverse weights coming from $H(X^2, X^1)$, and $\widehat{V}^{(n)}(X^n, \dots, X^1)$ recursively by $\widehat{V}^{(1)}(X^1) = X^1$ and 
\begin{equation}
	\widehat{V}^{(n + 1)}(X^{n + 1}, \dots, X^1) = \widehat{V}(X^{n + 1}, H^{(n)}(X^{n}, \dots, X^1)).
\end{equation}

A simple induction gives an intertwining identity involving the $H^{(n)}$.
\begin{lemma}
	\label{lem:VMultiIntertwining}
	Let $X^1, \dots, X^n$ be sequences such that each triple $(X^{k}, X^{k + 1}, X^{k + 2})$ satisfies the assumptions of \cref{prop:VIntertwining}. Then
	\begin{multline}
		\label{eq:VMultiIntertwining}
		V\brac[\big]{H^{(n - 1)}(X^n, \dots, X^2), X^1} = \\H^{(n)}\brac[\Big]{V\brac[\big]{X^n, \widehat{V}^{(n - 1)}(X^{n - 1}, \dots, X^1)}, V\brac[\big]{X^{n - 1}, \widehat{V}^{(n - 2)}(X^{n - 2}, \dots, X^1)}, \dots, V(X^2, X^1)}.
	\end{multline}
	\begin{proof}
		The $n = 2$ case is \cref{prop:VIntertwining}, so suppose it holds for $n - 1$. Then using \eqref{eq:VIntertwining},
		\begin{align}
			V\brac[\big]{H^{(n - 1)}(X^n, \dots, X^2), X^1} &= V\brac[\Big]{H\brac[\big]{H^{(n - 2)}(X^n, \dots, X^3), X^2}, X^1}\\
			&= H\brac[\Big]{V\brac[\big]{H^{(n - 2)}(X^n, \dots, X^3), \widehat{V}(X^2, X^1)}, V(X^2, X^1)}\\
			&= H\bigg( 
				H^{(n - 1)} \Big ( 
					V\brac[\big]{
						X^n, \widehat{V}^{(n - 2)}\brac[\big]{
							X^{n - 1}, \dots, X^3, \widehat{V}(X^2, X^1)}}, \dots \notag \\
			&\qquad \qquad \dots, V^{(2)}(X^3, \widehat{V}(X^2, X^1) \Big ), V(X^2, X^1) \bigg)\\
			&= H^{(n)}\bigg( V\brac[\big]{X^n, \widehat{V}^{(n - 1)}(X^{n - 1}, \dots, X^1)},\notag\\
			&\qquad \qquad \qquad V\brac[\big]{X^{n - 1}, \widehat{V}^{(n - 2)}(X^{n - 2}, \dots, X^1)}, \dots, V(X^2, X^1) \bigg),
		\end{align}
		where on the last line we used the recursive definition of $\widehat{V}^{(n)}$.
	\end{proof}
\end{lemma}

\begin{proof}[Proof of \cref{thm:MultiStat}]
	Let us prove this only for $V$. The proof for the map $A$ is essentially unchanged, and the statement for $H$ is just \cref{thm:HMultiStat}. Let $\rho = (\rho^1, \dots, \rho^n)$ be the vector of means in the statement of the theorem, let $W \sim \nu$ be a sequence of weights, and let $(X^1, \dots, X^n) \sim \nu_{V}^{\rho}$ (that is, they are all independent, have means $\rho$, and have marginals belonging to the family in the third column of \cref{tab:SWFPPInvariantDists} corresponding to the weight distribution). Let $\widetilde{X}^k = H^{(k)}(X^k, \dots, X^1)$, so that $(\widetilde{X}^1, \dots, \widetilde{X}^n) \sim \mu_V^\rho$. Then:
	\begin{align}
		\label{eq:MultiStatApply}
		V(\widetilde{X}^k, W) &= V(H^{(k)}(X^k, \dots, X^1), W)\\
		&= H^{(k)}(V(X^k, \widehat{V}^{(k)}(X^{k - 1}, \dots, X^1, W)), V(X^{k - 1}, \widehat{V}^{(k - 1)}(X^{k - 2}, \dots, X^1, W)), \dots, V(X^1, W)).
	\end{align}
	Observe that each entry $Z^j = V(X^j, \widehat{V}^{(j)}(X^{j - 1}, \dots, X^1, W))$ has marginal $\nu_{V}^{\rho^j}$. This is because $X^j \sim \nu_{V}^{\rho^j}$ and by induction $\widehat{V}^{(j)}(X^{j - 1}, \dots, X^1, W) \sim \nu_{V}^{\rho^{j - 1}}$ (recall the distribution of the reverse weights in \cref{prop:AVReversibility}), and since the distributions are compatible, the distribution of the input is preserved. Moreover, observe that $Z^j$ is independent of $Z^i$ for $i \ne j$. To see this, suppose $i < j$ and recall again from \cref{prop:AVReversibility} that the reverse weights of an operation are independent of the output. Our $Z^i$ is the output of $V(X^j, \widehat{V}^{(i)}(X^{i - 1}, \dots, X^1, W))$, while $Z^j$ depends on these variables only through the reverse weights of this operation. Rewrite \eqref{eq:MultiStatApply} as
	\begin{equation}
		V(\widetilde{X}^k, W) = H^{(k)}(Z^k, \dots, Z^1).
	\end{equation}
	This is the same form as in the definition of $\widetilde{X}^k$, except with $X^k$ replaced by $Z^k$. But our argument above says $(Z^1, \dots, Z^n) \disteq (X^1, \dots, X^n)$, so 
	\begin{equation}
		(V(\widetilde{X}^1, W), \dots, V(\widetilde{X}^n, W)) \disteq (\widetilde{X}^1, \dots, \widetilde{X}^n).
	\end{equation} 
	This is what it means for the distribution $\mu_{V}^{\rho}$ to be a fixed point.
\end{proof}

\subsection{The $H$ map as reverse weights in LPP} 
\label{ssec:LPPConnection}
Let us now observe a connection between our $H$ update and the update map in last passage percolation (LPP). The usual form of LPP considered on the integer lattice differs from our model in that weights live on \emph{vertices}, rather than edges, and our passage time is the maximal weight of a directed path, rather than the minimum. 

Consider $\Gamma = \set{k e_1 : k \in \bZ}$ and to each vertex $x$ associate a weight $W(x)$. Let $\overline{G} : \Gamma \to \bR$ give passage times on the boundary $\Gamma$. Then we define passage times on $\Gamma + e_2$ by
\begin{equation}
	\overline{L}^{\Gamma}(k e_1 + e_2) = \max_{l \le k} G(l e_1) + \sum_{j = l}^{k}W_{j e_1 + e_2}.
\end{equation}
It is important that the weight of the endpoint is included in the sum. Consider increments
\begin{equation}
	\label{eq:LPPIncrements}
	\overline{I}_k = \overline{L}^\Gamma((k + 1)e_1) - \overline{L}^\Gamma(k e_1),\quad \overline{J}_k = \overline{L}^\Gamma(k e_1 + e_2) - \overline{L}^\Gamma(k e_1),\quad \overline{I}'_k = \overline{L}^\Gamma((k + 1) e_1 + e_2) - \overline{L}^\Gamma(k e_1 + e_2)
\end{equation}
and write $\overline{W}_k = W_{(k + 1)e_1 + e_2}$. Then the increments obey the local update rules
\begin{equation}
	\label{eq:LPPLocalUpdate}
	\overline{I}'_k = \overline{W}_k + (\overline{I}_k - \overline{J}_k)^+,\quad \overline{J}_{k + 1} = \overline{W}_k + (\overline{J}_k - \overline{I}_k)^+.
\end{equation}

The map $(\overline{I}_k, \overline{J}_k, \overline{W}_k) \mapsto (\overline{I}'_k, \overline{J}_{k + 1})$ may be augmented by $\overline{W}'_k = \overline{I}_k \wedge \overline{J}_k$ to form an involution $(\overline{I}_k, \overline{J}_k, \overline{W}_k) \mapsto (\overline{I}'_k, \overline{J}_{k + 1}, \overline{W}')$. There is a vast literature on exponential-weight LPP (and its discrete geometric-weight sibling), and in particular it is well known that if $I_k \sim \Exp(\rho)$ for $0 < \rho < 1$, then this involution fixes the distribution of $(\overline{I}_k, \overline{J}_k, \overline{W}_k)$. Thus the $\overline{W}'_k$ play exactly the role of the reverse weights we have discussed above. 

Analogously to our map $H$ from SWFPP, given a sequence $(\overline{I}_k)$ which form the increments of boundary values $\overline{G} : \Gamma \to \bR$, and weights $\overline{W}_k$, let $D(\overline{I}, \overline{W})$ be defined as the sequence of increments of the passage times along $\Gamma + e_2$, namely the $(\overline{I}'_k)$ from \eqref{eq:LPPIncrements}. This $D$ can also be seen as the interdeparture times of an M/M/1 queue with interarrival times $\overline{I}$ and service $\overline{W}$. Let $R(\overline{I}, \overline{W})$ be the sequence of $\overline{W}'_k$ which give the reverse weights in the exponential setting. See \cite{sepp-cgm-18, fan-sepp-joint-buss} for further discussion of these relations.

\begin{lemma}
	\label{lem:HEqualsR}
	Let $I,\, W$ be sequences for which $H(I, W)$ and $D(W, I)$ make sense. Then
	\begin{equation}
		\label{eq:HEqualsR}
		H(I, W) = \sigma_{-1} R(\sigma_1 W, I).
	\end{equation}
	Here $\sigma_n$ is the operator shifting a sequence forward by $n$, so that index $(\sigma_n I)_k = I_{k - n}$.
	\begin{proof}
		Let $\Gamma = \set{k e_1 : k \in \bZ}$ be our boundary and let $G(k e_1) = \sum_{j = 0}^{k - 1}I_j - \sum_{j = k}^{-1}I_k$ be the passage times along the boundary. Let $L^\Gamma$ be the SWFPP passage times on $\Gamma + e_2$ determined by $G$ and $W$, given by
		\begin{equation}
			L^\Gamma(k e_1 + e_2) = \min_{l \le k} G(l e_1) + \sum_{j = l + 1}^{k}W_j,
		\end{equation}
		and let $I'_k$ and $J_k$ be the horizontal and vertical increments:
		\begin{equation}
			I'_k = L^\Gamma((k + 1) e_1 + e_2) - L^\Gamma(k e_1 + e_2), \quad J_k = L^\Gamma(k e_1) - L^\Gamma(k e_1 + e_2).
		\end{equation}
		Further, let $\overline{I}_k = W_{k - 1} = (\sigma_1 W)_k$ and $\overline{W}_k = I_{k}$, and $\overline{G}(k e_1) = \sum_{j = 0}^{k - 1}\overline{I}_j - \sum_{j = -1}^{k}\overline{I}_k$. We let $\overline{L}^\Gamma$ the now the LPP passage times, defined as
		\begin{equation}
			\overline{L}^{\Gamma}(k e_1 + e_2) = \max_{l \le k} \overline{G}(l e_1) + \sum_{j = l}^{k}\overline{W}_{j}.
		\end{equation}
		Finally, set
		\begin{equation}
			\overline{I}'_k = \overline{L}^\Gamma((k + 1) e_1 + e_2) - \overline{L}^\Gamma(k e_1 + e_2),\quad \overline{J}_k = \overline{L}^\Gamma(k e_1 + e_2) - \overline{L}^\Gamma(k e_1),
		\end{equation}
		noting that the order of $\overline{J}_k$ is swapped relative to $J_k$.
		
		We first claim that $\overline{J}_{k + 1} = I_k + J_k$. This is clear once we expand the vertical increments:
		\begin{equation}
			J_k = \max_{l \le k}\sum_{j = l}^{k - 1}I_j - W_j, \quad \text{and} \quad \overline{J}_{k + 1} = \overline{W}_k + \max_{l \le k}\sum_{j = l}^{k - 1}\overline{W}_j - \overline{I}_{j + 1}.
		\end{equation}
		Now plugging in the definitions of $\overline{I}_k$ and $\overline{W}_k$ gives the identity. Then we find, using \eqref{eq:SquareIncs} and \eqref{eq:LPPIncrements}, that
		\begin{align}
			I'_{k} &= (I_k + J_k) \wedge W_k\\ 
			&= \overline{J}_{k + 1} \wedge \overline{I}_{k + 1}\\
			&= \overline{W}'_{k + 1}.
		\end{align}
		This is what we wanted to show.
	\end{proof}
\end{lemma}

Consider the reversed operators $\overleftarrow{D}$ and $\overleftarrow{R}$ . These are useful because we can succinctly write the time reversal property satisfied by the reverse weights:
\begin{equation}
	\label{eq:BackwardsD}
	\overline{I} = \overleftarrow{D}(D(\overline{I}, \overline{W}), R(\overline{I}, \overline{W})) \quad \text{and} \quad \overline{W} = \overleftarrow{R}(D(\overline{I}, \overline{W}), R(\overline{I}, \overline{W})).
\end{equation} 
From this and \cref{lem:HEqualsR} we may derive a number of interesting distributional inequalities, ultimately leading to a distributional inequality between exponential SWFPP Busemann functions and those in LPP.
\begin{lemma}
	\label{lem:IDIWEqualsRIWW}
	Take $\rho_1 > \rho_2 > 0$ and let $I^1$, $I^2$ be independent i.i.d sequences with marginals $\Exp(\rho_1)$ and $\Exp(\rho_2)$. Then
	\begin{equation}
		\label{eq:IDIWEqualsRIWW1}
		(I^1, D(I^2, I^1)) \disteq (\overleftarrow{R}(I^2, I^1), I^2)
	\end{equation}
	and
	\begin{equation}
		\label{eq:IDIWEqualsRIWW2}
		(I^1, \overleftarrow{D}(I^2, I^1)) \disteq (R(I^2, I^1), I^2)
	\end{equation}
	\begin{proof}
		Let $\widetilde{I}^1 = R(I^2, I^1)$ and $\widetilde{I}^2 = D(I^2, I^1)$. Burke's theorem (for example \cite[Section 4]{drai-mair-ocon}) tells us that $(\widetilde{I}^1, \widetilde{I}^2) \disteq (I^1, I^2)$. Using the time reversal of \eqref{eq:BackwardsD},
		\begin{equation}
			(I^1, D(I^2, I^1)) \disteq (\widetilde{I}^1, D(\widetilde{I}^2, \widetilde{I}^1)) = (R(I^2, I^1), \overleftarrow{I^2}) \disteq (\overleftarrow{R}(I^2, I^1), I^2).
		\end{equation}
		For the last equality we replaced $(I^1, I^2)$ by $(\overleftarrow{I}^1, \overleftarrow{I}^2)$. This is \eqref{eq:IDIWEqualsRIWW1}, and \eqref{eq:IDIWEqualsRIWW2} follows by reversing the directions of all the sequences.
	\end{proof}
\end{lemma}

\begin{lemma}
	\label{lem:WHIWEqualsDWII}
	Take $\rho_1 > \rho_2 > 0$ and let $I^1$, $I^2$ be independent i.i.d sequences with marginals $\Exp(\rho_1)$ and $\Exp(\rho_2)$. Then
	\begin{equation}
		\label{eq:WHIWEqualsDWII}
		(I^2, H(I^1, I^2)) \disteq (\overleftarrow{D}(I^2, I^1), I^1).
	\end{equation}
	\begin{proof}
		From \cref{lem:HEqualsR} and using translation invariance of the distributions, one gets
		\begin{equation}
			(I^2, H(I^1, I^2)) = (I^2, \sigma_{-1} R(\sigma_1 I^2, I^1)) \disteq (\sigma_1 I^2, R(\sigma_1 I^2, I^1)) \disteq (I^2, R(I^2, I^1)).
		\end{equation}
		We have just shown in \cref{lem:IDIWEqualsRIWW} that this last pair is equal in distribution to $(\overleftarrow{D}(I^2, I^1), I^1)$.
	\end{proof}
\end{lemma}

Now we extend \cref{lem:WHIWEqualsDWII} to $n$-tuples. With $H^{(n)}$ as before, define $\overleftarrow{D}^{(2)}(I^2, I^1) = \overleftarrow{D}(I^2, I^1)$, and $\overleftarrow{D}^{(n)}(I^n, \dots, I^1) = \overleftarrow{D}(I^n, \overleftarrow{D}^{(n - 1)}(I^{n - 1}, \dots, I^1))$.
\begin{proposition}
	\label{prop:NestedHEqualsNestedD}
	Let $\rho_1 > \cdots > \rho_n > 0$ and let $I^1, \dots, I^n$ be independent i.i.d sequences where $I^k$ has marginals $\Exp(\rho_k)$. Then
	\begin{equation}
		(I^n, H^{(2)}(I^{n - 1}, I^n), \dots, H^{(n)}(I^1, \dots, I^n)) \disteq (\overleftarrow{D}^{(n)}(I^n, \dots, I^1), \dots, \overleftarrow{D}^{(2)}(I^2, I^1), I^1).
	\end{equation}
	\begin{proof}
		The $n = 2$ case is just \cref{lem:WHIWEqualsDWII}. Suppose we have the equality for $n - 1$. Then since $\overleftarrow{D}^{(k)}(I^k, \dots, I^1) \disteq \overleftarrow{D}^{(k - 1)}(I^k, \dots, I^3, D(I^2, I^1))$, we have
		\begin{equation}
			(\overleftarrow{D}^{(n)}(I^n, \dots, I^1), \dots, \overleftarrow{D}^{(2)}(I^2, I^1)) \disteq (\overleftarrow{D}^{(n - 1)}(I^n, \dots, I^2), \dots, \overleftarrow{D}^{(2)}(I^3, I^2), I^2).
		\end{equation}
		To the latter we may apply the induction hypothesis, and end up with
		\begin{equation}
			\label{eq:NestedHEqualsNestedDInductionStep}
			(\overleftarrow{D}^{(n)}(I^n, \dots, I^1), \dots, \overleftarrow{D}^{(2)}(I^2, I^1)) \disteq (I^n, H^{(2)}(I^{n - 1}, I^n), \dots, H^{(n - 1)}(I^2, \dots, I^n)).
		\end{equation}
		Now suppose we have sampled the right side of \eqref{eq:NestedHEqualsNestedDInductionStep} and wish to sample $I^1$. Since the only dependence of our conditioning with $I^1$ is through $\overleftarrow{D}(I^2, I^1)$, this is equivalent to sampling $I^1$ conditional on $\overleftarrow{D}(I^2, I^1)$. We may do this by sampling independent reverse weights $O^1$ as a sequence of i.i.d $\Exp(\rho_1)$ variables, and setting $I^1 = R(\overleftarrow{D}(I^2, I^1), O^1)$\footnote{There is some complication given the various reversals of indexing, but we can check that this does indeed sample from $I^1 \given \overleftarrow{D}(I^2, I^1)$.}.
		
		Suppose on the other hand we have sampled from the right side of \eqref{eq:NestedHEqualsNestedDInductionStep}. Sampling $H^{(n)}(I^1, \dots, I^n)$ is easy --- we just draw $I^1$ independently and take $H(I^1, H^{(n - 1)}(I^2, \dots, I^n))$. Our equality between $H$ and $R$ in \cref{lem:HEqualsR}, and the manipulations we did in the proof of \cref{lem:WHIWEqualsDWII} showing that the shifts do not affect the distributions, means that our two conditional sampling procedures are equivalent. Thus the full joint distributions are equal, and
		\begin{equation}
			(I^n, H^{(2)}(I^{n - 1}, I^n), \dots, H^{(n)}(I^1, \dots, I^n)) \disteq (\overleftarrow{D}^{(n)}(I^n, \dots, I^1), \dots, \overleftarrow{D}^{(2)}(I^2, I^1), I^1).
		\end{equation}
	\end{proof}
\end{proposition}

All that remains is to quote the distribution of the exponential LPP distribution and we will have proved \cref{prop:EqualDists}.
\begin{theorem}[Theorem 3.2 of \cite{fan-sepp-joint-buss}]
	Take $1 > \rho_1 > \cdots > \rho_n > 0$ and let $\overline{B}^{k}(x, y)$ be the north-east Busemann function in exponential LPP such that $\Ex{\overline{B}^k(0, e_1)} = \rho_k^{-1}$. Let $I^1, \dots, I^n$ be independent sequences of i.i.d exponential variables, with the marginal of $I^k$ being $\Exp(\rho_k)$. Then
	\begin{equation}
		\set{(\overline{B}^1((k + 1)e_1, k e_1), \dots, \overline{B}^n((k + 1)e_1, k e_1)) : k \in \bZ} \disteq (\overleftarrow{D}^{(1)}(I^1), \dots \overleftarrow{D}^{(n)}(I^n, \dots, I^1).
	\end{equation}
\end{theorem}
The use of the reversed $\overleftarrow{D}$ map rather than the forward $D$ reflects the use of north-east rather than south-west Busemann functions.

\begin{remark}
	All of the properties proved in \cite{fan-sepp-joint-buss} for the multi-line distribution carry over to the Busemann process on an antidiagonal in exponential SWFPP. In particular, we have the distribution on a single edge as that of a certain inhomogeneous Markov process. Convergence to the stationary horizon for these measures was proved in \cite{bus-buse-scaling}. There is a similar statement for the horizontal edge Busemann process.
\end{remark}

\section{Semi-infinite geodesics near the axis}
\label{sec:NearAxisComps}
\subsection{Branch points along the vertical axis}
\label{ssec:Branches}
Recall from \cref{ssec:HighwaysResults} that we defined $\Gamma_\infty$ to be the tree of semi-infinite geodesics rooted at the origin. From this, define
\begin{equation}
	\label{eq:BranchPoints}
	\cB = \set{k : (0,0) \to (0,k) \to (1, k) \subseteq \Gamma_\infty}
\end{equation}
to be the set of those heights where the tree $\Gamma_\infty$ branches from the $y$-axis. In light of the procedure for producing Busemann geodesics, the set 
\begin{equation}
	\label{eq:BuseBranchPoints}
	\cB' = \set{0} \cup \set{k \ge 1 : \text{for some } \xi \in U \text{ we have } B^\xi((k - 1) e_2, k e_2) = 0,\, B^\xi(k e_2, (k + 1)e_2) \ne 0}
\end{equation}
is the set of heights where Busemann geodesics branch from the $y$-axis. Clearly $\cB' \subseteq \cB$.  In the general case, this inclusion allows us to say at least one thing which is not immediately obvious:

\begin{proposition}
	In SWFPP, the set $\cB$ is almost surely infinite.
	\begin{proof}
		If the weight distribution has an atom at $0$, then every $k$ such that $W(k e_2) = 0$ may be seen to belong to $\cB$. Suppose then that there is no atom at $0$. The arguments of \cite{hass-directed-fpp}\footnote{Or, indeed, easier arguments given our simpler setting.} may be used to show that the limit shape has a sequence of extreme points converging to the $e_2$-axis. It follows that there is a sequence of directions $\set{\xi_k} \subseteq \scrU$ such that $\xi_k \to e_2$. Our more general existence theorem, \cref{thm:BusemannExistenceGeneral}, ensures the existence of Busemann functions associated to these directions. 
		
		An elementary argument shows that almost surely, the critical angle $\xi^*$ is not $e_2$. This is to say that for angles sufficiently close to the vertical axis, the first step of the semi-infinite geodesic will be up. If $\alpha_k = \Prob{B^{\xi_k} (0, e_2) \ne 0}$, then in light of the rule for producing semi-infinite geodesics in \cref{rem:GeodesicConstruction}, we have $\alpha_k \to 0$.
		
		Let $A_N$ be the event that $\max \cB = N$. Then on this event, for all $k$ large enough,  $B^{\xi_k} (0, e_2) \ne 0$. So $\Prob{A_N} \le \liminf \Prob{B^{\xi_k} (0, e_2) \ne 0} = \liminf \alpha_k = 0$. Hence $\Prob{\bigcup_{N \ge 0} A_N} = 0$, and the branch points are infinite almost surely.
	\end{proof}
\end{proposition}

Having $\cB' = \cB$ is equivalent to the statement that in a given realisation of $\Gamma_\infty$, there are no three semi-infinite geodesics with the same direction. The corresponding statement has been shown to hold for exponential LPP \cite{coup-n3g} and in the directed landscape \cite{bus-n3g}, and is expected to hold in wide generality. This is the so-called N3G problem.

\begin{theorem}
	\label{thm:SWFPPN3G}
	Under exponential weights, almost surely we have 
	\begin{equation}
		\cB' = \cB.
	\end{equation}
\end{theorem}

To prove this for exponential LPP, Coupier reduced the problem to the statement that in the TASEP speed process, two particles with the same asymptotic speed eventually collide. This in turn had been shown in \cite{amir-angel-valko-tasep-speed}. In our case, the speed process for the associated discrete-time particle system has been studied in \cite{mart-schi-discrete}, wherein the authors claim (without proof) that two particles having the same speed will eventually collide. This statement, combined with Coupier's observation, gives the theorem. 

Being able to identify $\cB$ with $\cB'$ in the exponential case and studying the joint distributions of the Busemann functions here gives a precise, albeit difficult-to-work-with, description of the branch points. For a direction $\xi \in \scrU$, let $\alpha(\xi) = \Prob{B^\xi (0, e_2) \ne 0}$. Then due to the independence of the Busemann increments in the solvable case, we know that the variable
\begin{equation}
	\label{eq:BranchPointDef}
	b(\xi) = \min \set{k : B^{\xi-} (0, (k + 1)e_2) \ne 0}
\end{equation}
is geometrically distributed with parameter $\alpha(\xi)$. Moreover, $b(\xi_1) \le b(\xi_2)$ whenever $\xi_1 \le \xi_2$. Hence
\begin{equation}
	\label{eq:BranchSetAsValues}
	\cB = \set{b(\xi) : \xi \in \scrU}	
\end{equation}
may be regarded as the values realised by a particular monotonic coupling of geometric variables.

The most natural such coupling of geometric variables can be found as a function of a sequence $(U_k : k \ge 0)$ of i.i.d $\Unif [0, 1]$ variables, by setting $X(t) = \min \set{k \ge 0 : U_k \ge 1 - t}$. One may verify that the process $\set{X(t) : 0 \le t < 1}$ is a cadlag pure jump process with  $X(t) \sim \Geom_0(1 - t)$. The jump points are given by a Poisson process of intensity $(1 - t)^{-1}dt$ and a jump at $t$ follows a $\Geom_+(1 - t)$ distribution, independent of everything else. Up to a time change, the continuous space analogue of this coupling appears in \cite{fan-sepp-joint-buss} as the distribution of the Busemann process along an edge. We may also verify that the Poisson process giving the jump points $\set{t_k}$ is equivalent to choosing $t_0 = 0$ and choosing each $t_{k + 1}$ uniformly on $[t_k, 1)$.

As an abuse of notation, write $b(t) = b(\xi)$ when $t = 1 - \alpha(\xi)$. The process $\set{b(t) : 0 \le t < 1}$ is once again a cadlag pure jump process with positive jumps, but in contrast to the above is neither a Markov process nor has its jump points placed independently. We do however have a sort of Markov property for the jump points. Define $\cD(b) = \set{t : b(t-) \ne b(t)}$, the random set of jump points (discontinuities) of $\set{b(t)}$, and write $\cD(b) = (d_0, d_1, \dots)$, where $0 = d_0 < d_1 < \cdots$\footnote{It will be convenient to have $0 \in \cD(b)$, even if it isn't strictly a jump.}. Suppose we know $d_0, \dots, d_k$ and wish to sample $d_{k + 1}$. Then 
\begin{equation}
	\label{eq:BranchJumpDist}
	\Prob{d_{k + 1} \le t} = t^{b(d_k)} \brac[\big]{1 - \frac{1 - t}{1 - d_k}}.	
\end{equation}
Note that the dependence on $\set{b(t) : t \le d_k}$ is solely through $b(d_k)$. In particular, since $b(d_k) \ge 0$ and positive values skew the distribution towards $1$, we see that a time $d_{k + 1}$ sampled in this way is stochastically larger than under the uniform sampling procedure for $\set{X(t)}$. Thus our $b(t)$ will tend to take fewer jumps than the more natural coupling, and must compensate for this by taking larger jumps. We return to this point in \cref{lem:XCoupling}.

The jumps are somewhat more complicated and do depend on the entire trajectory. Let $H_0$ be a map from $(\bZ^\bN \times \bZ^\bN) \to \bZ^\bN$ representing a one-sided version of our store map $H$, where initially the store is empty. That is, if we have a sequence inputs $I = (I_0, \dots)$ and services $W = (W_0, \dots)$, we produce the sequence $I' = H_0(I, W)$ according to
\begin{align}
	X_k &= X_{k - 1} + I_{k - 1} - I'_{k - 1}\\
	I'_k &= (I_k + X_k) \wedge W_k,	
\end{align}
where we take $X_{0} = 0$ as our initial value. 

Given the collection of jumps $d_k$, sample $I^k \sim \nu^k$, independently, where $\nu^k$ is the marginal distribution of the vertical increment Busemann functions in the direction corresponding to $t_k$, namely $\alpha^{-1}(1 - t_k)$. Let $\sigma_n I$ be the sequence where $n$ copies of $0$ are appended to the start of $I$. We set $b(0) = 0$ (there is always a branch at height $0$) and $\widetilde{I}^1 = \sigma_1 I_1$. Then we let $b(t_1) = \min \set{n : \widetilde{I}^1 \ne 0}$, and recursively define
\begin{align}
	\label{eq:BranchRecc}
	\widetilde{I}^{k + 1} &= H_0(\sigma_{b(t_k) + 1} I^{k + 1}, \widetilde{I}^k),\\
	b(t_{k + 1}) &= \min \set{n : \widetilde{I}^{k + 1} \ne 0}.
\end{align}
Remarkably, the jump process $b(t)$ defined this way has geometric marginals.

\begin{proposition}
	\label{prop:BranchSample}
	The procedure described above does indeed sample from the process $\set{b(t) : 0 \le t < 1}$.
	\begin{proof}
		Consider sampling at finitely many times $0 = t_0 < t_1 < \cdots < t_n$. We have a representation of the Busemann functions $\set{B^{\xi_k}((j + 1) e_2, j e_2)}$ in terms of the multi-line process. The branch points can be recovered from \eqref{eq:BranchPointDef}, giving us the joint distribution of $(b(t_1), \dots, b(t_n))$.  
		
		The multi-line process takes as input independent sequences $I^k \sim \nu^k$ and recursively applies the $H$ map. The first sequence, $I^0$, may be considered a constant sequence of $\infty$, and $b(t_0) = 0$ as expected. Then $\widetilde{I}^1 = H(I^1, I^0) = I^1$, and $b(t_1)$ is just a geometric random variable, as expected. Now consider $\widetilde{I}^2 = H(I^2, \widetilde{I}^1)$. The value of $b(t_2)$ will be the first non-negative time when the store with input $I^2$ and service $\widetilde{I}^1$ has non-zero output. The store has its first non-zero service at time $b(t_1)$. There will be an output at this time if there is material in the store. This fails to happen if there was both no material in the store already at time $0$, and if the input from $I^2$ was zero between time $0$ and time $b(t_1)$. Thus the probability of no output at time $b(t_1)$ may be computed as
		\begin{equation}
			t^{b(t_1)} \brac[\big]{1 - \frac{1 - t_2}{1 - t_1}}.
		\end{equation}
		Here we use the explicit distributions for the $I^k$, and also for the corresponding store quantities in \cref{tab:SWFPPInvariantDists}. Hence $b(t_2) = b(t_1)$ with probability $1 - t_2^{b(t_1)} (1 - (1 - t_2)/ (1 - t_1))$, and in the same way
		\begin{equation}
			\label{eq:BranchSampleDist}
			\Prob{b(t_k) = b(t_1)} = 1 - t_k^{b(t_1)}\brac[\big]{1 - \frac{1 - t_k}{1 - t_1}}.
		\end{equation}
		This naturally leads to the distribution of the next jump point given in \eqref{eq:BranchJumpDist}.
		
		Let $t_{k_1}$ be the first time at which the value is different from $b(t_1)$. The Busemann functions $(B^{\xi_1}, \, B^{\xi_{k_1}})$ are jointly distributed as $H(I^{k_1}, \widetilde{I}^1)$, conditioned on the store being empty at time $0$, and $I^{k_1}$ having zero entries up to time $b(t_1)$. This is the relation we described by \eqref{eq:BranchRecc}. Now we look for the first subsequent time $t_{k_2}$ such that $b(t_{k_2}) \ne b(t_{k_1})$, whose distribution we may read from \eqref{eq:BranchSampleDist}, and repeat the procedure until we reach $t_n$. Going to continuous time by taking $n \to \infty$ and the spacing of the times to $0$ yields the claimed sampling procedure.
	\end{proof}	
\end{proposition}

Our \cref{thm:HighwaysBranch} estimates the number of heights in $\cB$ below a threshold $N$. This is equivalent to estimating the number of jumps in our process $\set{b(t)}$ before it hits $N$. Using \eqref{eq:BranchJumpDist} to analyse the distribution of the jump points directly seems rather difficult, so our strategy of proof is more indirect. We discretise the process and produce a dominating process $\overline{b}(t)$ with independent jumps for which the hitting time is more easily understood. We then show that a positive proportion of the jumps of $\set{\overline{b}(t)}$ are shared with $\set{b(t)}$, yielding a lower bound. For the upper bound, we use a comparison with the $\set{X(t)}$ process to show that the jumps in the branch process don't cluster ``too much'', so that the number in the discretised process is a good proxy. One last comparison with a stochastically smaller process gives an upper bound for the discretised process, completing the argument. For ease of presentation, we prove the lower  and upper bounds separately, in \cref{prop:BranchLowerBound} and \cref{prop:BranchUpperBound}.

\subsubsection{A lower bound on jumps}
\begin{proposition}
	\label{prop:BranchLowerBound}
	There is $C_1 > 0$ such that
	\begin{equation}
		\abs{\cB \cap [0, N]} \ge C_1 \log N	
	\end{equation}
	for all $N$ large enough, almost surely.	
\end{proposition}

The argument goes through our first coupled process, which is just the sum of certain independent increments. Here and in the remainder of the section, we set $t_k = 1 - 2^{-k}$.

\begin{lemma}
	\label{lem:bOver}
	There is a coupling of our process $\set{b(t)}$ and a process $\set{\overline{b}(t)}$ such that $b(t_k) \le \overline{b}(t_k)$ for all $k \ge 0$, where $\overline{b}(t)$ may be represented as
	\begin{equation}
		\overline{b}(t) = \sum_{k = 1}^{\floor{- \log_2 (1 - t)}} \sqrbrac[\bigg]{1 + \sum^{k}_{j = 1} z_{k, j}},
	\end{equation}
	with $z_{k, j} \sim \Geom_0 (2^{-j})$ and the $z_{k, j}$ all mutually independent.
	
	\begin{proof}
		Consider $b(t_1), b(t_2), \dots$, whose joint distribution we may access from the multi-line process. Take $I^k \sim \nu^k$ to be the independent inputs into the multi-line process and $\widetilde{I}^k$ the associated outputs. As in the proof of \cref{prop:BranchSample}, we describe how the values of $b(t_k)$ depend on the inputs $I^k$. 
		
		Observe that for us to have $b(t_{k}) \ne b(t_{k - 1})$, the store with inputs $I^{k}$ and service $\widetilde{I}^{k - 1}$ must have no output at time $b(t_{k - 1})$. Let $E_{k}(b(t_{k - 1}))$ be the event that the store with input $I^k$ is empty immediately before the input at time $b(t_{k - 1})$. Since $\widetilde{I}^{k - 1}_j = 0$ for $0 \le j < b(t_{k - 1})$, it must be that the store is empty before the input at time $0$, and subsequently $I^{k}_j = 0$ for $0 \le j \le b(t_{k - 1})$. This probability may be computed as 
		\begin{equation}
			\Prob{E_{k}(b(t_{k - 1}))} = \frac{1 - t_{k}}{1 - t_{k - 1}} t_{k}^{b(t_{k - 1})} = \frac{1}{2} (1 - 2^{-k})^{b(t_{k - 1})}.	
		\end{equation}
		Again we use the distributions given in \cref{tab:SWFPPInvariantDists}. We will have $b(t_{k}) \ne b(t_{k - 1})$ if $E_{k}(b(t_{k - 1}))$ occurs along with having $I^k_{b(t_{k - 1})} = 0$, for a total probability of
		\begin{equation}
			\label{eq:pkDef}
			p_k = \frac{1}{2} (1 - 2^{-k})^{b(t_{k - 1})}.	
		\end{equation}
		though we do not need this expression in proving this lemma, it will be useful later on.
		
		Now, let us look at the value of $b(t_{k})$. If $E_{k}(b(t_{k - 1}))$ does not occur, then $b(t_k) = b(t_{k - 1})$. If it does, then the new value will be the time of the first output for the store with input $I^k$ and service $\widetilde{I}^{k - 1}$. This must be after the next nonzero input to the store, say at time $y_{k, k} = b(t_{k - 1}) + n_k(b(t_{k - 1}))$, and thereafter occurs at the first nonzero service. Here $n_k(b) = \min \set{n \ge 0 : I^k_{n + b} \ne 0}$, and $n_k(b) \sim \Geom_0(1 - t_k)$.
		
		The sequence of services, $\widetilde{I}^{k - 1}$, is itself the output of a store, with input $I^{k - 1}$ and service $\widetilde{I}^{k - 2}$. Let $E_{k - 1}(y_{k, k})$ be the event that this store is empty before input at time $y_{k, k}$. On this event, we must now wait for the first subsequent nonzero input from $I^{k - 1}$, and then for the first nonzero service. If the store is not empty, then we need only wait for the first nonzero service. We then wait time
		\begin{equation}
			y_{k, k - 1} = \bone_{E_{k - 1}(z_{k, k})} n_{k - 1}(y_{k, k})	
		\end{equation}
		for the input, and for the service we look at the store with input $I^{k - 2}$ and service $\widetilde{I}^{k - 3}$, and so on. Ultimately,
		\begin{align}
			\label{eq:bJumpDecomp}
			b(t_k) &= b(t_{k - 1}) + \sum_{j = 1}^{k}y_{k, j},\\
			&= b(t_{k - 1}) + \sum_{j = 1}^{k}\bone_{E_{j}(y_{k, j + 1})} n_{j}(y_{k, j + 1}).	
		\end{align}
		For this to work at $j = 1$ we should define $y_{k, k + 1} = b(t_{k - 1})$. Observe that the quantities $n_{j}(y_{k, j + 1})$ are all independent of each other.
		
		The only troublesome point in \eqref{eq:bJumpDecomp} is the dependency on the event $E_{j}(y_{k, j + 1})$. We define the process $\overline{b}(t)$ to ignore this dependence. Define $\overline{b}(t_0) = b(t_0) = 0$, and thereafter define
		\begin{equation}
			\overline{b}(t_k) = \overline{b}(t_{k - 1}) + 1 + \sum_{j = 1}^{k}z_{k, j},	
		\end{equation}
		where $z_{k, j}$ is defined recursively by $z_{k, k + 1} = \overline{b}(t_{k - 1})$ and 
		\begin{equation}
			z_{k, j} = n_{j}(z_{k, j + 1}).
		\end{equation}
		The additional $1$ in the sum is merely to ensure that $\overline{b}(t_k) \ne \overline{b}(t_{k - 1})$. The collection $\set{z_{k, j} : k \ge 1,\, 1 \le j \le k}$ consists of independent geometric variables. Since both $\set{b(t_k)}$ and $\set{\overline{b}(t_k)}$ are defined through the same input variables $I^k$, and since for $\overline{b}(t_k)$ we always overestimate the time we wait to see output in a queue, we ensure that $b(t_k) \le \overline{b}(t_k)$.
	\end{proof}
\end{lemma}

We are interested in the number of values taken by $\set{b(t)}$ before it hits a level $N$. Let $T_N = \inf\set{t : b(t) \ge N}$, and let $V_N = \abs{\set{b(t) : 0 \le t < T_N}}$. Similarly, let $\overline{T}_N$ and $\overline{V}_N$ be the corresponding quantities for $\set{\overline{b}(t)}$.

\begin{lemma}
	\label{lem:HittingTimeAsympts}	
	Almost surely, we have the asymptotics
	\begin{equation}
		\label{eq:HittingTimeAsympt}
		\lim_{N \to \infty} \frac{-\log (1 - T_N)}{\log N} = 1
	\end{equation}
	and 
	\begin{equation}
		\lim_{N \to \infty} \frac{\overline{V}_N}{\log_2 N} = 1.	
	\end{equation}
	\begin{proof}
		Find $K$ such that $2^K \le N < 2^{K + 1}$. The first limit is a very coarse estimate which only uses the marginals of $b(t)$, namely that they are geometric with mean $(1 - t)^{-1}$. We again use the dyadic times $t_k = 1 - 2^{-k}$. We know 
		\begin{equation}
			\Prob{T_N \ge t_k} = \Prob{b(t_k) \le N} = 1 - (1 - 2^{-k})^N \le 1 - (1 - 2^{-k})^{2^{K + 1}}.	
		\end{equation}
		So 
		\begin{equation}
			\Prob{T_N \ge t_{K(1 + \epsilon)}} \le 1 - (1 - 2^{-K(1 + \epsilon)})^{2^K} \le C (1 - e^{-2^{-K \epsilon}}).	
		\end{equation}
		This upper bound decays very quickly and in particular is summable. We may thus have $T_N \ge t_{K(1 + \epsilon)}$ only finitely often. Thus
		\begin{equation}
			\limsup_{N \to \infty} \frac{-\log_2(1 - T_N)}{\log_2 N} \le 1 + \epsilon	
		\end{equation}
		almost surely. Taking $\epsilon$ to zero and finding the matching lower bound (similar to what we do below), gives the limit.
		
		For the second part, note that since $\overline{b}(t)$ is constant between its jumps at times $t_k$, we may write $\overline{V}_N = - \log_2 (1 - \overline{T}_N)$. We must show that 
		\begin{equation}
			\lim_{N \to \infty} \frac{- \log_2(1 - \overline{T}_N)}{\log_2 N} = 1.	
		\end{equation}
		A short calculation gives $\Prob{z_{k, j} \ge 2^{j(1 + \epsilon)}} \le e^{-2^{j \epsilon}}$, and therefore there will be some (random) $C > 0$ such that $z_{k, j} \le C 2^{j(1 + \epsilon)}$, for all $k,\, j$. Hence
		\begin{align}
			\overline{b}(t_k) &\le k + \sum_{l = 1}^{k}\sum_{j = 1}^{l}C 2^{j(1 + \epsilon)}\\
			&\le k + C \sum_{l = 1}^{k} 2^{(l + 1)(1 + \epsilon)}\\
			&\le k + C 2^{(k + 2)(1 + \epsilon)}\\
			&\le C 2^{k(1 + \epsilon)},	
		\end{align}
		where in the last line we enlarge $C$ to absorb the other terms. Some manipulation gives a bound on the hitting time: if $\overline{b}(t_k) \ge 2^K$, then 
		\begin{equation}
			k \ge (1 + \epsilon)^{-1}(K - \log_2 C).	
		\end{equation}
		So since $\epsilon > 0$ may be chosen arbitrarily,
		\begin{equation}
			\liminf_{N \to \infty} \frac{- \log_2(1 - \overline{T}_N)}{\log_2 N} \ge 1.	
		\end{equation}
		The same procedure, with a lower bound of $c 2^{j(1 - \epsilon)}$ on the $z_{k, j}$, gives the matching upper bound, and we conclude that in fact
		\begin{equation}
			\lim_{N \to \infty} \frac{- \log_2(1 - \overline{T}_N)}{\log_2 N} = 1,	
		\end{equation}
		which is equivalent to what we want.
	\end{proof}
\end{lemma}
%
%\begin{lemma}
%	We have
%	\begin{equation}
%		\label{lem:HittingTimeLowerBound}
%		\liminf_{N \to \infty} \frac{- \log_2(1 - T_N)}{\log_2 N} \ge 1.	
%	\end{equation}
%	\begin{proof}
%		This follows from the previous lemma and the bound $b(t_k) \le \overline{b}(t_k)$.
%	\end{proof}
%\end{lemma}

We now know how many jumps to expect before $\set{\overline{b}(t)}$ reaches a given level. We will get a lower bound on the jumps of $\set{b(t)}$ by showing that the two processes jump ``together'' at least a positive proportion of the time. This will follow easily after a technical lemma.

\begin{lemma}
	\label{lem:bOverUpper}
	Let $A = \set{k \ge 0  : \overline{b}(t_k) \le 2^k}$. There is a deterministic $p > 0$ such that 
	\begin{equation}
		\lim_{K \to \infty}\frac{\abs{A \cap [0, K]}}{K} = p\as.
	\end{equation}
	\begin{proof}
		We relate $\overline{b}(t)$ to an $\bR$-valued ergodic process, for which the statement is immediate. Let $Y_k = \sum_{j = -\infty}^k 2^{j - k}y_{k, j}$, where $y_{k, j} \sim \Exp(1)$ and these are mutually independent across all $k,\, j \in \bZ$. Consider the process $R_k = \sum_{l = -\infty}^k 2^{l - k}Y_l$, which is quite clearly stationary and ergodic under shifts of the index. The support of $R_k$ is $[0, \infty)$, so for each $\alpha$ we can say 
		\begin{equation}
			\lim_{K \to \infty}\frac{\sum_{k = 1}^K \bone\set{R_k \le \alpha}}{K} = \Prob{R_k \le \alpha} \as.	
		\end{equation}
		
		Now observe that $\floor{2^j y_{k, j}} \sim \Geom_0(2^{-j})$. We may couple our $R_k$ with $\set{\overline{b}(t)}$ such that 
		\begin{equation}
			\overline{b}(t_k) = \sum_{l = 1}^k \sqrbrac[\bigg]{1 + \sum_{j = 1}^l 2^{j} \floor{y_{l, j}}}.	
		\end{equation}
		Under this coupling, the typical difference between $2^k R_k$ and $\overline{b}(t_k)$ is quite small:
		\begin{align}
			\label{eq:bOverDiff}
			2^k R_k - \overline{b}(t_k) &= \sum_{l = -\infty}^k \sum_{j = -\infty}^l 2^{j}y_{l, j}  - \sum_{l = 1}^k \sqrbrac[\big]{1 + \sum_{j = 1}^l 2^{j} \floor{y_{l, j}}}\\
			&= \sum_{l = -\infty}^{0} \sum_{j = -\infty}^l 2^{j} y_{l, j} + \sum_{l = 0}^{k} \sum_{j = -\infty}^{0} 2^{j} y_{l, j} + \sum_{l = 1}^k \sum_{j = 1}^l(2^j y_{l, j} - \floor{2^j y_{l, j}}) - k.	
		\end{align}
		Write $S^1_k$, $S^2_k$, $S^3_k$ for the three sums. Among these, $S^1_k$ doesn't actually depend on $k$, so let us write $H^1 = S^1_k$. For $S^2_k$, an application of the Borel-Cantelli lemma to the $y_{k, j}$ shows there is a random $H^2$ such that $S^2_k \le H^2 k$ for all $k \ge 0$. Lastly, the terms of $S^3_k$ are individually bounded by $1$, so trivially $S^3_k \le k^2$. Using the triangle inequality in \eqref{eq:bOverDiff} and plugging in these bounds,
		\begin{equation}
			\abs{2^k R_k - \overline{b}(t_k)} \le H^1 + (H^2 + 1)k + k^2 \le H k^2,	
		\end{equation}
		for all $k \ge 1$ (and $H$ chosen large enough).
		
		Now to prove the lemma,
		\begin{align}
			\set{k \ge 0 : \overline{b}(t_k) \le 2^k} &\subseteq \set{k \ge 0 : 2^k R_k - H k^2 \le 2^k}\\
			&= \set{k \ge 0 : R_k \le 1 + H 2^{-k}k^2}.	
		\end{align}
		With $\epsilon > 0$ arbitrary, let $k_\epsilon$ be large enough that $H 2^{-k_\epsilon} k_\epsilon^2 < \epsilon$. Then 
		\begin{align}
			\lim_{K \to \infty} \frac{\sum_{k = 0}^K \bone\set{\overline{b}(t_k) \le 2^k}}{K} &\le \lim_{K \to \infty} \frac{\sum_{k = 0}^K \bone\set{R_k \le 1 + H 2^{-k}k^2}}{K}\\
			&= \lim_{K \to \infty} \frac{\sum_{k = k_\epsilon}^K \bone\set{R_k \le 1 + H 2^{-k}k^2}}{K}\\
			&\le \lim_{K \to \infty} \frac{\sum_{k = 0}^K \bone\set{R_k \le 1 + \epsilon}}{K}\\
			&= \Prob{R_0 \le 1 + \epsilon}.	
		\end{align}
		Taking $\epsilon \to 0$, we can replace the last bound by $\Prob{R_0 \le 1}$. Repeating this calculation with a matching lower bound leaves us with 
		\begin{equation}
			\lim_{K \to \infty} \frac{\sum_{k = 0}^K \bone\set{\overline{b}(t_k) \le 2^k}}{K} = \Prob{R_0 \le 1} > 0.
		\end{equation}
	\end{proof}
\end{lemma}

\begin{proof}[Proof of \cref{prop:BranchLowerBound}]
	Recall the expression for $p_k$ from \eqref{eq:pkDef}, which gives the probability that $b(t_k) \ne b(t_{k - 1})$ and hence there is a jump point in $(t_{k - 1}, t_k]$. Moreover, our description of the discretised process in the proof of \cref{lem:bOver} makes clear that whether this happens is a function of fresh, independent inputs whose only dependence on the trajectory thus far is through the value of $b(t_{k - 1})$. Now note that 
	\begin{equation}
		p_k = \frac{1}{2} (1 - 2^{-k})^{b(t_{k - 1})} \ge \frac{1}{2} (1 - 2^{-k})^{\overline{b}(t_{k - 1})}.	
	\end{equation}
	If $k \in A$, the set from \cref{lem:bOverUpper}, then we have a further estimate 
	\begin{equation}
		p_k \ge \frac{1}{2}e^{-(1 + \epsilon)} > 0,
	\end{equation}
	uniformly for all $k$ large enough. Then \cref{lem:bOverUpper} shows that this estimate holds a positive proportion of the $k$, and thus there must be a jump for $b$, within $(t_{k - 1}, t_k]$ for a positive proportion of the $k$. \cref{lem:HittingTimeAsympts} says that $T_N$ will typically be around $1 - N^{-1}$, so there will be of order $\log N$ such dyadic intervals before we hit $N$. The lower bound on the number of jumps follows.
\end{proof}

\subsubsection{An upper bound on jumps}
For the upper bound, we wish to show that the number of jumps of the discretised process is a good proxy for the total number in the original process, and then give an upper bound on the former. The dyadic spacing may no longer be sufficient, so we now choose $\alpha > 1$ and look at times $t_k = 1 - \alpha^{-k}$. Taking $\alpha \searrow 1$ improves the accuracy of our discretisation.

\begin{proposition}
	\label{prop:BranchUpperBound}
	There is $C_2 < 1$ such that
	\begin{equation}
		\label{eq:BranchUpperBound}
		\abs{\cB \cap [0, N]} \le C_2 \log N	
	\end{equation}
	for all $N$ large enough, almost surely.	
\end{proposition}

\begin{remark}
	We can prove \eqref{eq:BranchUpperBound} for general continuous weights if we allow $C_2 = 1$, and in particular we don't need the power of \cref{thm:SWFPPN3G}. Observe that the geodesic $\gamma_{(1, k)}$ will branch off the vertical axis at the height $0 \le j \le k$ such that $\weight{(0, j)}$ is minimised, and so $k$ will fail to be in $\cB$ unless $\weight{(0, k)}$ sets a new minimum for the weights along the vertical axis. The behaviour of such minima is discussed in \cite[Example 2.3.10]{durr}, where it is shown that up to height $N$, there will asymptotically be $\log N$ minima.
\end{remark}

\begin{lemma}
	\label{lem:XCoupling}
	There is a coupling of $\set{b(t)}$ and the process $\set{X(t)}$ defined above, such that if $\cD(b) = (d_0, d_1, \dots)$ and $\cD(X) = (d'_0, d'_1,  \dots)$ are the jump points for the processes, we have 
	\begin{equation}
		\label{eq:XCouplingGaps}
		\frac{1 - d_{k + 1}}{1 - d_k} \le \frac{1 - d'_{k + 1}}{1 - d'_k}
	\end{equation}
	for all $k \ge 0$. That is, on a logarithmic scale, the times $\set{1 - d_k}$ are farther separated than are $\set{1 - d'_k}$.
	\begin{proof}
		Recall that we define $d_0 = d'_0 = 0$, and have $b(0) = X(0) = 0$. For both processes, the first jump is chosen uniformly on $[0, 1]$, so we may let $d_1 = d'_1$. Subsequently, suppose we have sampled the first $k$ points such that \eqref{eq:XCouplingGaps} holds and wish to do the same with the $k + 1$-th. We have 
		\begin{align}
			\Prob*{\frac{1 - d_{k + 1}}{1 - d_k} \ge r} &= \Prob*{d_{k + 1} \le 1 - (1 - d_k)r}\\
			&= (1 - (1 - d_k)r)^{b(d_k)}(1 - r)\\
			&\le 1 - r\\
			&= \Prob*{\frac{1 - d'_{k + 1}}{1 - d'_k} \ge r},	
		\end{align}
		so we may sample $d_{k + 1}$ and $d'_{k + 1}$ such the relation \cref{eq:XCouplingGaps} continues to hold. 	
	\end{proof}
\end{lemma}

\begin{lemma}
	\label{lem:DiscreteResolution}
	With $t_k = 1 - \alpha^{-k}$ as above, let $D_k = \set{\cD(b) \cap [t_k, t_{k + 1}] \ne \emptyset} = \set{b(t_k) \ne b(t_{k + 1})}$. Then we have 
	\begin{equation}
		\label{eq:DiscreteResolution}
		\liminf_{k \to \infty}\frac{\sum_{l = 0}^k \bone_{D_l}}{\abs{\cD(b) \cap [0, 1 - t_k]}} \ge \frac{1}{\alpha}.
	\end{equation}
\end{lemma}

That is, in lieu of counting the actual number of zeroes up to $t_{k + 1}$, we may count only the number of times that $b(t_l) \ne b(t_{l + 1})$, $0 \le l \le k$, and in doing so we miss at most a portion of $\alpha^{-1}$ of the total.

\begin{proof}
	Consider the jumps of the process $\set{X(t)}$, which we again write as $\cD(X) = (d'_0, d'_1,  \dots)$. Suppose $d'_k$ lies in $[t_l, t_{l + 1}]$. Conditional on $d'_k$, $1 - d'_{k + 1}$ is distributed uniformly in $[0, 1 - d'_k]$. It will certainly lie beyond $t_{l + 1}$ if 
	\begin{equation}
		\label{eq:DiscreteResolutionRatio}
		\frac{1 - d'_{k + 1}}{1 - d'_k} \le \frac{1}{\alpha}. 
	\end{equation}
	This happens with probability $\alpha^{-1}$, and the ratios in \eqref{eq:DiscreteResolutionRatio} are independent across different values of $k$. This implies the statement of the lemma for the points $d'_k$.
	
	As before, write $\cD(b) = (d_0, d_1, \dots)$ for the jumps of $\set{b(t)}$. The coupling from \cref{lem:XCoupling} gives us  
	\begin{equation}
		\frac{1 - d'_{k + 1}}{1 - d'_k} \ge \frac{1}{\alpha} \Rightarrow \frac{1 - d_{k + 1}}{1 - d_k} \ge \frac{1}{\alpha}.	
	\end{equation}
	We see then that the statement must also hold for the $d_k$.
\end{proof}

At this point, we may notice that since $\Prob{b(t_{k + 1}) \ne b(t_k)} = (1 - \alpha^{-1})t_{k + 1}^{b(t_k)} \le (1 - \alpha^{-1})$, and since these jumps capture at least a portion of $\alpha^{-1}$ of the total, we have found the following asymptotic bound for the number of jumps up until time $t$:
\begin{equation}
	\alpha \cdot (1 - \alpha^{-1}) \cdot \log_\alpha (1 - t) = \frac{(\alpha - 1)\log (1 - t)}{\log \alpha}
\end{equation}
Taking $\alpha \searrow 1$ brings this upper bound to simply $\log (1 - t)$. Along with \cref{lem:HittingTimeAsympts}, would give the desired upper bound in \cref{prop:BranchUpperBound} with constant $C_2 = 1$. That we can take $C < 1$ means that the factor of $t_{k + 1}^{b(t_k)}$ which we discarded does play a role.

\begin{lemma}
	\label{lem:bUnder}
	For $\epsilon > 0$ arbitrary and $k$ large enough, we have that 
	\begin{equation}
		\label{eq:bUnder}
		\Prob{b(t_k) \ge \alpha^{k} \given b(t_k) \ne b(t_{k - 1})} \ge e^{-(1 + \epsilon)}.	
	\end{equation}
	\begin{proof}
		Recall our description of the jump from $b(t_{k - 1})$ to $b(t_k)$ from the proof of \cref{lem:bOver}. While waiting for the first output of the relevant store, we must at least wait for the first non-zero input from the sequence $I^k$. This is a $\Geom_+(\alpha^k)$ random variable which lower bounds $b(t_k)$, from which \eqref{eq:bUnder} is immediate.
	\end{proof}
\end{lemma}

\begin{proof}[Proof of \cref{prop:BranchUpperBound}]
	Let $D = \set{k \ge 0 : b(t_{k}) \ne b(t_{k - 1})}$ be those indices where we see jumps in $\set{b(t)}$. Consider also $D_{\mathrm{ind}}$, where each $k \ge 0$ is a member independently with probability $1 - \alpha^{-1}$. We couple $D$ and $D_{\mathrm{ind}}$ in an obvious way. Take an i.i.d collection $\set{U_k : k \ge 1}$ of $\Unif [0, 1]$ random variables and put $0$ in both sets. Assume that we have determined the membership for all $j \le k$, and moreover we have a sample of $\set{b(t) : t \le t_k}$ such that this process has jumps in the intervals indicated by $D$.  We then add $k + 1$ to the sets according to the following rule:
	\begin{align}
		\label{eq:DSamplingRule}
		U_{k + 1} \le 1 - \alpha^{-1} &\Rightarrow k + 1 \in D_{\mathrm{ind}},\\
		U_{k + 1} \le (1 - \alpha^{-1})t_{k + 1}^{b(t_k)} &\Rightarrow k + 1 \in D.
	\end{align}
	
	We will try to show that for the correct choice of $\alpha > 1$, we have
	\begin{equation}
		\label{eq:DUpperDensity}
		\limsup_{k \to \infty} \frac{\abs{D \cap [0, k]}}{k} \le C_2 \log \alpha,
	\end{equation}
	with $C_2 < 1$. We will compare it to $D_{\mathrm{ind}}$, for which we have trivially that 
	\begin{equation}
		\label{eq:DIndepDensity}
		\lim_{k \to \infty} \frac{\abs{D_{\mathrm{ind}} \cap [0, k]}}{k} = 1 - \alpha^{-1} = \log \alpha + o(\alpha - 1).	
	\end{equation}
	Thus if we can show that a positive proportion of elements of $D_{\mathrm{ind}}$ fail to lie in $D$, then \eqref{eq:DUpperDensity} will follow from taking $\alpha \approx 1$ in  \eqref{eq:DIndepDensity}.
	
	Rearranging \eqref{eq:DSamplingRule} gives
	\begin{equation}
		\Prob{k + 1 \in D \given k + 1 \in D_{\mathrm{ind}}} = t_{k + 1}^{b(t_{k})},
	\end{equation} 
	and also that conditional on the value of $b(t_{k})$, this event is independent of $\set{b(t): t < t_k}$. Suppose $k \in D$. By \cref{lem:bUnder},
	\begin{equation}
		\label{eq:ProbbBig}
		\Prob{b(t_k) \ge \alpha^{k}} \ge e^{-(1 + \epsilon)},	
	\end{equation}
	for $\epsilon > 0$ arbitrary and $k$ large enough. Let $Z = \min\set{l \ge 1 : k + l \in D_{\mathrm{ind}}}$, the gap until the next element of $D_{\mathrm{ind}}$. Then $Z \sim \Geom_+ (1 - \alpha^{-1})$, independent of everything else. For $\alpha$ close enough to $1$, $\Prob{Z \le \alpha / (\alpha - 1)} \ge 1 - e^{-(1 - \epsilon)}$. Using \eqref{eq:ProbbBig}, 
	\begin{align}
		\Prob{k + Z \not\in D \given k \in D} &\ge \Prob{b(t_k) \ge \alpha^k}\Prob{Z \le \alpha / (\alpha - 1)}(1 - t_{k + (\alpha / (\alpha - 1))}^{\alpha^k})\\
		&\ge e^{-(1 + \epsilon)}(1 - e^{-(1 - \epsilon)})(1 - e^{-(1 - \epsilon)\alpha^{-\alpha / (\alpha - 1)}}).	
	\end{align}
	With some calculus we may get the lower bound
	\begin{equation}
		\Prob{k + Z \not\in D \given k \in D} \ge e^{-(1 + \epsilon)}(1 - e^{-(1 - \epsilon)})(1 - e^{-(1 - \epsilon)/ 2 e}).	
	\end{equation}
	Write $p_\epsilon$ for the quantity on the right hand side. In sampling $D$, we can begin from $D_{\mathrm{ind}}$ and go through the points in increasing order, performing the test in \eqref{eq:DSamplingRule} to determine membership of $D$. If we have a point in $D$, then the subsequent point in $D_{\mathrm{ind}}$ is discarded from $D$ with probability at least $p_\epsilon$. Thus
	\begin{equation}
		\limsup_{k \to \infty}\frac{\abs{D \cap [0, k]}}{\abs{D_{\mathrm{ind}} \cap [0, k]}} \le \frac{1}{1 + p_\epsilon}.	
	\end{equation}
	Hence 
	\begin{equation}
		\label{eq:DDensityUpper}
		\limsup_{k \to \infty}\frac{\abs{D \cap [0, k]}}{k} \le \frac{1}{1 + p_\epsilon} \lim_{k \to \infty}\frac{\abs{D_{\mathrm{ind}} \cap [0, k]}}{k} = \frac{1 - \alpha^{-1}}{1 + p_\epsilon}.	
	\end{equation}
	Since $p_\epsilon > 0$ was uniform in $\alpha$, we can make $\alpha$ small enough so 
	\begin{equation}
		\label{eq:DLogApprox}
		\frac{1 - \alpha^{-1}}{1 + p_\epsilon} \le C_2 \alpha^{-1} \log \alpha
	\end{equation} holds with $C_2 < 1$.
	
	We are ready to conclude. Suppose $t_{K_N} \le T_N < t_{K_N + 1}$, where $T_N$ is the hitting time from \cref{lem:HittingTimeAsympts}. From the various definitions and \cref{lem:DiscreteResolution}, one has
	\begin{equation}
		\frac{\abs{\cB \cap [0, N]}}{\log N} = \frac{\abs[\big]{\cD(b) \cap [0, T_N]}}{-\log(1 - T_N)} \cdot \frac{-\log(1 - T_N)}{\log N} \le \frac{\alpha \abs{D \cap [0, K_N + 1]}}{K_N \log \alpha} \cdot  \frac{-\log(1 - T_N)}{\log N},
	\end{equation}
	and so
	\begin{equation}
		\limsup_{N \to \infty} \frac{\abs[\big]{\cB \cap [0, N]}}{\log N} \le \limsup_{N \to \infty} \frac{\alpha \abs[\big]{D \cap [0, K_N + 1]}}{K_N \log \alpha} \cdot \frac{-\log(1 - T_N)}{\log N} \le C_2 \cdot 1.
	\end{equation}
	Here we used \eqref{eq:DUpperDensity}, \eqref{eq:HittingTimeAsympt}, \eqref{eq:DDensityUpper} and \eqref{eq:DLogApprox}.
\end{proof}

\subsection{Highways passing through a column}
\label{ssec:HighwaysThroughBox}
The infimum in \eqref{eq:HighwaysRect} can be done directly without knowledge of the Busemann process, and indeed for any continuously distributed weights.

\begin{proposition}
	Under continuously distributed weights and for any $y \in \bZ^2_{\ge 0}$,
	\begin{equation}
		\label{eq:HighwaysLower}
		\liminf_{k \to \infty}\frac{\abs{[y, y + k e_2] \cap \Gamma_\infty}}{k} = 0 \as. 	
	\end{equation}
	\begin{proof}
		Without loss of generality we may assume $y = l e_1$. Fix $m \ge 1$. We describe a sequence of events, happening infinitely often almost surely, on which the ratio in \eqref{eq:HighwaysLower} is bounded by $1 / m$. Taking $m \to \infty$ gives the result. 
		
		Let $A_n$ be the event on which
		\begin{enumerate}
			\item $\argmin_{0 \le j \le l m n} W_{i, j} \in [i n, (i + 1) n)$ for all $0 \le i \le l - 1$,
			\item $\argmin_{0 \le j \le 2 l m n} W_{i, j} \in [(l + j) m n, (l + j + 1) m n)$ for all $0 \le i \le l - 1$,
			\item $\argmin_{0 \le j \le 2 l m n} W_{l, j} \in [l n, (l + 1) n)$.
		\end{enumerate}
		The first condition ensures that $\gamma{(l, lmn)}$ reaches column $l$ at height below $l n$. The second means that $\gamma{(l, 2 l m n)}$ initially follows the $e_2$-axis up to height at least $l m n$. The third says that $\gamma{(l + 1, 2 l m n)}$ cannot overlap with $\gamma{(l, lmn)}$ in column $l$ beyond height $(l + 1)n$. Taken together, these imply that $\gamma{(l, l m n)}$ is ``boxed in'' by  $\gamma{(l, 2 l m n)}$ and $\gamma{(l + 1, 2 l m n)}$. In particular, the points $(l, k)$ which have $(l + 1) n \le k \le 2 l m n$ belong to $\gamma(l, l m n)$ but do not form part of a semi-infinite geodesic. The situation is illustrated in \cref{fig:BoxedInEvent}. On the event $A_n$, we have 
		\begin{equation}
			\label{eq:HighwayLowerEvent}
			\frac{\abs{[y, y + 2 l m n e_2] \cap \Gamma_\infty}}{2 l m n} \le \frac{(l + 1) n}{2 l m n} \le \frac{1}{m}. 	
		\end{equation}
		
		\begin{figure}
			\centering
			\resizebox{0.35\textwidth}{!}{\includegraphics[scale=1, page=6]{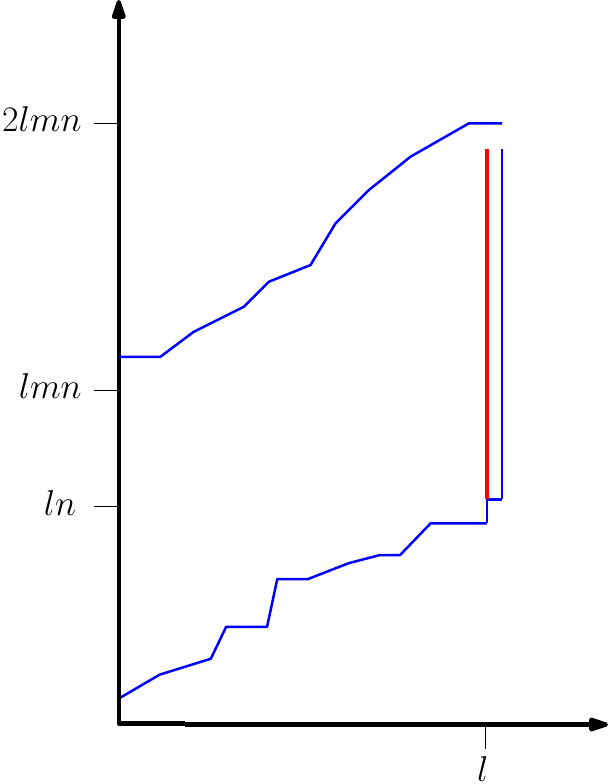}}
			\caption{An event guaranteeing a small proportion of highways in column $l$. All of the paths represent geodesics. Here the vertices along the red path are hemmed in above and below by more attractive vertices, and hence cannot be part of a semi-infinite geodesic.}
			\label{fig:BoxedInEvent}
		\end{figure}
		
		The probability of $A_n$ can be calculated easily when we recall that the position of the minimum among a finite collection of independent continuous random variables is distributed uniformly among the possible indices. So 
		\begin{equation}
			\label{eq:HighwaysLowerProb}
			\Prob{A_n} = \brac*{\frac{n}{1 + l m n}}^{l}\brac*{\frac{m n}{1 + 2 l m n}}^{l} \frac{n}{1 + 2 l m n} \ge (2 l m)^{-l} (3 l)^{-l} (3 l m)^{-1}.
		\end{equation}
		The lower bound is uniform in $n$.
		
		Another useful property of the positions of minima is that when $J \subseteq I$ are sets of indices and $\set{X_i : i \in I}$ are i.i.d, the event $\set{\argmin_{i \in I} X_i \in J}$ is independent of $\argmin_{i \in J} X_i$. This implies that $A_{n_1}$ is independent of $A_{n_2}$ whenever $2 l m n_1 < n_2$, or vice versa. Now we note that $\set{A_{(4 l m)^n} : n \ge 1}$ are mutually independent. Thus through the second Borel-Cantelli lemma and the lower bound in \eqref{eq:HighwaysLowerProb}, we have \eqref{eq:HighwayLowerEvent} infinitely often, almost surely. Then 
		\begin{equation}
			\liminf_{k \to \infty}\frac{\abs{[y, y + k e_2] \cap \Gamma_\infty}}{k} \le \frac{1}{m},	
		\end{equation}
		and we conclude by sending $m \to \infty$.
	\end{proof}
\end{proposition}

The proof in the other direction is of a similar flavour, but the events won't be so nice as to be independent. Here it is useful to have knowledge of the joint distributions of the Busemann functions to control the correlations. 

\begin{proposition}
	\label{prop:HighwaysUpper}
	Under exponential weights and for any $y \in \bZ^2_{\ge 0}$,
	\begin{equation}
		\label{eq:HighwaysUpper}
		\limsup_{k \to \infty}\frac{\abs{[y, y + k e_2] \cap \Gamma_\infty}}{k} = 1 \as. 
	\end{equation}
	\begin{proof}
		Recall our procedure for producing semi-infinite geodesics from the Busemann functions $\set{B^\xi (x, y)}$, described in \cref{rem:GeodesicConstruction}. Starting at $0$, the Busemann geodesic in direction $\xi \in U$ will leave the $e_2$-axis at height $k_1 = \min \set{k \ge 0 : B^\xi (k e_2, (k + 1)e_2) \ne 0}$. Subsequently, it will leave the $l$ column at height
		\begin{equation}
			k_l = \min \set{k \ge k_{l - 1} : B^\xi (l e_1 + k e_2, l e_1 + (k + 1)e_2) \ne 0}.    
		\end{equation}
		
		Once again we may assume that $y = l e_1$ and fix $m \ge 1$. We look at events $A'_n$, under which a semi-infinite geodesic spends a relatively long time in the $l$ column, and where the ratio in \eqref{eq:HighwaysUpper} is at least $(m - 1) / m$. Namely, let $\xi_n = \xi(l m n)$ (that is, the direction in which $\Prob{B^{\xi_n}(0, e_2) \ne 0} = (1 + l m n)^{-1}$), and let $A'_n$ be the event that 
		\begin{enumerate}
			\item There is $k'_l \in (l m n, 2 l m n]$ such that $B_2^{\xi_n} (l e_1 + k'_l e_2) \ge (l m n)^{-1}$, and $B_2^{\xi_n} (l e_1 + k e_2) = 0$ for all other $k \in [0, 2 l m n)$, $k \ne k'_l$.
			\item For $0 \le j \le l - 1$, there is $k'_j \in [j n, (j + 1) n)$ such that $W_{j, k'_j} \le (2 l^2 m n)^{-1}$, and $W_{j, k} \ge 2 (l m n)^{-1}$ for all other $k \in [0, 2 l m n)$, $k \ne k'_j$.
		\end{enumerate}
		Recall that the Busemann functions along the $j$ column may be found through the vertical update map $V$ by
		\begin{equation}
			\set{B_2^{\xi}(j e_1 - k e_2) : k \in \bZ}_{k \in \bZ} = V \brac*{\set{B_2^{\xi}((j + 1) e_1 - k e_2)}_{k \in \bZ}, (W_{j e_1 - k e_2})_{k \in \bZ}}.	
		\end{equation} 
		
		On event $A'_n$, the stores represented by $\set{B_2^{\xi_n}(j e_1 - k e_2)}_{k \in \bZ}$ are initially empty between $k = 0$ and $k = k'_l - 1$, with a mass at $k = k'_l$. In the first update, which produces the Busemann functions on the $l - 1$ column, this mass is flushed through the tandem by large services $W_{j, k}$, except for a bottleneck at $k = k'_{l - 1}$, where at least $(2 l - 1) / (2 l^2 m n)$ is left in the store here. The stores for $0 \le k \le k'_{l - 1} - 1$ end up empty. This process continues, with the stores in column $j$ empty for $0 \le k \le k'_{j} - 1$, and a mass of at least $(l + j) / (2 l^2 m n) > 0$ at $k = k'_j$. If we then carry out the procedure for finding the semi-infinite geodesic from the origin in direction $\xi_n$, we see that this geodesic will leave column $j$ at height $k'_j$, $0 \le j \le l$. Thus we have
		\begin{equation}
			\label{eq:HighwaysUpperEvent}
			\frac{\abs[\big]{[y, y + (l m n) e_2] \cap \Gamma_\infty}}{l m n} \ge \frac{l (m - 1) n}{l m n} = \frac{m - 1}{m}.	
		\end{equation}
		
		The remainder of the proof is to compute the probabilities and pairwise correlations of the $A'_n$, to verify that they are large enough and small enough, respectively, to ensure that we see \eqref{eq:HighwaysUpperEvent} infinitely often. We have chosen $\xi_n$ so that $B_2^{\xi_n}(0) \sim \Ber((1 + l m n)^{-1})\Exp(l m n)$, so 
		\begin{equation}
			\Prob{A'_n} = (l m n) \frac{e^{- l m n / (l^2 m n)}}{1 + l m n} \brac*{1 - \frac{1}{1 + l m n}}^{2 l m n - 1} \brac*{n(1 - e^{- 1 / (2 l^2 m n)})}^{l} (e^{- 2 / (l m n)})^{(2 l m n - 1) l}.	
		\end{equation}
		We can approximate each factor when $n$ is large and we arrive at 
		\begin{equation}
			\label{eq:AprimeApprox}
			\Prob{A'_n} = (2 l^2 m)^{-l}e^{-2 -4 l - 1 / l} + O(n^{-1}),	
		\end{equation}
		
		Now take $n_1,\, n_2$, with $n_2 \gg n_1$ (it will be enough that $n_2 > 4 l m n_1$). From \eqref{eq:AprimeApprox} we can estimate
		\begin{equation}
			\Prob{A'_{n_1}} \Prob{A'_{n_2}} = (2 l^2 m)^{-2 l}e^{-4 - 8l - 2 / l} + O(n_1^{-1} + n_2^{-1}).
		\end{equation}
		We would like to estimate $\Prob{A'_{n_1} \cap A'_{n_2}}$ and show that it cannot be much larger. Decompose $A'_n$ into independent events $A^{(1)}_n \cap A^{(2)}_n$, where $A^{(1)}_n$ is just condition 1 defining $A'_n$, and $A^{(2)}_n$ is just condition 2. Now 
		\begin{equation}
			\Prob{A'_{n_1} \cap A'_{n_2}} = \Prob{A^{(1)}_{n_1} \cap A^{(1)}_{n_2}} \Prob{A^{(2)}_{n_1} \cap A^{(2)}_{n_2}}.	
		\end{equation}
		In the first, we may infer from \cref{thm:MultiStat} that $\set{B^{\xi_{n_1}}_2(y + k e_2) : 0 \le k \le 2 l m n_1}$ and $\set{B^{\xi_{n_2}}(y + k e_2)_2 : k > 2 l m n_1}$ are independent. The event $A^{(1)}_{n_2}$ demands that $B^{\xi_{n_2}}_2(y + k e_2) = 0$ for all $0 \le k \le l m n_2$, which includes $2 l m n_1 + 1 \le k \le l m n_2$. Then we can use independence to bound 
		\begin{align}
			\Prob{A^{(1)}_{n_1} \cap A^{(1)}_{n_2}} &\le (l m n_1) \frac{e^{- l}}{1 + l m n_1} \brac*{1 - \frac{1}{1 + l m n_1}}^{2 l m n_1 - 1} (l m n_2) \frac{e^{- l}}{1 + l m n_2} \brac*{1 - \frac{1}{1 + l m n_2}}^{2 l m (n_2 - n_1) - 2}\\
											&= e^{-2 l}\frac{l m n_1}{1 + l m n_1}\frac{l m n_2}{1 + l m n_2}\brac*{1 - \frac{1}{1 + l m n_1}}^{2 l m n_1 - 1}\brac*{1 - \frac{1}{1 + l m n_2}}^{2 l m (n_2 - n_1) - 2}\\
											&= e^{- (2 + 2(n_2 - n_1)/n_2 +  2 l)} + O(n_1^{-1} + n_2^{-1}).	
		\end{align}
		
		The second pair is also fairly easy and the calculation may be carried out directly:
		\begin{align}
			\Prob{A^{(2)}_{n_1} \cap A^{(2)}_{n_2}} &= \brac*{n_2(1 - e^{- 1 / (2 l^2 m n_2)})}^{l} (e^{- 2 / (l m n_2)})^{2 l m (n_2 - n_1) l} \cdot\\
			&\qquad \cdot \brac*{n_1(e^{- 2 / (l m n_2)} - e^{- 1 / (2 l^2 m n_1)})}^{l} (e^{- 2 / (l m n_1)})^{(2 l m n_1) l}\\
			&= (2 l^2 m)^{-2l}e^{-4 l (1 + (n_2 - n_1) / n_2)} \brac*{\frac{1}{2 l^2 m} - \frac{n_1}{n_2}\frac{2}{2 l^2 m}}^{l} + O(n_1^{-1}n_2^{-1})\\
			&= (2 l^2 m)^{-2l} e^{-4 l (1 + (n_2 - n_1) / n_2)} \brac*{1 - \frac{2 n_1}{n_2}}^{l} + O(n_1^{-1} + n_2^{-1}).
		\end{align}
		So then our bound on the overlap is
		\begin{align}
			\Prob{A'_{n_1} \cap A'_{n_2}} \le  e^{- 2(1 + (n_2 - n_1)/n_2) - 4l (1 + (n_2 - n_1)/n_2) - 2 / l)}\brac*{1 - \frac{2 n_1}{n_2}}^{l} + O(n_1^{-1} + n_2^{-1}).
		\end{align}
		The bound on the correlation is the rather messy
		\begin{align}
			\Cov{A'_{n_1}, A'_{n_2}} &= \Prob{A'_{n_1} \cap A'_{n_2}} - \Prob{A'_{n_1}}\Prob{A'_{n_2}}\\
			&= (2 l^2 m)^{-2 l}e^{-4 - 8l - 2 / l} \brac*{1 - e^{2(1 - (n_2 - n_1)/n_2) + 4l (1 - (n_2 - n_1)/n_2)}\brac*{1 - \frac{2 n_1}{n_2}}^{l}} + \notag\\
			&\qquad + O(n_1^{-1} + n_2^{-1}). \label{eq:CovBound}
		\end{align}
				
		Choose $n_r = (4 l m)^r$. Then using our bound in \eqref{eq:CovBound}, one finds that
		\begin{equation}
		\label{eq:SummableCovs}
			\sum_{r \ne s}\Cov{A'_{n_r}, A'_{n_s}} < \infty.	
		\end{equation}
		We do not claim that the sum converges, merely that it does not diverge to positive infinity. From \eqref{eq:AprimeApprox}, we also have
		\begin{equation}
			\sum_{r = 1}^\infty \Prob{A'_{n_r}} = \infty.	
		\end{equation}
		These two facts are enough to conclude that the $A'_{n_r}$ occur infinitely often, almost surely (for example, by following the argument of Theorem 2.3.9 in \cite{durr}). Hence 
		\begin{equation}
			\frac{\abs[\big]{[y, y + k e_2] \cap \Gamma_\infty}}{k} \ge \frac{m - 1}{m}	
		\end{equation}
		for infinitely many $k$, and sending $m \to \infty$ gives the proposition.
	\end{proof}
\end{proposition}

\subsection{Convoys on the vertical axis}
\label{ssec:Convoys}
For concreteness, let us prove \Cref{thm:SWFPPCompetitionAngles} (at least, the latter two claims) for exponential rate $1$ weights. The techniques we use apply equally to Bernoulli and Bernoulli-geometric weights, except that the precise expressions and resulting constants will necessarily differ. 

Also, rather than looking at the critical angle $\xi^*$, it will be helpful to look at the critical parameter $\rho^* \in (0, \infty)$ with $B^{\xi^*}_2(0) \sim \Ber((1 + \rho^*)^{-1})\Exp(\rho^*)$. This $\rho^*$ is a decreasing function of $\xi^*$, so the various inequalities in \cref{thm:SWFPPCompetitionAngles} should be reversed.

It will be useful to define some probabilities, whose meaning will become apparent when we introduce our random walk below:
\begin{enumerate}
	\item Let $q(\rho) = (1 + \rho)^{-1}$ be the probability that $B^{\rho}_2(k e_2) \ne 0$.
	\item Let $p^{\uparrow}(\rho_1, \rho_2) = (1 - q(\rho_1))q(\rho_2) + q(\rho_1)q(\rho_2)\rho_1(\rho_1 + \rho_2)^{-1}$, $p^{\rightarrow} = (1 - q(\rho_1))(1 - q(\rho_2))$ and $p^{\downarrow}(\rho_1, \rho_2) = 1 - p^{\uparrow} - p^{\rightarrow}$ be the probabilities of going up, flat, and down, respectively, in the random walk.
	\item Let $q^{\uparrow}(\rho_1, \rho_2) = p^{\uparrow}(1 - p^{\downarrow})^{-1}$ and $q^{\rightarrow}(\rho_1, \rho_2) = p^{\rightarrow}(1 - p^{\downarrow})^{-1}$ be the probabilities when we condition on the step being up or flat.
\end{enumerate}

\begin{proposition}
	\label{prop:CompetitionBernoulli}
	For a fixed $\rho > 0$, the set $\set{k : \rho^*(k e_2) \ge \rho}$ is a Bernoulli process with density $q(\rho) = (1 + \rho)^{-1}$.
	\begin{proof}
		We have $\rho^*(k e_2) \ge \rho$ if and only if $B^{\rho}_2(k e_2) \ne 0$, which happens with probability $q(\rho)$. Moreover, the variables $\set{B^{\rho}_2(k e_2) : k \in \bZ}$ are all independent.
	\end{proof}
\end{proposition}

\begin{proposition}
	\label{prop:CompetitionRenewal}
	For fixed $0 < \rho_1 < \rho_2$, the set $C(\rho_1, \rho_2) = \set{k : \rho_1 \le \rho^*(k e_2) \le \rho_2}$ is a renewal process whose holding times may be represented in terms of stopping times in a particular random walk. Specifically, consider $S_k = \sum_{j = 1}^{k}X_j$, where $X_j = X^2_j - X^1_j$, and 
	\begin{equation}
		X^1_j \sim \Ber(q(\rho_1))\Exp(\rho_1),\quad X^2_j \sim \Ber(q(\rho_2))\Exp(\rho_2)
	\end{equation}
	are i.i.d. Define $\tau = \min\set{k \ge 1 : S_k < S_{k - 1} \le \min_{0 \le j \le k - 1} S_k}$, and let $\tau^1,\, \tau^2, \dots$ be independent copies of $\tau$. Let $p^{\downarrow \downarrow}$ be the probability defined below and take $Z \sim \Geom_+(p^{\downarrow \downarrow})$. Then the increments between elements of $C(\rho_1, \rho_2)$ are independent and equal in distribution to $\sum_{n = 1}^{Z}\tau^n$.
	\begin{proof}
		We look at our multi-line representation and identify this walk $S_k$. Suppose $0 \in C(\rho_1, \rho_2)$, and let $I^1,\, I^2$ be the inputs into our multi-line process. The distribution of the sequence $\set{(B^{\rho_1}_2(k e_2), B^{\rho_2}_2(k e_2)) : k \in \bZ}$ is that of $(I^1, H(I^2, I^1))$. Write $\widetilde{I}^2 = H(I^2, I^1)$. That $0$ belongs to $C(\rho_1, \rho_2)$ means that $I^1_0 > 0$ and $\widetilde{I}^2_0 = 0$. The store had nonzero service at time $0$ but zero output, meaning that the store was empty before this time and that $I^2_0 = 0$. 
		
		We will have another element of $C(\rho_1, \rho_2)$ the next time the store is empty and subsequently receives zero input and nonzero service. Consider $S_k = \sum_{j = 0}^k (I^2_j - I^1_j)$. An increase in $S_k$ means that the input was greater than the available service and thus the quantity in the store has increased. A decrease in $S_k$ means that the store has decreased (possibly to zero). The store is empty exactly when $S_k$ attains its (weak) running minimum. If $S_{k - 1}$ is a running minimum, then the store is empty at the beginning of time step $k$. If the value of the walk decreases at $S_k$, then the store had nonzero service and possibly zero input. Conditional on a down step, the probability the input was zero is
		\begin{equation}
			p^{\downarrow \downarrow} = \frac{(1 - q(\rho_2))q(\rho_1)}{p^\downarrow} = \frac{\rho_1 + \rho_2}{1 + \rho_1 + \rho_2}.
		\end{equation}
		In this case we have found our next element of $C(\rho_1, \rho_2)$. Otherwise, the store is again empty and we begin the wait once more.
	\end{proof}
\end{proposition}

To make the dependence on the parameters more explicit, denote by $\tau(\rho_1, \rho_2)$ the stopping time in \cref{prop:CompetitionRenewal}. So long as $\rho_1 < \rho_2$, $\tau(\rho_1, \rho_2)$ has finite mean and moreover, exponential decay (as the return time of a biased random walk). From ergodicity considerations, we know that the density of $C(\rho_1, \rho_2)$ must be $q(\rho_1) - q(\rho_2)$, and so we must have 
\begin{equation}
	\label{eq:HoldingFiniteMean}
	\Ex{\tau(\rho_1, \rho_2)} = \frac{(1 + \rho_1)(1 + \rho_2)}{\rho_2 - \rho_1}.
\end{equation}
We could also derive this through Wald's identities, and we will take this approach later on.

The more interesting case is when we look at $C(\rho) = \set{k : \rho^*(k e_2) = \rho}$. Since $\rho^*$ is continuously distributed, this set must be empty almost surely. But if we condition on it being nonempty, then we get the same description as above, now with $\tau(\rho_1, \rho_2)$ replaced by $\tau(\rho) = \tau(\rho, \rho)$.
\begin{proposition}
	\label{prop:ConvoyRenewal}
	Let $\rho = \rho^*(0)$. Let $\tau^1, \tau^2, \dots$ be independent copies of $\tau(\rho)$ and let $Z \sim \Geom_+(p^{\downarrow \downarrow})$, where now $p^{\downarrow \downarrow} = 2\rho (1 + 2\rho)^{-1}$. Then the convoy for $\rho$, $C(\rho)$, is a renewal process whose holding times are equal in distribution to $\sum_{n = 1}^{Z}\tau^n$.
	\begin{proof}
		We have in mind the same situation as in \cref{prop:CompetitionRenewal}, now taking $\rho_2 \searrow \rho_1$. We are left with two sets of independent inputs $I^1,\,I^2$, both with $\Ber(q(\rho))\Exp(\rho)$ marginals. In general there may be difficulties defining $H(I^2, I^1)$, but our conditioning means that we may assume the store is empty before and after service at time $0$. Now we repeat the argument from \cref{prop:CompetitionRenewal} linking elements of our set to $\tau(\rho)$.
	\end{proof}
\end{proposition}

We turn to finding the number of vertices in our convey below a given height, or equivalently the number of renewals before a given time. If the mean of the holding time were finite then this would just be an application of the renewal theorem. However, the mean here is infinite, as seen from taking $\rho_1 = \rho_2$ in \eqref{eq:HoldingFiniteMean}, or by ergodicity considerations. We use an extension of the standard renewal theorem to heavy-tailed distributions:
\begin{theorem}[Theorem 5 of \cite{eric-renewal}]
	\label{thm:RenewalTheorem}
	Let $F$ be the distribution for our holding times and suppose $1 - F(x) = L x^{- \alpha} + o(x^{- \alpha})$, for some $0 < \alpha < 1$. Let $U(n)$ be the number of renewal times in $[0, n]$ (where for simplicity we may condition on a renewal at time $0$). Then
	\begin{equation}
		\lim_{n \to \infty}\frac{U(n)}{n^\alpha} = \frac{1 - \alpha}{L \Gamma(\alpha + 1)\Gamma(2 - \alpha)} \as.
	\end{equation}
\end{theorem}
The full statement found in \cite{eric-renewal} is quite a bit more general, but this version will suffice. The theorem suggests that we should study the tails of our holding time distribution, but it will be easier to instead study tails of hitting times of our walk and relate the two.

\begin{lemma}
	\label{lem:tauDecomp}
	The stopping time $\tau$ admits a further decomposition into a sum of geometrically many hitting times. Let $\sigma = \min\set{k \ge 1 : S_k \le 0 \given S_1 \ge 0}$, or the first time the walk reaches zero conditioned on taking a non-negative first step. Then if $\sigma^1, \sigma^2, \dots$ are independent copies of $\sigma$, we have $\tau \disteq 1 + \sum_{j = 1}^{Y} \sigma^j$, where $Y \sim \Geom_0(p^{\downarrow})$.
	\begin{proof}
		Recall that $\tau$ records the first time that we reach a strict running minimum, where the previous time is itself a running minimum. Initially the walk is at a running minimum, so we will get $\tau = 1$ if the first step is down, with probability $p^{\downarrow}$. Otherwise, the run takes a non-negative step, and we must wait for the walk to reach its prior low point. This time is $\sigma_1$. After time $\sigma_1$, we are again at a running minimum, and will have $\tau = 1 + \sigma_1$ if we subsequently take a down step, and this situation happens with probability $(1 - p^{\downarrow})p^{\downarrow}$. Iterating the argument gives the claimed decomposition.
	\end{proof}
\end{lemma}

Random walks with Bernoulli or exponential increments are special in that their hitting times are particularly amenable to analysis through Wald's identities. Still, the time we care about, the sum in \cref{prop:CompetitionRenewal}, is rather complicated. It is necessary to first understand $\sigma$ and go through the decomposition from the last lemma.
\begin{lemma}
	\label{lem:sigmaCHF}
	Suppose $\rho_1 = \rho_2 = \rho$ and set $\beta = \rho^2 (1 + \rho)^{-2}$. Let $\widetilde{\sigma} = (\sigma - 1)\given \set{\sigma \ge 1}$, just the value $\sigma - 1$ conditioned on having it being non-negative. Then $\widetilde{\sigma}$ has characteristic function 
	\begin{equation}
		\label{eq:sigmaCHF}
		\widetilde{\phi}(\theta) = \Ex{e^{i \theta \widetilde{\sigma}}} = \frac{\sqrt{1 - \beta e^{i \theta}} - \sqrt{1 - e^{i \theta}}}{\sqrt{1 - \beta e^{i \theta}} + \sqrt{1 - e^{i \theta}}}.
	\end{equation}
	Further, the generating function of $\widetilde{p}_k = \Prob{\widetilde{\sigma} = k}$ is
	\begin{equation}
		\label{eq:sigmaOGF}
		\widetilde{f}(z) = \sum_{k = 0}^{\infty}\widetilde{p}_k z^k = \frac{\sqrt{1 - \beta z} - \sqrt{1 - z}}{\sqrt{1 - \beta z} + \sqrt{1 - z}}.
	\end{equation}
	\begin{proof}
		The characteristic function of our steps $X^2 - X^1$ is easily seen to be
		\begin{align}
			M(t) &= \Ex{e^{i t (X^2 - X^1)}}\\ 
			&= \brac*{(1 - q(\rho)) + \frac{q(\rho)}{1 - i t / \rho}}\brac*{(1 - q(\rho)) + \frac{q(\rho)}{1 + i t / \rho}}\\ 
			&= \frac{\rho^2((1 + \rho)^2 + t^2)}{(1 + \rho^2)(\rho^2 + t^2)}.
		\end{align}
		Consider $\sigma^a = \min\set{k \ge 0 : S_k \le -a}$, the hitting time for $-a < 0$. This $\sigma^a$ is equal in distribution to $\widetilde{\sigma}$ conditioned on $S_1 = a$. Wald's third identity applied here tells us that 
		\begin{equation}
			1 = \Ex*{M(t)^{- \sigma^a} e^{i t S_{\sigma^a}} }.
		\end{equation}
		Memorylessness means that $S_{\sigma^a}\sim - a - \Exp(\rho)$ and is independent of $\sigma^a$. So we factor the expectation and get 
		\begin{equation}
			1 = \Ex*{M(t)^{- \sigma^a}} e^{- a i t} \Ex{e^{i t (S_{\sigma^a} + a) }}.
		\end{equation}
		Plugging in the characteristic function for an exponential and rearranging, this becomes
		\begin{equation}
			e^{a i t}(1 + i t / \rho) = \Ex*{M(t)^{- \sigma^a}}.
		\end{equation}
		Now with $a = S_1$ and taking an expectation over $a$ (we condition it to be positive):
		\begin{equation}
			\frac{1 + i t / \rho}{1 - i t / \rho} = \Ex*{M(t)^{- \widetilde{\sigma}}}.
		\end{equation}
		
		Now let $t(\theta)$ solve $M(t) = e^{- i \theta}$. We would have 
		\begin{equation}
			\label{eq:CharPreSub}
			\Ex*{e^{i \theta \widetilde{\sigma}}}. = \frac{1 + i t(\theta) / \rho}{1 - i t(\theta) / \rho}. 
		\end{equation}
		One can solve this equation to find 
		\begin{equation}
			t(\theta) = i \frac{\rho(1 + \rho)\sqrt{1 - e^{i \theta}}}{\sqrt{(1 + \rho)^2 -  \rho^2 e^{i \theta}}}.
		\end{equation}
		Now plugging in to \eqref{eq:CharPreSub} gives the desired expression \eqref{eq:sigmaCHF}.
		
		For the second claim, observe that $\widetilde{\phi}(\theta) = \widetilde{f}(e^{i\theta})$, and $\widetilde{\phi}$ may be seen as the restriction of $\widetilde{f}$ to the unit circle. Moreover, $\widetilde{f}$ is analytic on the unit disc. The inversion formula for characteristic functions reduces here to the statement that Fourier coefficients of $\widetilde{\phi}$ are the probabilities $\widetilde{p}_k$. The Fourier coefficients of $\widetilde{\phi}$ are in turn the coefficients of the power series expansion of its extension $\widetilde{f}$ about the origin. This is the substance of \eqref{eq:sigmaOGF}.
	\end{proof} 
\end{lemma}

To recover the distribution of $\widetilde{\sigma}$, and so of $\sigma$ itself, we would like to find the power series expansion of $\widetilde{f}$, or at least understand the asymptotic behaviour of its coefficients. Fortunately, we can simply borrow standard tools from the theory of analytic combinatorics, which exist to answer such questions.
\begin{proposition}
	\label{prop:SigmaProbs}
	The probabilities $\widetilde{p}_k$ have
	\begin{equation}
		\widetilde{p}_k = k^{-3/2}\frac{1}{\sqrt{\pi(1 - \beta)}} + o(k^{-3/2}).
	\end{equation}
	Hence $p_k = \Prob{\sigma = k}$ has $p_0 = q^{\rightarrow}$, and for $k \ge 1$,
	\begin{equation}
		p_k = k^{-3/2}\frac{q^{\uparrow}}{\sqrt{\pi(1 - \beta)}} + o(k^{-3/2}).
	\end{equation}
	\begin{proof}
		The function of interest, $\widetilde{f}$, may be written in a less symmetric but more approachable form as
		\begin{equation}
			\widetilde{f}(z) = \frac{1}{(1 - \beta)z}\brac[\big]{2 - (1 + \beta) z - 2 \sqrt{1 - \beta z}\sqrt{1 - z}}.
		\end{equation}
		Dividing by $z$ just shifts the series coefficients downward, and the $2 - (1 + \beta)z$ term only affects the first two coefficients. For the asymptotic behaviour of the coefficients, it is enough to look at
		\begin{equation}
			-\frac{2}{1 - \beta}\sqrt{1 - \beta z}\sqrt{1 - z}.
		\end{equation}
		We look to \cite[Theorem VII.8]{flaj-sedg-ac}, which states that it is enough to understand the function near $z = 1$, then \cite[Theorem VI.1]{flaj-sedg-ac} to see that the coefficients are approximately
		\begin{equation}
			-\frac{2}{1 - \beta} \cdot \sqrt{1 - \beta} \cdot \frac{-k^{-3/2}}{\sqrt{\pi}} = k^{-3/2}\frac{1}{\sqrt{\pi(1 - \beta)}}.
		\end{equation}
		In fact, the theorems we quoted may be used to find the asymptotic expansion up to arbitrary order, but we will not use anything beyond the leading term.
	\end{proof}
\end{proposition}
	
\begin{remark}
	The order of the $p_k$ being $k^{-3/2}$ can be expected from what is known generally about return times of random walks \cite[Theorem 5.1.7]{lawler-limic-rw}. It is only when our steps are exponential, geometric, Bernoulli, or some combination that we may solve for the generating function and compute the coefficient of the leading order term.
\end{remark}
	
\begin{corollary}
\label{cor:sigmaRenewal}
Consider a renewal process whose holding times have the same distribution as $\sigma$, and let $U^s(n)$ be the number of renewals in $[0, n]$. Then
\begin{equation}
	\lim_{n \to \infty}\frac{U^s(n)}{\sqrt{n}} = \frac{\sqrt{1 - \beta}}{q^{\uparrow} \sqrt{\pi}}\as.
\end{equation}
\begin{proof}
	Using \cref{prop:SigmaProbs}, we can see that the tails of $\sigma$ have 
	\begin{equation}
		\Prob{\sigma \ge n} = \sum_{k = n}^{\infty}p_k = n^{-1/2}\frac{2 q^{\uparrow}}{\sqrt{\pi(1 - \beta)}} + o(n^{-1/2}).
		\end{equation}
		We may then apply the strong renewal theorem as stated in \cref{thm:RenewalTheorem}, with $\alpha = 1 / 2$ and $L = 2 q^{\uparrow} (\pi(1 - \beta))^{-1/2}$.
	\end{proof}
\end{corollary}
	
\begin{corollary}
	\label{cor:tauRenewal}
	Consider a renewal process whose holding times have the same distribution as $\tau$, and let $U^t(n)$ be the number of renewals in $[0, n]$. Then
	\begin{equation}
		\lim_{n \to \infty}\frac{U^t(n)}{\sqrt{n}} = \frac{p^{\downarrow}\sqrt{1 - \beta}}{(1 - p^{\downarrow})q^{\uparrow} \sqrt{\pi}}\as.
	\end{equation}
	\begin{proof}
		Consider instead the renewal process with holding times distributed as $\widetilde{\tau} = \tau - 1$, and call the corresponding counting function $\widetilde{U}(n)$. Note that since $\widetilde{\tau}$ may be zero, we can have multiple renewal times in the same instant. The number of renewals at a given renewal time is geometrically distributed (this is a general fact). Since by \cref{lem:tauDecomp} $\widetilde{\tau}$ is just the sum of $\Geom_0(p^{\downarrow})$ many copies of $\sigma$, we can sample $\set{\widetilde{U}(n)}$ as follows: sample $\set{U^s(n)}$, and to each renewal point $k \in \cD(U^s)$, attach a $\Geom_0(p^{\downarrow})$ variable $Z_k$. Then $\widetilde{U}$ can be taken to be $\sum_{\ov{k \le n}{k \in \cD(U^s)}} Z_k$. 
		
		It is easy to see from \cref{cor:sigmaRenewal} that
		\begin{equation}
			\label{eq:UTildeRenewal}
			\lim_{n \to \infty}\frac{\widetilde{U}(n)}{\sqrt{n}} = \Ex{Z} \lim_{n \to \infty} \frac{U^s(n)}{\sqrt{n}} = \frac{p^{\downarrow}\sqrt{1 - \beta}}{(1 - p^{\downarrow})q^{\uparrow} \sqrt{\pi}}.
		\end{equation}
		
		The process we care about, $U^t(n)$, can be found by adding a unit space between renewals of $\widetilde{U}$. So they may be coupled such that $U^t(n + \widetilde{U}(n)) = \widetilde{U}(n)$, and
		\begin{equation}
			\widetilde{U}(n - \widetilde{U}(n)) \le U^t(n) \le \widetilde{U}(n).
		\end{equation}
		Dividing by $\sqrt{n}$ and taking the limit of the left and right with \eqref{eq:UTildeRenewal} gives the result.
	\end{proof}
\end{corollary}

We may finally prove the asymptotic result for the convoys.

\begin{proposition}
	\label{prop:ConvoyDensityConstant}
	Let $\rho = \rho^*(0)$ and let $C(\rho) = \set{k : \rho^*(k e_2) = \rho}$. Then 
	\begin{equation}
		\label{eq:ConvoyDensity}
		\lim_{n \to \infty}\frac{\abs{C(\rho) \cap [0, n]}}{\sqrt{n}} = \frac{p^{\downarrow}\sqrt{1 - \beta}}{p^{\downarrow \downarrow}(1 - p^{\downarrow})q^{\uparrow} \sqrt{\pi}}\as,
	\end{equation}
	recalling that we defined $p^{\downarrow \downarrow} = 2\rho (1 + 2\rho)^{-1}$.
	\begin{proof}
		The idea of the proof is the same as for \cref{cor:tauRenewal}. \cref{prop:ConvoyRenewal} says that the holding times are sums of $\Geom_+(p^{\downarrow \downarrow})$ many copies of $\tau$, and so we may sample $C(\rho)$ from $U^t$ by keeping the renewal points of $U^t$ independently with probability $1 - p^{\downarrow \downarrow}$. From this it is immediate that 
		\begin{equation}
			\lim_{n \to \infty}\frac{\abs{C(\rho) \cap [0, n]}}{\sqrt{n}} = \frac{1}{p^{\downarrow \downarrow}}\lim_{n \to \infty}\frac{U^t(n)}{\sqrt{n}} = \frac{p^{\downarrow}\sqrt{1 - \beta}}{p^{\downarrow \downarrow}(1 - p^{\downarrow})q^{\uparrow} \sqrt{\pi}}.
		\end{equation}
	\end{proof}
\end{proposition}

Unwrapping the definitions of the various quantities on the right of \eqref{eq:ConvoyDensity} leaves us with
\begin{equation}
	\frac{1 + 2 \rho}{2 \rho \sqrt{\pi(1 + \rho^2)}}.
\end{equation}

\begin{remark}
	Using the decompositions in \cref{prop:ConvoyRenewal,cor:tauRenewal} and the expression for the characteristic function of $\sigma$ in \cref{lem:sigmaCHF}, one can express the characteristic functions of $\tau$ and the holding times of $C(\rho)$ as an infinite sum over powers of $\widetilde{\phi}$. Mathematica does produce closed forms for these series, but they aren't especially simple and it is easier to argue more indirectly, as we have done.
\end{remark}

\section*{}
\subsubsection*{Acknowledgements} The author would like to thank his advisor, Timo Sepp\"al\"ainen,  for suggesting many of the problems considered here, for a great deal of encouragement and patience throughout the project, and for insightful comments on the early drafts of this paper. He would also like to thank Ofer Busani, Malte Hassler and Diego Rojas La Luz for helpful discussions. The author was partially supported by Timo Sepp\"al\"ainen through National Science Foundation grant DMS-2152362.

\appendix
%\addappheadtotoc
\crefalias{section}{app}

\section{Uniform convergence to the limit shape}
\label{app:UniformShape}
Uniform convergence to the limit shape will be needed later in establishing the Busemann limits. We begin by restating \cite[Theorem 2.4]{geor-rass-sepp-16} for our case, which establishes the existence of the limit shape very generally.

\begin{theorem}
	\label{thm:BasicLimitShape}
	There exists a deterministic, convex function $\ell: \bR^+_2 \to \bR$ such that almost surely, it holds for each $(x_1, x_2) \in  \bR^+_2$ simultaneously that 
	\begin{equation}
		\ell(x_1, x_2) = \lim_{n \to \infty}n^{-1}L(\floor{n x_1}, \floor{n x_2}).
	\end{equation}
	These limits also hold in the $L^1$ sense.
\end{theorem}

The rest of the section consists largely of proving the following improved convergence, which is a small extension of Theorem 5.1 in \cite{mart-limit-shape}. 

\begin{theorem}
	\label{thm:UniformLimitShape}
	Almost surely
	\begin{equation}
		\lim_{n \to \infty}\frac{1}{n}\max_{x \in \bZ_{+}^{2}, \abs{x}_{1} = n}\abs{L(x) - \ell(x)} = 0.
	\end{equation}
\end{theorem}

The argument in \cite{mart-limit-shape} deals with vertex weights, but can be carried out in the same fashion for edge weights. It relies essentially on only two facts: that in a directed model the number of steps (in each direction) in a geodesic is deterministic; and that one has a powerful concentration inequality for passage times of bounded weights \cite[Theorem 8.1.1]{tala-concentration}. As we are only interested in the planar case we specialise the proof to dimension $d = 2$, but in principle the same steps can be carried out for any $d \ge 1$.

The proof goes through a number of lemmas, which we restate for our edge-weight FPP setup. The arguments for the most part require very little modification. The only subtlety is in Martin's Lemma 3.1, which is the statement of the concentration inequality which the proof ultimately rests on. The statement in our context is:

\begin{lemma}
	\label{lem:MartinConcentration}
	Let $X_i$, $i \in I$, be a finite collection of independent random variables taking values in $[0, L]^{d}$ and write $X_{i, k}$ for the $k$-th component of $X_i$. Let $\cC$ be a set of subsets of $I \times [d]$, such that 
	\begin{equation}
		\max_{C \in \cC} \abs{C}  \le R,
	\end{equation}
	and additionally such that if $C \in \cC$ and $(i, k_1), (i, k_2) \in C$, then $k_1 = k_2$. Set
	\begin{equation}
		Z = \max_{C \in \cC} \sum_{(i, k) \in C} X_{i, k}.
	\end{equation}
	Then for any $u > 0$, 
	\begin{equation}
		\Prob{\abs{Z - \Ex{Z}} > u} \le \exp\brac*{-\frac{u^2}{64RL^2} + 64}.
	\end{equation}
\end{lemma}

The proof in \cite[Lemma 5.1]{martin-lattice-animals} is a direct application of an inequality due to Talagrand \cite[Theorem 8.1.1]{tala-concentration} on the concentration of passage times, which itself is a quick consequence of his isoperimetric inequality, as stated in \cite[Theorem 4.1.1]{tala-concentration}. An appropriate vector-valued version of the isoperimetric inequality, such as the one found in Section 7.6 of \cite{alon-spen}, gives the corresponding vector-valued version of his passage time concentration.

\begin{theorem}
	Let $(X_{i})_{i \le N}$ be a collection of independent random variables with $X_{i} \in [0, 1]^{d_{i}}$, and for $1 \le k \le d_i$ write $X_{i, k}$ for the $k$-th component of $X_i$. Consider a family $\cF$ of $N$-tuples of pairs $(\alpha_{i}, k_{i})_{i \le N}$, where $\alpha_{i} \ge 0$ is a non-negative coefficient and $k_i$ is an index with $1 \le k_{i} \le d_{i}$. Set $\sigma = \sup_{(\alpha, k) \in \cF}\norm{\alpha}_{2}$. Define a maximum over this family
	\begin{equation}
		Z = \sup_{(\alpha, k) \in \cF}\sum_{i \le N}\alpha_{i}X_{i, k_i}.
	\end{equation}
	If $M$ is the median of $Z$, then for all $u > 0$ we have
	\begin{equation}
		\Prob{\abs{Z - M} \ge u} \le 4\exp\brac*{- \frac{u^{2}}{4 \sigma^{2}}}.
	\end{equation}
\end{theorem}

Now \cref{lem:MartinConcentration} can be proved in precisely the same way as in \cite[Lemma 5.1]{martin-lattice-animals}.

One can use the strong control in the bounded case to prove \cref{thm:UniformLimitShape} for these weights. We follow Martin and begin with continuity of the limit shape.

\begin{lemma}
	\label{lem:MartinUniformBoundary}
	Let $X_i$, $i \in I$, be a finite collection of independent random variables taking values in $[0, L]^{d}$ and take $\epsilon > 0$. Then there is $\delta > 0$ such that if $x \in \bR^2_+$ and $\norm{x} \le 1$, with $x_1 = 0$, then
	\begin{equation}
		\abs{\ell(x + h e_1) - \ell(x)} < \epsilon
	\end{equation}
	for all $0 \le h \le \delta$.
\end{lemma}

\begin{lemma}
	\label{lem:MartinContinuousShape}
	Suppose $\abs{W_i(x)} < M$ for some $M > 0$. Then $\ell$ is continuous on $\bR_+^2$.
\end{lemma}

The proofs of \cite[Lemmas 3.2, 3.3]{mart-limit-shape} go through word-for-word\footnote{Martin uses $T$ for the passage times in place of $L$ and has weights $X(v)$ rather than $W(e)$.}. 

We now give \cref{thm:UniformLimitShape} for bounded weights. This is a combination of Lemmas 5.3, 5.4 in \cite{mart-limit-shape}. The proofs go through word-for-word\footnote{Again, with the passage times $T$ and weights $X(v)$}.

\begin{lemma}
	\label{lem:MartinBoundedPassageToShape}
	Suppose $\abs{W_i(x)} < M$ for some $M > 0$, and let $\epsilon > 0$ be given. Then almost surely, we have for all but finitely many $z \in \bZ^2_+$ that
	\begin{equation}
		\abs{L(z) - \ell(z)} \le \epsilon \norm{z}.
	\end{equation}
\end{lemma}

Having proved these lemmas for bounded weights, Martin proceeds to generalise to unbounded distributions satisfying a certain decay assumption. Namely, we need $\int_0^{\infty} (1 - F(s))^{1 / 2}\, ds < \infty$, where $F$ is the distribution of the vertex weights in the LPP model Martin considers. After taking negatives to bring us into FPP, the condition becomes $\int_{-\infty}^{0} F(s)^{1 / 2}\, ds < \infty$. This condition is automatically implied by the existence of $2 + \epsilon$ moments.

As our inequalities do not have to be especially sharp, we bound our edge-weight model between vertex-weight models and apply Martin's results to these. Given a directed path $\pi = (x = \pi_{0}, \pi_{1}, \dots, \pi_{n} = y)$, define upper and lower vertex passage times
\begin{align}
	\overline{L}(\pi) &= \sum_{i = 0}^{n - 1} \brac[\big]{W(\pi_i; 1) \vee W(\pi_i; 2)},\\
	\underline{L}(\pi) &= \sum_{i = 0}^{n - 1} \brac[\big]{W(\pi_i; 1) \wedge W(\pi_i; 2)}.
\end{align}
We define $\overline{L}(x, y)$ and $\underline{L}(x, y)$ as minimums over admissible paths, as before.

\begin{lemma}
	\label{lem:VertexBound}
	In the notation above,
	\begin{equation}
		\underline{L}(x, y) \le L(x, y) \le \overline{L}(x, y).
	\end{equation}    
	\begin{proof}
		Let $\pi$ be a minimising path for $\overline{L}(x, y)$. Then 
		\begin{equation}
			\overline{L}(\pi) = \sum_{i = 0}^{n - 1} \brac[\big]{W(\pi_i; 1) \vee W(\pi_i; 2)} \ge \sum_{i = 0}^{n - 1} W(\pi_i, \pi_{i + 1}) = L(\pi),
		\end{equation} 
		and in turn $L(\pi) \ge L(x, y)$. Similarly for $\underline{L}(x, y)$.
	\end{proof}
\end{lemma}

Write $\overline{F}$ and $\underline{F}$ for the distributions of $W(0; 1) \vee W(0; 2)$ and $W(0; 1) \wedge W(0; 2)$, respectively. These again have finite $2 + \epsilon$ moments.

\begin{lemma}
	\label{lem:MartinPTBounds}
	There exists $c$ (independent of the weight distribution) such that:
	\begin{thmlist}
		\item for all $z \in \bZ_+^2$,
		\begin{equation}
			\Ex{L(z)} \ge -c \norm{z} \int_{-\infty}^0 \underline{F}(s)^{1 / 2}\, ds.
		\end{equation}
		\label[lem]{lem:MartinPTBounds1}
		\item with probability 1,
		\begin{equation}
			\label{eq:MartinPTBounds2}
			\liminf_{n \to \infty}\frac{1}{n}\min_{\norm{z}_1 \le n} L(z) \ge -c \int_{-\infty}^0 \underline{F}(s)^{1 / 2}\, ds.
		\end{equation}
		\label[lem]{lem:MartinPTBounds2}
		\item for all $x \in \bR_+^2$, 
		\begin{equation}
			\sum_{i = 1}^{2}\inner{x, e_i}\Ex{W(0; i)} \ge \ell(x) \ge -c \norm{x} \int_{-\infty}^0 \underline{F}(s)^{1 / 2}\, ds.
		\end{equation}
		\label[lem]{lem:MartinPTBounds3}
	\end{thmlist}
	\begin{proof}
		The lower bounds all come from taking negatives in \cite[Lemma 3.5]{mart-limit-shape} and applying the statement to the lower vertex passage times $\underline{L}$. For the upper bound in (iii) we need only follow Martin's calculation. Let $\widetilde{\pi} \in \Pi(z)$ be some path connecting $z$ to the origin. Then
		\begin{align}
			\Ex{L(z)} &= \bE \min_{\pi} \sum_{i = 0}^{\norm{z} - 1}W(\pi_i, \pi_{i + 1})\\
			&\le \bE{\sum_{i = 0}^{\norm{z} - 1}W(\widetilde{\pi}_i, \widetilde{\pi}_{i + 1})}\\
			&= \sum_{\widetilde{\pi}_{i + 1} - \widetilde{\pi}(i) = e_1}\bE{W(0; 1)} + \sum_{\widetilde{\pi}_{i + 1} - \widetilde{\pi}(i) = e_2}\bE{W(0; 2)}\\
			&= \inner{z, e_1}\bE{W(0; 1)} + \inner{z, e_2}\bE{W(0; 2)}.
		\end{align}
		Dividing by $n$ and taking limits, we get the desired expression.
	\end{proof}
\end{lemma}

Take $M \ge 0$ and consider the environment of truncated weights $\set{W^{(M)}(x; i)}$, where $W^{(M)}(x; i) = (W(x; i) \vee (-M)) \wedge M$. Let $L^{(M)}, \ell^{(M)}$ be the passage times and limit shape under the truncated weights. The next lemma quantifies the rate at which $\ell^{(M)} \to \ell$ as $M \to \infty$.

\begin{lemma}
	\label{lem:MartinTruncationApprox}
	For any $x \in \bR^2_+$, 
	\begin{equation}
		\ell^{(M)}(x) - c \norm{x} \int_{-\infty}^{-M} \underline{F}(s)^{1 / 2}\, ds \le \ell(x) \le \ell^{(M)}(x) + \norm{x}\int_{M}^{\infty}1 - \overline{F}(s)\, ds.
	\end{equation}
	In particular, for any $R > 0$,
	\begin{equation}
		\sup_{x \in \bR^2_+, \norm{x} \le R}\abs{\ell(x) - \ell^{(M)}(x)}\underset{{M \to \infty}}{\longrightarrow} 0.
	\end{equation}
	
	\begin{proof}
		The argument is largely identical to \cite[Lemma 3.6]{mart-limit-shape}. For lower bound, take $x \in \bR^2_+$. We have
		\begin{align*}
			\ell(x) - \ell^{(M)}(x) &= \lim_{n \to \infty}n^{-1}\Ex{L(\floor{n x})} - \lim_{n \to \infty}n^{-1}\Ex{L^{(M)}(\floor{n x})}\\
			&= \lim_{n \to \infty}n^{-1}\Ex{\min_{\pi \in \Pi(0, \floor{n x})}\sum_{e \in \pi}W(e) - \min_{\pi \in \Pi(0, \floor{n x})}\sum_{e \in \pi}W^{(M)}(e)}\\
			&\ge \lim_{n \to \infty}n^{-1}\Ex {\min_{\pi \in \Pi(0, \floor{n x})} \sum_{e \in \pi} W(e) - W^{(M)}(e)}\\
			&\ge \lim_{n \to \infty}n^{-1}\Ex {\min_{\pi \in \Pi(0, \floor{n x})} \sum_{e \in \pi} (W(e) - M)_-}\\
			&\ge -c \norm{x} \int_{-\infty}^{-M} \underline{F}(s)^{1 / 2}\, ds.
		\end{align*}
		The last inequality comes from applying \cref{lem:MartinPTBounds3} to the weights $\set{(W(e) - M)_-}$.
		
		We need an auxiliary calculation before the upper bound. For the other side, take $z \in \bZ^2_+$ and let $\pi^*$ be the rightmost geodesic for $L^{(M)}(z)$. One sees that the presence of an edge $e \in E(\bZ^2_+)$ in $\pi^*$ is a nonincreasing function of $W(e)$. From Harris's inequality then, the probability that $W(e)$ exceeds the threshold $L$ can only decrease when we condition on it belonging to $\pi^*$:
		\begin{equation}
			\Prob{W(e) \ge M \given e \in \pi^*} \le \Prob{W(e) \ge M}.
		\end{equation}
		
		Conditional on $\set{W(e) \ge M}$, the event $e \in \pi^*$ is independent of $W(e)$. This is because the presence of $e \in \pi^*$ depends on $W(e)$ only through the truncation $W^{(M)}(e)$, and upon assuming $W(e) \ge M$, additional knowledge of $W(e)$ is irrelevant. So
		\begin{align}
			\Ex{(W(e) - M)_+ \given e \in \pi^*} &\le \Ex[\big]{(W(e) - M)_+\bone\set{W(e) \ge M} \given e \in \pi^*}\\
			&= \Ex[\big]{(W(e) - M)_+ \given W(e) \ge M}\Prob{W(e) \ge M \given e \in \pi^*}\\
			&\le \Ex[\big]{(W(e) - M)_+ \given W(e) \ge M}\Prob{W(e) \ge M}\\
			&= \Ex{(W(e) - M)_+}\\
			&\le \int_{M}^{\infty}1 - \overline{F}(s)\, ds.
		\end{align}
		
		Now
		\begin{align}
			\Ex{L(z)} &= \Ex[\big]{\min_{\pi} \sum_{e \in \pi}W(e)}\\
			&\le \Ex[\big]{\min_{\pi} \sum_{e \in \pi} W^{(M)}(e) + (W(e) - M)_+}\\
			&\le \Ex[\big]{\sum_{e \in \pi^*} W^{(M)}(e)} + \Ex{\sum_{e \in \pi^*} (W(e) - M)_+}\\
			&= \Ex{L^{(M)}(z)} + \sum_{e \in E(\bZ_+^2)} \Prob{e \in \pi^*}\Ex{(W(e) - M)_+ \given e \in \pi^*}\\
			&\le \Ex{L^{(M)}(z)} + \int_{M}^{\infty}1 - \overline{F}(s)\, ds \sum_{e \in E(\bZ_+^2)}\Prob{e \in \pi^*}\\
			&= \Ex{L^{(M)}(z)} + \norm{z}\int_{M}^{\infty}1 - \overline{F}(s)\, ds.
		\end{align}
		Putting $z = \floor{nx}$ and taking the limit, we arrive at the upper bound. 
	\end{proof}
\end{lemma}

There are two final lemmas establishing uniform convergence of the truncated passage times and limit shape to the untruncated counterparts.  

\begin{lemma}
	\label{lem:MartinTruncPassageToPassage}
	Let $\epsilon > 0$ be given. Then there is $M$ large enough such that almost surely, we have for all but finitely many $z \in \bZ^2_+$ that
	\begin{equation}
		\label{eq:MartinTruncPassageToPassageBound}
		\abs{L(z) - L^{(M)}(z)} \le \epsilon \norm{z}.
	\end{equation}
\end{lemma}

\begin{lemma}
	\label{lem:MartinTruncShapeToShape}
	Let $\epsilon > 0$ be given. Then there is $M$ large enough such that almost surely, we have for all but finitely many $z \in \bZ^2_+$ that
	\begin{equation}
		\abs{\ell(z) - \ell^{(M)}(z)} \le \epsilon \norm{z}.
	\end{equation}
\end{lemma}

Of these, \cref{lem:MartinTruncShapeToShape} is immediate from \cref{lem:MartinTruncationApprox}.

\begin{proof}[Proof of \cref{lem:MartinTruncPassageToPassage}]
	Except for swapping signs and the need to involve $\underline{F}, \overline{F}$ due to their appearance in \cref{lem:MartinPTBounds}, we can largely follow Martin's argument unchanged. Choose $M$ so that $c \int_{-\infty}^{-M}\underline{F}(s)^{1 / 2}\, ds < \epsilon$ and $c\int_{M}^{\infty}(1 - \overline{F}(s))^{1 / 2}\, ds < \epsilon$.
	
	Take $z \in \bZ_+^2$. Let $\pi^*$ be a geodesic for $L^{(M)}$. Then 
	\begin{align}
		L(z) - L^{(M)}(z) &= \min_{\pi \in \Pi(z)}\sum_{e \in \pi} W(e) - \sum_{e \in \pi^*}W^{(M)}(e)\\ 
		&\le \sum_{e \in \pi^*} W(e) - W^{(M)}(e)\\
		&\le \sum_{e \in \pi^*} (W(e) - M)_+\\
		&\le - \min_{\pi \in \Pi(z)}\sum_{e \in \pi} V^{(M)}(e),
	\end{align}
	where $V^{(M)} = -(W(e) - M)_+$. Similarly,
	\begin{equation}
		L^{(M)}(z) - L(z) \le - \min_{\pi \in \Pi(z)}\sum_{e \in \pi} U^{(M)}(e),
	\end{equation}
	where now $U^{(M)} = -(-M - W(e))_+$. In either case the bound is non-negative, so we combine them and get
	\begin{equation}
		\abs{L^{(M)}(z) - L(z)} \le - \min_{\pi \in \Pi(z)} \sum_{e \in \pi}V^{(M)}(e) - \min_{\pi \in \Pi(z)} \sum_{e \in \pi}U^{(M)}(e).
	\end{equation}
	
	Observe that the $\set{V^{(M)}(e)}$ and $\set{U^{(M)}(e)}$ fall under \cref{ass:WeightAssump}. Write $F^{(M)}_{V, i}$ for the distribution of $V^{(M)}(0; i)$, and $\underline{F}^{(M)}_{V}$ for the distribution of $V^{(M)}(0; 1) \wedge V^{(M)}(0; 2)$. Then $\underline{F}^{(M)}_{V}(s) = 1 - \overline{F}(M - s)$ on $s \le 0$, and $\underline{F}^{(M)}_{V}(s) = 1$ elsewhere. We apply \cref{lem:MartinPTBounds2} to find that almost surely
	\begin{align}
		\liminf_{n \to \infty}\frac{1}{n}\min_{\norm{z}_1 \le n} \min_{\pi \in \Pi(z)} \sum_{e \in \pi}V^{(M)}(e) &\ge -c \int_{-\infty}^0 \underline{F}^{(M)}_{V}(s)^{1 / 2}\, ds\\
		&= -c \int_{-\infty}^0 (1 - \overline{F}(M - s)^{1 / 2}\, ds\\
		&= -c \int_{M}^\infty (1 - \overline{F}(s))^{1 / 2}\, ds\\
		&\ge - \epsilon / 2.
	\end{align}
	Then there are only finitely many $z$ for which
	\begin{equation}
		- \min_{\pi \in \Pi(z)} \sum_{e \in \pi}V^{(M)}(e) \ge \frac{\epsilon}{2}\norm{z}.
	\end{equation}
	
	We can do the same for the $\set{U^{(M)}(e)}$ to find
	\begin{align}
		\liminf_{n \to \infty}\frac{1}{n}\min_{\norm{z}_1 \le n} \min_{\pi \in \Pi(z)} \sum_{e \in \pi}U^{(M)}(e) &\ge -c \int_{-\infty}^0 \underline{F}^{(M)}_{U}(s)^{1 / 2}\, ds\\
		&= -c \int_{-\infty}^{-M} \underline{F}(s)^{1 / 2}\, ds\\
		&\ge - \epsilon / 2,
	\end{align}
	so that there are only finitely many $z$ with
	\begin{equation}
		- \min_{\pi \in \Pi(z)} \sum_{e \in \pi}U^{(M)}(e) \ge \frac{\epsilon}{2}\norm{z}.
	\end{equation}
	
	Looking back now at \eqref{eq:MartinTruncPassageToPassageBound}, we find that only finitely many $z$ have
	\begin{equation}
		\abs{L(z) - L^{(M)}(z)} \ge \epsilon \norm{z}.
	\end{equation}
	This is what we wanted.
\end{proof}

Combining \cref{lem:MartinBoundedPassageToShape,lem:MartinTruncPassageToPassage,lem:MartinTruncShapeToShape} gives \cref{thm:UniformLimitShape}.

\section{The Busemann process in directed FPP}
\label{app:Busemann}
Under \cref{ass:WeightAssump}, the results of \cite{groa-janj-rass-gen-buse} apply to give the existence of \emph{generalised Busemann functions}. In what follows, let $\scrU = \set{(t, 1 - t) : 0 < t < 1}$ be the set of directions into the first open  quadrant, and let $\scrU_0$ be some countable subset (which can be assumed to be dense). To better match notation, we use $\omega_x = (\omega_{x, 1}, \omega_{x, 2})$ for the weights into vertex $x$. Denote by $T_x$ translations of the environment $T_x(\omega_y)$ = $\omega_{y - x}$.

Below is essentially a restatement of Theorem 4.4 of \cite{groa-janj-rass-gen-buse}, but specialised to $\beta = \infty$ and the face $\cA \in \bA$ being the entire limit shape\footnote{Such a degenerate choice of $\cA$ is explicitly allowed in their statement.}. The $m$ in their statement is a member of a superdifferential of the limit shape, and here the role is taken by $(\zeta, \pm)$.

\begin{theorem}
	\label{thm:GenBusemannExistence}
	There exists a probability space $(\widehat{\Omega}, \widehat{\cB}, \widehat{\bP})$ with a measurable projection onto $\Omega$, and real-valued measurable functions $B^{\xi}(\widehat{\omega}, x, y)$ of $(\widehat{\omega}, \xi, x, y) \in \widehat{\Omega} \times \scrU_0 \times \bZ^{2} \times \bZ^{2}$ and a translation invariant Borel probability measure $\widehat{\bP}$ on $(\widehat{\Omega}, \widehat{\cB})$, such that the following properties hold:
	\begin{thmlist}
		\item (Consistency) Under $\widehat{\bP}$, the marginal distribution of the configuration $\omega$ is i.i.d with the specified marginals. For each $\xi \in \scrU_0$, the $\bR^{4}$-valued process $\set{\psi_{x}^{ \xi}}_{x \in \bZ^{2}}$ defined by
		\begin{equation}
			\psi_{x}^{\xi}(\widehat{\omega}) = (\omega_{x, 1}, \omega{x, 2}, B^{\xi}(\widehat{\omega}, x, x + e_{1}), B^{\xi}(\widehat{\omega}, x, x + e_{2}))
		\end{equation}
		is stationary under translations $T_{x}$. For any $I \subseteq \bZ^{2}$, the variables
		\begin{equation}
			\set[\big]{(\widehat{\omega}_{x}, B^{\xi}(\widehat{\omega}, x, x + e_{i})) : x \in I, \xi \in \scrU_0, i \in \set{1, 2}}
		\end{equation}
		are independent of $\set{\omega_{x} : x \in I^{<}}$.
		\label[thm]{thm:GenBusemannExistence1}
		
		\item (Adaptedness) For a fixed $\xi \in \scrU_0$, the process $B^{\xi} = \set{B^{\xi}(x, y)}_{x, y \in \bZ^{2}}$ is a stationary $L^{1}(\widehat{\bP})$ cocycle (in the sense of \cite{geor-rass-sepp-16}) that recovers the potential:
		\begin{equation}
			\max_{i \in \set{1, 2}} B^{\xi}(x, x + e_{i}) - \omega_{x, i} = 0.
		\end{equation}
		\label[thm]{thm:GenBusemannExistence2}
		
		\item (Distinct means) The mean vectors $h(\xi) = h(B^{\xi})$ defined by
		\begin{equation}
			h(\xi) \cdot e_{i} = \Ex{B^{\xi}(0, e_{i})}
		\end{equation}
		satisfy
		\begin{equation}
			h(\xi) = \nabla \ell(\xi).
		\end{equation}
		If $h(\xi) = h(\zeta)$ then $B^{\xi}(x, y) = B^{\zeta}(x, y) \as$.\label[thm]{thm:GenBusemannExistence3}
\end{thmlist}
\end{theorem}

\begin{remark}
	The generalised Busemann functions are related to the underlying environment only through the adaptedness property. Thus, any proof which takes the above as its starting point and produces these objects as functions of the environment must use adaptedness in an essential way. The utility of this property was identified in \cite{geor-rass-sepp-16}, where it is connected to maximisers of variational formulas for the limit shape. It is also interesting to note that the adaptedness here is exactly the notion of max-plus harmonicity described in \cite{aki-gau-wal-max-plus}, the graph here being the integer lattice with up-right directed edges.
	
	The property has been called \emph{recovery} in the context of vertex-weight LPP \cite{sepp-cgm-18}, referring to the fact that one can completely recover the weight configuration from the collection of Busemann functions in a fixed direction. In edge-weight FPP we have only this weaker statement: if $\scrU_0$ is dense, then almost surely the weight configuration is determined by the full collection of generalised Busemann functions.
\end{remark}

We are interested in showing that these functions arise as limiting differences of the passage times, and moreover that they extend to a full-fledged \emph{Busemann process} indexed by $\scrU$. Specifically, for a direction $\xi \in \scrU$ and $x,\, y \in \bZ^2$, we look at limits
\begin{equation}
	\label{eqn:BusemannLimit}
	B^{\xi}(x, y) = \lim_{n \to \infty} L(x, v_n) - L(y, v_n),
\end{equation}
where $(v_n)_{n \in \bZ_{\ge 0}} \subset \bZ^2$ is a sequence of vertices with limiting direction $\xi$ and $\abs{v_n} \to \infty$.

This task has been carried out in \cite{janj-rass-plane-buse} for planar vertex-weight models. The extension to the remaining directions relies on the “path crossing trick'' to give monotonicity, after which limits can be taken to extend the process from $\xi \in \scrU_0$ to $\xi \in \scrU$. The relevant statement for our setup is below.

\begin{lemma}[Path-crossing trick]
	\label{lem:PathCrossing}
	Suppose $\abs{u}_{1} = \abs{v}_{1}$ and $u_{1} \le v_{1}$. Then
	\begin{equation}
		\label{eq:PathCrossing1}
		L(0, u) - L(e_{1}, u) \le L(0, v) - L(e_{1}, v)
	\end{equation}
	and
	\begin{equation}
		\label{eq:PathCrossing2}
		L(0, u) - L(e_{2}, u) \ge L(0, v) - L(e_{2}, v),
	\end{equation}
	whenever these passage times are defined.
	\begin{proof}
		We prove \eqref{eq:PathCrossing1}. This holds trivially if $u_{1} = v_{1}$, so assume $u_{1} < v_{1}$. This implies $u_{2} > v_{2}$. Fix a geodesic connecting $e_{1}$ to $u$ and $0$ to $v$. The relative positions of $u$ and $v$ ensure that the geodesic (or any other directed path) connecting $e_{1}$ to $u$ must cross the geodesic connecting $0$ to $v$. Let $x$ be the first point of intersection. Passage times are sub-additive, hence
		\begin{equation}
			L(0, x) + L(x, u) \ge L(0, u),\quad L(e_{1}, x) + L(x, v) \ge L(e_{1}, v).
		\end{equation}
		These can be combined and rearranged to give
		\begin{equation}
			L(0, u) - L(e_{1}, x) - L(x, u) \le L(0, x) + L(x, v) - L(e_{1}, v).
		\end{equation}
		That $x$ lies on both geodesics means $L(e_{1}, x) + L(x, u) = L(e_{1}, u)$ and $L(0, x) + L(x, v) = L(0, v)$. Thus we arrive at
		\begin{equation}
			L(0, u) - L(e_{1}, u) \le L(0, v) - L(e_{1}, v).
		\end{equation}
	\end{proof}
\end{lemma}

It is then quite transparent from the construction of the generalised Busemann functions in \cite{groa-janj-rass-gen-buse} that we have the following additional property.

\begin{lemma}
	\label{lem:BusemannMonotonicity}
	There exists an event $\widehat{\Omega}_{0}$ with $\widehat{\bP}(\widehat{\Omega}_{0}) = 1$ and such that if $\xi, \zeta \in \scrU$ with $\xi \cdot e_{1} < \zeta \cdot e_{1}$, then
	\begin{equation}
		B^{\xi}(\widehat{\omega}, x, x + e_{1}) \le B^{\zeta}(\widehat{\omega}, x, x + e_{1})
	\end{equation}
	and 
	\begin{equation}
		B^{\xi}(\widehat{\omega}, x, x + e_{2}) \ge B^{\zeta}(\widehat{\omega}, x, x + e_{2}).
	\end{equation}
\end{lemma}

It now makes sense to define for $\xi \in \scrU$
\begin{gather}
	\label{eqn:BusemannLimitDef}
	B^{\xi +}(\widehat{\omega}, x, y) = \lim_{\zeta \cdot e_1 \searrow \xi \cdot e_1} B^{\zeta}(\widehat{\omega}, x, y),\\
	B^{\xi -}(\widehat{\omega}, x, y) = \lim_{\zeta \cdot e_1 \nearrow \xi \cdot e_1} B^{\zeta}(\widehat{\omega}, x, y).
\end{gather}
The limits are taken through $\zeta \in \scrU_0$. That these limits exist follows from monotonicity and the cocycle property.

Finally, we summarise the properties of the collection of these extended generalised Busemann functions.

\begin{theorem}\label{thm:BusemannExistenceGeneral}
	Let $(\widehat{\Omega}, \widehat{\cB}, \widehat{\bP})$ be as in \cref{thm:GenBusemannExistence}. There are functions $B^{\xi \pm}(\widehat{\omega}, x, y)$ of $(\widehat{\omega}, \xi, x, y) \in \widehat{\Omega} \times \scrU\times \bZ^{2} \times \bZ^{2}$, such that the following properties hold:
	\begin{thmlist}
		\item \label{thm:BusemannExistenceGeneralConsistency} (Consistency) Under $\widehat{\bP}$, the marginal distribution of the configuration $\omega$ is i.i.d with the specified marginals. For each $\xi \in \scrU$, the $\bR^{3}$-valued process $\set{\psi_{x}^{\xi \pm}}_{x \in \bZ^{2}}$ defined by
		\begin{equation}
			\psi_{x}^{\xi \pm}(\widehat{\omega}) = (\omega_{x}, B^{\xi \pm}(\widehat{\omega}, x, x + e_{1}), B^{\xi \pm}(\widehat{\omega}, x, x + e_{2}))
		\end{equation}
		is stationary under translations $T_{x}$. For any $I \subseteq \bZ^{2}$, the variables
		\begin{equation}
			\set[\big]{(\omega_{x}, B^{\xi +}(\widehat{\omega}, x, x + e_{i}), B^{\xi -}(\widehat{\omega}, x, x + e_{i})) : x \in I, \xi \in \scrU_0, i \in \set{1, 2}}
		\end{equation}
		are independent of $\set{\omega_{x} : x \in I^{<}}$.
		
		\item \label{thm:BusemannExistenceGeneralAdaptedness} (Adaptedness) For a fixed $\xi \in \scrU$, the process $B^{\xi \pm} = \set{B^{\xi \pm}(x, y)}_{x, y \in \bZ^{2}}$ is a stationary $L^{1}(\widehat{\bP})$ cocycle satisfying
		\begin{equation}
			\min_{i \in \set{1, 2}} \omega_{x, i} - B^{\xi}(x, x + e_{i}) = 0.
		\end{equation}
		
		\item \label{thm:BusemannExistenceGeneralMonotonicity} There exists an event $\widehat{\Omega}_{0}$ with $\widehat{\bP}(\widehat{\Omega}_{0}) = 1$ and such that the following hold for all $\widehat{\omega} \in \widehat{\Omega}_{0}$, $x, y \in \bZ^{2}$ and $\xi,\, \zeta \in \scrU$.
		\begin{enumerate}[label=(\alph*)]
			\item (Monotonicity) If $\xi \cdot e_{1} < \zeta \cdot e_{1}$, then
			\begin{equation}
				B^{\xi -}(\widehat{\omega}, x, x + e_{1}) \le B^{\xi +}(\widehat{\omega}, x, x + e_{1}) \le B^{\zeta -}(\widehat{\omega}, x, x + e_{1})
			\end{equation}
			and
			\begin{equation}
				B^{\xi -}(\widehat{\omega}, x, x + e_{2}) \ge B^{\xi +}(\widehat{\omega}, x, x + e_{2}) \ge B^{\zeta - }(\widehat{\omega}, x, x + e_{2}).
			\end{equation}
			\item (One-sided continuity) If $\xi_{n} \cdot e_{1} \searrow \zeta \cdot e_{1}$, then
			\begin{equation}
				\lim_{n \to \infty}B^{\xi_{n}\pm}(\widehat{\omega}, x, y) = B^{\zeta +}(\widehat{\omega}, x, y).
			\end{equation}
			Similarly, if $\xi_{n} \cdot e_{1} \nearrow \zeta \cdot e_{1}$, then
			\begin{equation}
				\lim_{n \to \infty}B^{\xi_{n}\pm}(\widehat{\omega}, x, y) = B^{\zeta - }(\widehat{\omega}, x, y).
			\end{equation}
		\end{enumerate}

		\item \label{thm:BusemannExistenceGeneralDistinct} (Distinct means) The mean vectors $h(\xi \pm) = h(B^{\xi \pm})$ defined by
		\begin{equation}
			h(\xi \pm) \cdot e_{i} = \Ex{B^{\xi \pm}(0, e_{i})}
		\end{equation}
		satisfy
		\begin{equation}
			h(\xi \pm) = \nabla \ell(\xi \pm).
		\end{equation}
		If $h(\xi +) = h(\zeta -)$ then $B^{\xi +}(x, y) = B^{\zeta -}(x, y) \as$. Similarly for $h(\xi +) = h(\zeta +)$ and $h(\xi -) = h(\zeta -)$.
\end{thmlist}\end{theorem}

\subsection{As gradients of passage times}
\label{ssec:BusemannGradients}
It remains to see that these objects are genuine Busemann functions given by the limits of gradients, as in \eqref{eqn:BusemannLimit}. We will go through the lemmas of Section 6 in \cite{geor-rass-sepp-17-buse}, noting where the details differ. To that end, fix a $v \in \bZ^{2}$ and for $x \le v - e_{1}$, $y \le v - e_{2}$, define increments
\begin{align}
	&I(x, v) = L(x, v) - L(x + e_{1}, v)\\
	\text{and } &J(y, v) = L(y, v) - L(y + e_{2}, v).  
\end{align}
The path crossing trick applies to give various inequalities for these increments.
\begin{lemma}
	\label{lem:IncrementIneqs}
	\begin{align}
		&I(x, v + e_{2}) \le I(x, v) \le I(x, v + e_{1})\\
		\text{and } &J(x, v + e_{2}) \ge J(x, v) \ge J(x, v + e_{1}).  
	\end{align}
\end{lemma}

The next lemma links the Busemann functions to limiting directions of the FPP, but requires quite a bit of notation. Recalling that $\ell$ is the limit shape of our model, set $\gamma(s) = \ell(1, s)$. Note that $\ell$ won't in general be symmetric. The convexity of $\ell$ ensures the existence of one-sided derivatives for $\gamma$. Fix $\zeta \in \scrU$ and a cocycle $B = B^{\zeta \pm}$. Take the (inverse) slope $r = \zeta \cdot e_{1} / \zeta \cdot e_{2}$, so that $\alpha = \gamma'(r \pm)$ (the same choice of sign as for $B$) satisfies
\begin{equation}
	\alpha = \Exhat{B(0, e_{1})}.
\end{equation}
This is \cref{thm:BusemannExistenceGeneralDistinct}. Define $f(\alpha)$ by
\begin{equation}
	f(\alpha) = \Exhat{B(0, e_{2})}.
\end{equation}
Fix for the moment some $v \in \bZ^{2}$ and define, for $u \le v$,
\begin{align}
	\label{eq:EdgeRecur}
	\GNE(u, v) =
	\begin{cases}
		B(u, v), \quad v - u = ke_{i},\, k \in \bZ_{+}, i \in \set{1, 2}\\
		(\omega(u, u + e_{1}) + \GNE(u + e_{1}, v)) \wedge (\omega(u, u + e_{2}) + \GNE(u + e_{2}, v)),\quad \text{otherwise.}
	\end{cases}
\end{align}
These are the passage times with the generalised Busemann functions as initial conditions along the northeast boundary. 

Write $\GNE_{v - e_{i} \in \pi}(0, v)$ for the minimal passage time among paths which reach $v$ through the edge $\set{v - e_{i}, v}$ (where we use in the definition the corresponding weights for $\GNE$, which can be recovered uniquely from the definition above).

\begin{lemma}
	Fix $0 < s, t < \infty$. Let $v_{n} \in \bZ^{2}$ be such that $v_{n} / \abs{v_{n}}_{1} \to (s, t) / (s + t)$ as $n \to \infty$ and such that $\abs{v_{n}} \ge \eta_{0}n$ for some $\eta_{0} > 0$. Then the following limits hold almost surely:
	\begin{equation}
		\abs{v_{n}}_{1}^{-1}\GNE_{v_{n} - e_{1} \in \pi}(0, v_{n}) \underset{{n \to \infty}}{\longrightarrow} (s + t)^{-1}\max_{0 \le \tau \le s}\set{\ell(s - \tau, t) + \alpha\tau}
	\end{equation}
	and
	\begin{equation}
		\abs{v_{n}}_{1}^{-1}\GNE_{v_{n} - e_{2} \in \pi}(0, v_{n}) \underset{{n \to \infty}}{\longrightarrow} (s + t)^{-1}\min_{0 \le \tau \le s}\set{\ell(s, t - \tau) + f(\alpha)\tau}.
	\end{equation}
	\begin{proof}
		The proof can be carried out as in \cite{geor-rass-sepp-17-buse} after minor modifications to the estimates. The symmetry of the setup means that it is enough to look at the $e_{1}$-axis. Fix $\epsilon > 0$, let $M = \floor{\epsilon^{-1}}$, and define steps
		\begin{equation}
			q^{n}_{j} = j \floor*{\frac{\epsilon \abs{v_{n}}_{1} s}{s + t}}, \text{ for } 0 \le j \le M - 1, \text{ and } q^{n}_{M} = v_{n} \cdot e_{1}.
		\end{equation}
		Notice that for $n$ large we have $q^{n}_{M - 1} < v_{n} \cdot e_{1} = q^{n}_{M}$.
		
		Suppose a minimal path for $\GNE_{v - e_{1} \in \pi}(0, v)$ enters the north boundary at the point $v_{n} - (l, 0)$ and choose $j$ so that $q^{n}_{j} < l \le q^{n}_{j + 1}$. Write $m_{0} = \Ex{\omega_{1}(x)}$. We have a bound
		\begin{align}
			\GNE_{v - e_{1} \in \pi}(0, v) &= L(0, v_{n} - (l, 1)) + \omega_{2}(v_{n} - (l, 1)) +  B(v_{n} - (l, 1), v_{n})\\
			&\ge
			\begin{multlined}[t]
				L(0, v_{n} - (q^{n}_{j}, 1)) + q^{n}_{j} \alpha + \omega_{2}(v_{n} - (l, 1)) + \sum_{k = q^{n}_{j} + 1}^{l}(\omega_{1}(v_{n} - (k, 1)) - m_{0})\\ + (l  - q^{n}_{j})m_{0} + (B(v_{n} - (l, 1), v_{n}) - l \alpha) + (l - q^{n}_{j})\alpha.
			\end{multlined}
		\end{align}
		
		Proceeding in the same way, define $F(x, y) = h(B) \cdot (x - y) - B(x, y)$, which has
		\begin{equation}
			B(v_{n} - (l, 0), v_{n}) - l \alpha = F(0, v_{n} - (l, 0)) - F(0, v_{n}).
		\end{equation}
		Here the adaptedness property is $0 = (B(0, e_{1}) - \omega_{1}(0)) \vee (B(0, e_{2}) - \omega_{2}(0))$. The resulting bound on the centred cocycle is 
		\begin{equation}
			F(0, e_{i}) \le \alpha \wedge f(\alpha) - \omega_{1}(0) \vee \omega_{2}(0).
		\end{equation}
		By \cref{ass:WeightAssump}, the variables $\set{\omega_{1}(x) \vee \omega_{2}(x)}_{x \in \bZ^{2}}$ are i.i.d with $2 + \epsilon$ moments, and so \cite[Theorem A.1]{geor-rass-sepp-17-buse} applies to $F$.
		Writing $S_{j, m}^{n} = \sum_{l = q_{j}^{n} + 1}^{q_{j}^{n} + m}(\omega_{1}(v_{n} - (k, 1)) - m_{0})$ and $C$ for some constant, then maximising over $j$ in our bound above,
		\begin{equation}
			\begin{multlined}[t]
				\GNE_{v - e_{1} \in \pi}(0, v) \ge 
				\max_{0 \le j \le M - 1}\Bigg\{ L(0, v_{n} - (q^{n}_{j}, 1)) + q^{n}_{j} \alpha + \max_{q^{n}_{j} \le l \le q^{n}_{j + 1}}\abs{\omega_{2}(v_{n} - (l, 1))} + \max_{0 \le m \le q^{n}_{j + 1} - q^{n}_{j}}\abs{S^{n}_{j, m + 1}}\\ + \max_{q^{n}_{j} \le l \le q^{n}_{j + 1}}(\abs{F(0, v_{n} - (l, 0))} + \abs{F(0, v_{n})})\Bigg\}.
			\end{multlined}
		\end{equation}
		Divide through by $\abs{v_{n}}_{1}$ and let $n \to \infty$. Each of these terms converges as in \cite{geor-rass-sepp-17-buse}, except that we have an additional term
		\begin{equation}
			\abs{v_{n}}_{1}^{-1}\max_{0 \le l \le v_{n} \cdot e_{1}}\abs{\omega_{2}(v_{n} - (l, 1))}.  
		\end{equation}
		But the finite variance of the weights is enough to ensure that this term too goes to zero almost surely. This gives the lower bound
		\begin{equation}
			\liminf_{n \to \infty}\abs{v_{n}}_{1}^{-1}\GNE_{v_{n} - e_{1} \in \pi}(0, v_{n}) \ge (s + t)^{-1}\max_{0 \le \tau \le s}\set{\ell(s - \tau, t) + \alpha\tau}.
		\end{equation}
		The argument for the upper bound is completely identical to the one in \cite{geor-rass-sepp-17-buse} (after swapping signs). 
	\end{proof}
\end{lemma}
\begin{lemma}
	Let $s \in (r, \infty)$. Let $v_{n} \in \bZ^{2}$ be such that $v_{n} / \abs{v_{n}}_{1} \to (s, t) / (s + t)$ as $n \to \infty$ and such that $\abs{v_{n}} \ge \eta_{0}n$ for some $\eta_{0} > 0$. Assume that $\gamma'(r+) > \gamma'(s-)$. Then $\widehat{\bP}$-$\as$ there exists a random $n_{0} < \infty$ such that for all $n \ge n_{0}$
	\begin{equation}
		\GNE(0, v_{n}) = \GNE_{v_{n} - e_{1} \in \pi}(0, v_{n}).
	\end{equation}
\end{lemma}
The proof in \cite{geor-rass-sepp-17-buse} makes no mention of the weights and goes through word-for-word, so we skip it.

There are some final definitions before the theorem. Write
\begin{equation}
	\scrD = \set{\xi \in \scrU : \ell \text{ is differentiable at } \xi}.
\end{equation}
For a direction $\xi \in \scrU$, consider the maximal line segments of $\ell$ to which $\xi$ belongs:
\begin{equation}
	\scrU_{\xi\pm} = \set{\zeta \in \scrU : \ell(\zeta) - \ell(\xi) = \nabla \ell(\xi \pm) \cdot (\zeta - \xi)}.
\end{equation}
Let
\begin{equation}
	\scrU_{\xi} = \scrU_{\xi-} \cup \scrU_{\xi+} = [\lowerbar{\xi}, \overbar{\xi}], \quad \text{where } \lowerbar{\xi} \cdot e_{1} \le \overbar{\xi}\cdot e_{1}.
\end{equation}

\begin{theorem}
	Fix a possibly degenerate segment $[\zeta, \eta] \subseteq \scrU$. Assume that either $[\zeta, \eta]$ consists of a single exposed point $\xi$ such that $\xi = \lowerbar{\xi} = \overbar{\xi} = \zeta = \eta$, or that $[\zeta, \eta]$ is a maximal, non-degenerate linear segment of $\ell$ so that $[\zeta, \eta] = [\lowerbar{\xi}, \overbar{\xi}]$ for all $\xi \in (\zeta, \eta)$. Then there exists an event $\widehat{\Omega}_{0}$ with $\widehat{\bP}(\widehat{\Omega}_{0}) = 1$ such that for each $\widehat{\omega} \in \widehat{\Omega}_{0}$ and for any sequence $v_{n} \in \bZ_{+}^{2}$ with
	\begin{equation}
		\abs{v_{n}}_{1} \to \infty \text{ and } \zeta \cdot e_{1} \le \liminf_{n \to \infty}\frac{v_{n} \cdot e_{1}}{\abs{v_{n}}_{1}} \le \limsup_{n \to \infty}\frac{v_{n} \cdot e_{1}}{\abs{v_{n}}_{1}} \le \eta \cdot e_{1},
	\end{equation}
	we have for all $x \in \bZ_{+}^{2}$
	\begin{align}
		B^{\zeta -}(\widehat{\omega}, x, x + e_{1}) &\le \liminf_{n \to \infty}(L(\omega, x, v_{n}) - L(\omega, x + e_{1}, v_{n}))\\
		&\le \limsup_{n \to \infty}(L(\omega, x, v_{n}) - L(\omega, x + e_{1}, v_{n})) \le B^{\eta +}(\widehat{\omega}, x, x + e_{1})
	\end{align}
	and
	\begin{align}
		B^{\eta +}(\widehat{\omega}, x, x + e_{2}) &\le \liminf_{n \to \infty}(L(\omega, x, v_{n}) - L(\omega, x + e_{2}, v_{n}))\\
		&\le \limsup_{n \to \infty}(L(\omega, x, v_{n}) - L(\omega, x + e_{2}, v_{n})) \le B^{\zeta -}(\widehat{\omega}, x, x + e_{2}).
	\end{align}
\end{theorem}
After the appropriate redefinition of the $L^{N}$ passage times, the proof is identical. Finally, we arrive at
\begin{corollary}
	Assume $\xi, \lowerbar{\xi}, \overbar{\xi} \in \scrD$. Then there exists an event $\widehat{\Omega}_{0}$ with $\widehat{\bP}(\widehat{\Omega}_{0}) = 1$ such that for each $\widehat{\omega} \in \widehat{\Omega}_{0}$, $x,\, y \in \bZ_{+}^{2}$, and for any sequence $v_{n} \in \bZ_{+}^{2}$ with
	\begin{equation}
		\abs{v_{n}}_{1} \to \infty \text{ and } \lowerbar{\xi} \cdot e_{1} \le \liminf_{n \to \infty}\frac{v_{n} \cdot e_{1}}{\abs{v_{n}}_{1}} \le \limsup_{n \to \infty}\frac{v_{n} \cdot e_{1}}{\abs{v_{n}}_{1}} \le \overbar{\xi} \cdot e_{1},
	\end{equation}
	we have 
	\begin{equation}
		B^{\xi}(\widehat{\omega}, x, y) = \lim_{n \to \infty}L(\omega, x, v_{n}) - L(\omega, y, v_{n}).
	\end{equation}
\end{corollary}

The corollary shows in particular that when the limit shape is differentiable, the Busemann process is a function of the weights, so is Borel measurable and ergodic under shifts of the environment.  

\subsection{Construction of semi-infinite geodesics}
Lastly, we prove that the procedure of \cref{rem:GeodesicConstruction} does indeed produce geodesics with the correct direction. We follow Section 8 of \cite{rass-cgm-18}. Recall that we have an angle $\xi \in \scrU$ and a procedure that produces a path $\pi = (x_0, x_1, \dots)$, such that $x_{k + 1} - x_{k} \in \set{e_1, e_2}$ and such that $B^\xi(x_k, x_{k + 1}) = W(x_k, x_{k + 1})$. For our purposes, we may assume that the limit shape $\ell$ is differentiable and strictly convex at $\xi$.

\begin{lemma}
	\label{lem:GeodesicMachine}
	The path $\pi$ is a semi-infinite geodesic.
	\begin{proof}
		Consider a truncated path $\pi^{(n)} = (x_0, \dots, x_n)$. The cocycle property of the Busemann functions and our construction mean that
		\begin{equation}
			B^\xi(x_0, x_n) = \sum_{k = 0}^{n - 1}B^\xi(x_{k}, x_{k + 1}) = \sum_{k = 0}^{n - 1}W(x_k, x_{k + 1}) = L(\pi^{(n)}).
		\end{equation}
		On the other hand, if $\tilde{\pi}^{(n)} = (\tilde{x}_0, \dots, \tilde{x}_n)$ is another path with the same endpoints, $\tilde{x}_0 = x_0$ and $\tilde{x}_n = x_n$, then we may also write
		\begin{equation}
			B^\xi(x_0, x_n) = \sum_{k = 0}^{n - 1}B^\xi(\tilde{x}_{k}, \tilde{x}_{k + 1}).
		\end{equation}
		However now, the adaptedness property means that $B^\xi(\tilde{x}_{k}, \tilde{x}_{k + 1}) \le W(\tilde{x}_{k}, \tilde{x}_{k + 1})$. Hence
		\begin{equation}
			L(\pi^{(n)}) = B^\xi(x_0, x_n) = \sum_{k = 0}^{n - 1}B^\xi(\tilde{x}_{k}, \tilde{x}_{k + 1}) \le \sum_{k = 0}^{n - 1}W(\tilde{x}_{k}, \tilde{x}_{k + 1}) = L(\tilde{\pi}^{(n)}).
		\end{equation}
		Taking a minimum over the possible $\tilde{\pi}^{(n)}$ shows that $L(\pi^{(n)}) = L(x_0, x_n)$, and this segment is a geodesic. We may do this for any $n$.
	\end{proof}
\end{lemma}

\begin{lemma}
	\label{lem:BusemannErgodic}
	Suppose the limit shape $\ell$ is differentiable and strictly convex at $\xi$, and take another direction $\zeta \in \scrU$. Almost surely, we have
	\begin{equation}
		\lim_{n \to \infty}\frac{B^\xi(0, \floor{n \zeta})}{n} = \nabla \ell (\xi) \cdot \zeta.
	\end{equation}
	\begin{proof}
		The assumption on the limit shape means that the Busemann increments $B^\xi_i(x)$ are ergodic with respect to shifts. Let $\floor{n \zeta} = (p_n, q_n)$ and use the cocycle property to write
		\begin{equation}
			B^\xi(0, \floor{n \zeta}) = \sum_{k = 0}^{p_n - 1}B^\xi((k, 0), (k + 1, 0)) + \sum_{l = 0}^{q_n - 1}B^\xi((p_n, l), (p_n, l + 1)).
		\end{equation}
		By the ergodic theorem, these sums have
		\begin{equation}
			\frac{1}{p_n}\sum_{k = 0}^{p_n - 1}B^\xi((k, 0), (k + 1, 0)) \to \Ex{B^\xi(0, e_1)}, \text{ and } \frac{1}{q_n}\sum_{l = 0}^{q_n - 1}B^\xi((p_n, l), (p_n, l + 1)) \to \Ex{B^\xi(0, e_2)},
		\end{equation}
		almost surely. Now note that $p_n / n \to \zeta \cdot e_1$ and $q_n / n \to \zeta \cdot e_2$, and recall from \cref{thm:GenBusemannExistence3} that 
		\begin{equation}
			(\Ex{B^\xi (0, e_1)}, \Ex{B^\xi (0, e_2)}) = \nabla \ell (\xi).
		\end{equation}
		We put everything together to conclude that
		\begin{equation}
			\lim_{n \to \infty}\frac{B^\xi(0, \floor{n \zeta})}{n} = \Ex{B^\xi(0, e_1)}(\zeta \cdot e_1) + \Ex{B^\xi(0, e_2)}(\zeta \cdot e_2) = \nabla \ell (\xi) \cdot \zeta.
		\end{equation}
	\end{proof}
\end{lemma}

\begin{proposition}
	Suppose the limit shape $\ell$ is differentiable and strictly convex at $\xi$. Then almost surely $\pi$ is a semi-infinite geodesic with direction $\xi$ (that is, $x_n / n \to \xi$).
	\begin{proof}
		We have already shown that it is a geodesic in \cref{lem:GeodesicMachine}. For the directedness, note that by compactness the sequence $\set{x_n / n}$ does have limit points. Suppose that we have a subsequence $x_{n_i}$ such that $x_{n_i} / n_i \to \zeta$. We would like to show $\zeta = \xi$. From \cref{lem:BusemannErgodic} we see that
		\begin{equation}
			\frac{B(x_0, x_{n_i})}{n_i} \to \nabla \ell(\xi) \cdot \zeta.
		\end{equation}
		But our argument in \cref{lem:GeodesicMachine} shows that $B(x_0, x_{n_i}) = L(x_0, x_{n_i})$. Thus also
		\begin{equation}
			\frac{B(x_0, x_{n_i})}{n_i} \to \ell(\zeta),
		\end{equation}
		thus showing that $\ell(\zeta) = \zeta \cdot \nabla \ell(\xi)$.
		
		At this point we observe that differentiating the scaling relation $\ell(t \xi) = t \ell(\xi)$ shows that $\ell(\xi) = \xi \cdot \nabla \ell(\xi)$. Subtracting this from the equation derived in the previous paragraph,
		\begin{equation}
			\ell(\zeta) - \ell(\xi) = (\zeta - \xi) \cdot \nabla \ell (\xi).
		\end{equation}
		Consider the convex function 
		\begin{equation}
			f(t) = \ell(t \xi + (1 - t)\zeta) - (t \xi + (1 - t)\zeta)\nabla \ell (\xi) - \ell(\xi) + \xi \cdot \nabla \ell (\xi).
		\end{equation}
		This $f$ satisfies $f(0) = f(1) = 0$ and $f'(1-) = 0$. Hence $f$ is identically zero on $[0, 1]$, and
		\begin{equation}
			\ell(t \xi + (1 - t)\zeta) = \ell(\xi) + (1 - t)(\zeta - \xi) \cdot \nabla \ell (\xi).
		\end{equation}
		Thus $\ell$ is affine along the line connecting $\zeta$ and $\xi$. We have assumed $\ell$ is strictly convex at $\xi$, and so the only possibility is $\zeta = \xi$.
	\end{proof}
\end{proposition}

\printbibliography[ heading=bibintoc, title={References}]
\end{document}